\documentclass[11pt,twoside]{amsart}

\usepackage[latin1]  {inputenc}%
\usepackage[T1]      {fontenc }%
\usepackage          {amsmath }%
\usepackage          {amsfonts}%
\usepackage          {amssymb }%
\usepackage          {amsthm  }%
\usepackage          {a4wide  }%
\usepackage          {url     }%
\usepackage          {tikz    }%
\usepackage[bookmarks=false,pdfborder={0 0 0.05}]{hyperref}
\usepackage[all]{xy}

\usetikzlibrary{arrows}

\usepackage{enumerate, amsmath, amsfonts, amssymb, amsthm,  wasysym, graphics, graphicx, xcolor, frcursive,comment,bbm}

\usepackage{etex}

\definecolor{darkblue}{rgb}{0.0,0,0.7} 
\definecolor{darkred}{rgb}{0.7,0,0} 

\usepackage{hyperref}
\usepackage[all]{xy}
\usepackage[T1]{fontenc}

\usepackage{array}

\usepackage{MnSymbol}

\usepackage{float}
\restylefloat{table}

\def\defn#1{{\sf #1}}

\newcommand{\edge}{\mathbin{\tikz [semithick, baseline=-0.2ex,-latex, ->] \draw [-] (0pt,0.4ex) -- (1em,0.4ex);}} 
\newcommand{\edgel}{\mathbin{\tikz [semithick, baseline=-0.2ex,-latex, ->] \draw [-] (0pt,0.4ex) -- (2.2em,0.4ex);}} 

\newcommand{\RR}{\mathbb R}
\newcommand{\ZZ}{\mathbb Z}

\DeclareMathOperator{\Red}{Red}
\DeclareMathOperator{\Span}{span}

\DeclareMathOperator{\idop}{id}

\DeclareMathOperator{\Aut}{Aut}

\DeclareMathOperator{\Sym}{Sym}

\DeclareMathOperator{\NN}{\mathbb{N}}

\DeclareMathOperator{\val}{val}
\DeclareMathOperator{\pr}{pr}
\DeclareMathOperator{\supp}{supp}

\newtheorem{theorem}{Theorem}[section]
\newtheorem{corollary}[theorem]{Corollary}
\newtheorem{Proposition}[theorem]{Proposition}
\newtheorem{Lemma}[theorem]{Lemma}

\theoremstyle{definition}
\newtheorem{Definition}[theorem]{Definition}
\newtheorem{remark}[theorem]{Remark}
\newtheorem{example}[theorem]{Example}

\newtheorem{question}[theorem]{Question}

\title[Reflection Groups and Quiver Mutation: Diagrammatics]{Reflection Groups and Quiver Mutation:\\ Diagrammatics}

\author[P.~Wegener]{Patrick Wegener}
\address{Patrick Wegener, Technische Universit\"at Kaiserslautern, Germany}
\email{wegener@mathematik.uni-kl.de}



\date{\today}

\begin{document}
\newcolumntype{C}[1]{>{\centering\arraybackslash}m{#1}}

\begin{abstract}
We extend Carter's notion of admissible diagrams and attach a ``Dynkin-like'' diagram to each reduced reflection factorization of an element in a finite Weyl group. We give a complete classification for the diagrams attached to reduced reflection factorizations. Remarkably, such a diagram turns out to be cyclically orientable if and only if it is isomorphic to the underlying graph of a quiver which is mutation-equivalent to a Dynkin quiver. Furthermore we show that each diagram encodes a natural presentation of the Weyl group as reflection group. The latter one extends work of Cameron--Seidel--Tsaranov as well as Barot--Marsh.
\end{abstract}

\maketitle

\tableofcontents

\section{Introduction}\label{sec:intro}

Any Weyl group is generated by reflections. In particular, each element of a Weyl group can be written as a product of reflections. We call this a \defn{reflection factorization}.
In 1972, Carter \cite{Car72} used a particular class of reflection factorizations of an element in a Weyl group to define so-called admissible diagrams. Using this diagrams, he was able to give a complete classification of the conjugacy classes in the Weyl groups. We extend this method and define a diagram attached to each reduced reflection factorization of every element in a Weyl group. The vertices of the diagram correspond to the reflections appearing in the reflection factorization. Two reflections $s$ and $t$ are connected by $w_{st}$ edges, where $w_{st}$ is the order of $st$ minus two. We call the resulting diagram a \defn{Carter diagram} (the precise definition will be given in Definition \ref{def:CarterDiagram}). 

In Section 2 we give a complete classification of Carter diagrams. A Carter diagram is said to be \defn{of type} $X_n$ if the reflection group generated by the reflections in the corresponding reduced reflection factorization is of type $X_n$. For the infinite families $A_n$, $B_n$ and $D_n$ we explicitely give a method to construct all these diagrams. Here we will benefit from work of Kluitmann \cite{Klu88}, who implicitely constructed the Carter diagrams of type $A_n$. For the exceptional types we obtain these diagrams by computational methods. We should mention that in the simply-laced types, that is types $A_n, D_n, E_6, E_7,E_8$, these diagrams were already described by Cameron--Seidel--Tsaranov \cite{CST94}, while the diagrams for the remaining types appear in the work of Felikson \cite{Fel04}.

In Section 3 we link these diagrams to those graphs arising as underlying unoriented graphs in the mutation classes of Dynkin quivers.

\begin{theorem} \label{thm:Main1}
Let $Q$ be a quiver which is mutation-equivalent to an orientation of a Dynkin diagram. Then the underlying undirected graph $\overline{Q}$ is a Carter diagram of the same Dynkin type.

Moreover, let $\Gamma$ be a Carter diagram of Dynkin type. Then there exists a quiver $Q$ which is mutation-equivalent to an orientation of the corresponding Dynkin diagram such that $\Gamma$ is isomorphic to $\overline{Q}$ if and only if $\Gamma$ is cyclically orientable.
\end{theorem}

We shortly recall the concept of quiver mutation and the definition of a cyclically orientable graph in Section \ref{sec:QuiverMutation}.

Given a crystallographic root system $\Phi$ and a Carter diagram $\Gamma$ of the same Dynkin type as $\Phi$, we extend work of Cameron--Seidel--Tsaranov \cite{CST94} as well as Barot--Marsh \cite{BM15} and show that $\Gamma$ encodes a natural presentation of the corresponding Weyl group $W=W_{\Phi}$. More precisely, for vertices $i$ and $j$ of $\Gamma$ we define $m_{ij}:=w_{ij} +2$, where $w_{ij}$ is the \defn{weight} of $i$ and $j$, that is, the number of edges connecting $i$ and $j$ in $\Gamma$. (Note that we possibly have $w_{ij}=0$, that is, there is no edge between $i$ and $j$.) Let $W(\Gamma)$ be the group with generators $t_i$, $i$ a vertex of $\Gamma$, subject to the following relations:
\begin{enumerate}
\item[(R1)] $t_i^2 =1$ for all vertices $i$ of $\Gamma$;
\item[(R2)] $(t_it_j)^{m_{ij}} =1$ for all vertices $i \neq j$ of $\Gamma$;
\item[(R3)] for any chordless cycle
$$
i_0 ~\stackrel{w_{i_0 i_1}}{\edgel} ~ i_1 ~ \stackrel{w_{i_1 i_2}}{\edgel}~ \ldots ~ \stackrel{w_{i_{d-2} i_{d-1}}}{\edgel} ~i_{d-1} ~ \stackrel{w_{i_{d-1} i_0}}{\edgel} ~ i_0,
$$
where either all weights are $1$ or $w_{i_{d-1} i_0} =2$, we have 
$$
(t_{i_0} t_{i_1} \cdots t_{i_{d-2}} t_{i_{d-1}} t_{i_{d-2}} \cdots t_{i_1})^2 =1.
$$
\end{enumerate}

\begin{theorem} \label{thm:Main2}
Let $\Phi$ be a crystallographic root system and let $\Gamma$ be a Carter diagram of the same Dynkin type as $\Phi$. Then $W(\Gamma)$ is isomorphic to the Weyl group $W_{\Phi}$. 
\end{theorem}

We prove this result in Section \ref{sec:Presentations}. There we also discuss to what extent this result can be generalized to the non-crystallographic cases.

The result presented in Theorem \ref{thm:Main2} extends the results of \cite{BM15} since we also consider diagrams which are not cyclically orientable. It also extends the results of \cite{CST94} since we also include the non simply-laced cases and we give an explicit construction of the relevant diagrams in the infinite families $A_n,~B_n$ and $D_n$.

\medskip

\noindent \textbf{Notation.} For $n \in \NN = \{ 1,2,\ldots \}$ we set $[n]:=\{1, \ldots , n \}$ and $[\pm n]:= \{ \pm k \mid k \in [n] \}$. If G is a group and $g,h \in G$, we put $g^h:= hgh^{-1}$ for conjugation. As usual a graph $\Gamma$ is a pair $(\Gamma_0, \Gamma_1)$, where $\Gamma_0$ is a finite set (the set of vertices) and $\Gamma_1 \subseteq \Gamma_0 \times \Gamma_0$ (set of edges). A subgraph $\Psi$ is called \defn{full subgraph} (or \defn{induced subgraph}) if $\Psi_0 \subseteq \Gamma_0$ and $\Psi_1 = \Gamma_1 \cap (\Psi_0 \times \Psi_0)$.

\medskip

\noindent \textbf{Acknowledgement.} The author would like to thank Robert Marsh and Sophiane Yahiatene for valuable comments on an earlier draft of this work.

\section{Reflection Groups and Carter Diagrams}
\subsection{Reflection Groups and Root systems} \label{sec:CoxeterI}
We collect some facts about reflection groups and root systems as can be found in \cite{Hu90}.

\medskip
Let $V$ be a (real) euclidean vector space with positive definite symmetric bilinear form $(-\mid -)$. A \defn{reflection} is a linear map $V \rightarrow V$ which sends some $\alpha \in V \setminus \{ 0 \}$ to its negative and fixes the hyperplane $\alpha^{\bot}$ orthogonal to $\alpha$. We denote such a reflection by $s_{\alpha}$. It is given by 
$$
s_{\alpha}: V \rightarrow V, ~v \mapsto v-\frac{2(\alpha \mid v)}{(\alpha \mid \alpha)} \alpha.
$$
A \defn{finite reflection group} is a finite subgroup of the orthogonal group $O(V)$ which is generated by reflections. Such a group acts (by reflections) on the ambient vector space $V$.

Let $W \leq O(V)$ be a finite reflection group. Each reflection $s_{\alpha} \in W$ determines a reflecting hyperplane $\alpha^{\bot}$ and a line $\RR \alpha$ orthogonal to $\alpha^{\bot}$. In the collection of these lines $\RR \alpha$ induced by all reflections $s_{\alpha} \in W$, it is possible to select a collection of vectors which is stable under the action of $W$. This leads to the following definition:

\begin{Definition}
A finite set $\Phi$ of nonzero vectors in $V$ is called \defn{root system} if 
\begin{itemize}
\item[(R1)] $\Span_{\RR}(\Phi)=V$;
\item[(R2)] $\Phi \cap \RR \alpha = \{ \pm \alpha \}$ for all $\alpha \in \Phi$;
\item[(R3)] $s_{\alpha}(\Phi) = \Phi$ for all $\alpha \in \Phi$.
\end{itemize}
The elements of $\Phi$ are called \defn{roots} and we define the \defn{rank} of $\Phi$ as the dimension of $V$. 
\end{Definition}

Each finite reflection group can be realized as the group $\langle s_{\alpha} \mid \alpha \in \Phi \rangle$ for some root system $\Phi$. We therefore write $W_{\Phi}$ for this group. Inside a root system $\Phi$ we always find a \defn{positive system} $\Phi^+$ which contains a unique \defn{simple system} $\Delta \subseteq \Phi^+$; see \cite[Ch.1.3]{Hu90}.

Let $\Phi \neq \varnothing$ be a root system. Then $\Phi$ is \defn{reducible} if $\Phi = \Phi_1 \cupdot \Phi_2$, where $\Phi_1$ and $\Phi_2$ are nonempty root systems such that $( \alpha \mid \beta)=0$ whenever $\alpha \in \Phi_1$ and $\beta \in \Phi_2$. Otherwise $\Phi$ is called \defn{irreducible}.

Given a simple system $\Delta \subseteq \Phi$, the group $W:=W_{\Phi}$ is generated by the set $S:=\{ s_{\alpha} \mid \alpha \in \Delta \}$ subject to the relations
\begin{align} \label{equ:CoxeterRelations}
(s_{\alpha}s_{\beta})^{m(\alpha,\beta)}=1 ~(\alpha, \beta \in \Delta),
\end{align}
where $m(\alpha, \beta)$ is the order of $s_{\alpha}s_{\beta}$ in $W$; see \cite[Ch.1.9]{Hu90}.

In general, a group having such a presentation is called \defn{Coxeter group}. The pair $(W,S)$ is called \defn{Coxeter system} and $S$ is called the set of \defn{simple reflections}.

Moreover, note that all reflections in $W_{\Phi}$ are of the form $s_{\alpha}$ for some $\alpha \in \Phi$. We therefore call the set 
$$
T:=\{ s_{\alpha} \mid \alpha \in \Phi \}
$$
the \defn{set of reflections}; see \cite[Ch. 1.14]{Hu90}.

\medskip
The presentation of $W=W_{\Phi}$ given by (\ref{equ:CoxeterRelations}) can be encoded in an undirected graph $\Gamma$ with vertex set corresponding to $\Delta$. Two vertices corresponding to $\alpha, \beta \in \Delta$ with $\alpha \neq \beta$ are joined by $m(\alpha, \beta)-2$ edges if $m(\alpha, \beta) \geq 3$. 
The graph $\Gamma$ is called \defn{Coxeter graph}. If this graph is connected, the corresponding Coxeter system $(W,S)$ is called \defn{irreducible}. By abuse of notation, we sometimes just say that $W$ is irreducible.

The irreducible Coxeter systems $(W,S)$ with $W$ a finite reflection group, are classified by their Coxeter diagrams; see \cite[Section 2.4]{Hu90}.

A reflection group $W_{\Phi}$ is called \defn{Weyl group} and the root system $\Phi$ is called \defn{crystallographic} if 
\begin{itemize}
\item[(R4)] $\frac{2(\alpha \mid \beta)}{(\beta \mid \beta)} \in \ZZ$ for all $\alpha, \beta \in \Phi$.
\end{itemize}
The irreducible Weyl groups are classified in terms of Dynkin diagrams, see Figure \ref{fig:Dynkin}. When $\Phi$ is an irreducible crystallographic root system, there are at most two root lengths. Therefore roots are called \defn{short} or \defn{long}, depending on their respective lengths.

\begin{figure}
\centering
\begin{tikzpicture}[scale=1.75, thick,>=latex]
    \node (0) at (0,-0.2) [] {$A_n$ ($n \ge 1$)};
    \node (A) at (1,-0.2) [circle, draw, fill=black!50, inner sep=0pt, minimum width=4pt] {};
    \node (B) at (1.5,-0.2) [circle, draw, fill=black!50, inner sep=0pt, minimum width=4pt]{};
    \node (C) at (2,-0.2) []{...};
    \node (D) at (2.5,-0.2) [circle, draw, fill=black!50, inner sep=0pt, minimum width=4pt]{};
    \node (E) at (3,-0.2) [circle, draw, fill=black!50, inner sep=0pt, minimum width=4pt]{};
    
        \node (02) at (0,-0.9) [] {$B_n$ ($n \ge 2$)};
           \node (A2) at (1,-0.9) [circle, draw, fill=black!50, inner sep=0pt, minimum width=4pt] {};
    \node (B2) at (1.5,-0.9) [circle, draw, fill=black!50, inner sep=0pt, minimum width=4pt]{};
    \node (C2) at (2,-0.9) []{...};
    \node (D2) at (2.5,-0.9) [circle, draw, fill=black!50, inner sep=0pt, minimum width=4pt]{};
    \node (E2) at (3,-0.9) [circle, draw, fill=black!50, inner sep=0pt, minimum width=4pt]{};
    
    
        \node (03) at (0,-1.75) [] {$D_n$ ($n \ge 4$)};
        \node (A3) at (1,-1.75) [circle, draw, fill=black!50, inner sep=0pt, minimum width=4pt] {};
    \node (B3) at (1.5,-1.75) [circle, draw, fill=black!50, inner sep=0pt, minimum width=4pt]{};
    \node (C3) at (2,-1.75) []{...};
    \node (D3) at (2.5,-1.75) [circle, draw, fill=black!50, inner sep=0pt, minimum width=4pt]{};
    \node (E3) at (3,-1.5) [circle, draw, fill=black!50, inner sep=0pt, minimum width=4pt]{};
        \node (F3) at (3,-2) [circle, draw, fill=black!50, inner sep=0pt, minimum width=4pt]{};

            \node (04) at (0,-2.5) [] {$E_6$};
    \node (A4) at (1,-2.5) [circle, draw, fill=black!50, inner sep=0pt, minimum width=4pt] {};
    \node (B4) at (1.5,-2.5) [circle, draw, fill=black!50, inner sep=0pt, minimum width=4pt]{};
    \node (C4) at (2,-2.5) [circle, draw, fill=black!50, inner sep=0pt, minimum width=4pt]{};
    \node (D4) at (2.5,-2.5) [circle, draw, fill=black!50, inner sep=0pt, minimum width=4pt]{};
    \node (E4) at (3,-2.5) [circle, draw, fill=black!50, inner sep=0pt, minimum width=4pt]{};
        \node (F4) at (2,-3) [circle, draw, fill=black!50, inner sep=0pt, minimum width=4pt]{};
        
     \node (05) at (0,-3.5) [] {$E_7$};    
    \node (A5) at (1,-3.5) [circle, draw, fill=black!50, inner sep=0pt, minimum width=4pt] {};
    \node (B5) at (1.5,-3.5) [circle, draw, fill=black!50, inner sep=0pt, minimum width=4pt]{};
    \node (C5) at (2,-3.5) [circle, draw, fill=black!50, inner sep=0pt, minimum width=4pt]{};
    \node (D5) at (2.5,-3.5) [circle, draw, fill=black!50, inner sep=0pt, minimum width=4pt]{};
    \node (E5) at (3,-3.5) [circle, draw, fill=black!50, inner sep=0pt, minimum width=4pt]{};
        \node (F5) at (3.5,-3.5) [circle, draw, fill=black!50, inner sep=0pt, minimum width=4pt]{};
                \node (G5) at (2,-4) [circle, draw, fill=black!50, inner sep=0pt, minimum width=4pt]{};

           \node (06) at (0,-4.5) [] {$E_8$};
        \node (A6) at (1,-4.5) [circle, draw, fill=black!50, inner sep=0pt, minimum width=4pt] {};
    \node (B6) at (1.5,-4.5) [circle, draw, fill=black!50, inner sep=0pt, minimum width=4pt]{};
    \node (C6) at (2,-4.5) [circle, draw, fill=black!50, inner sep=0pt, minimum width=4pt]{};
    \node (D6) at (2.5,-4.5) [circle, draw, fill=black!50, inner sep=0pt, minimum width=4pt]{};
    \node (E6) at (3,-4.5) [circle, draw, fill=black!50, inner sep=0pt, minimum width=4pt]{};
        \node (F6) at (3.5,-4.5) [circle, draw, fill=black!50, inner sep=0pt, minimum width=4pt]{};
                \node (G6) at (2,-5) [circle, draw, fill=black!50, inner sep=0pt, minimum width=4pt]{};
         \node (H6) at (4,-4.5) [circle, draw, fill=black!50, inner sep=0pt, minimum width=4pt]{};

           \node (07) at (0,-5.5) [] {$F_4$};
        \node (A7) at (1,-5.5) [circle, draw, fill=black!50, inner sep=0pt, minimum width=4pt] {};
    \node (B7) at (1.5,-5.5) [circle, draw, fill=black!50, inner sep=0pt, minimum width=4pt]{};
    \node (C7) at (2,-5.5) [circle, draw, fill=black!50, inner sep=0pt, minimum width=4pt]{};
    \node (D7) at (2.5,-5.5) [circle, draw, fill=black!50, inner sep=0pt, minimum width=4pt]{};
        
                   \node (08) at (0,-6) [] {$G_2$};
                \node (A8) at (1,-6) [circle, draw, fill=black!50, inner sep=0pt, minimum width=4pt] {};
    \node (B8) at (1.5,-6) [circle, draw, fill=black!50, inner sep=0pt, minimum width=4pt]{};
    \draw[-] (A) to (B);
    \draw[-] (B) to (C);
    \draw[-] (C) to (D);
    \draw[-] (D) to (E);

    \draw[-] (A2) to (B2);
    \draw[-] (B2) to (C2);
    \draw[-] (C2) to (D2);
    \draw ([xshift=0.5]D2.south) to ([xshift=0.5]E2.south);
    \draw[-] ([xshift=0.5]D2.north) to ([xshift=0.5]E2.north);

    
    \draw[-] (A3) to (B3);
    \draw[-] (B3) to (C3);
    \draw[-] (C3) to (D3);
    \draw[-] (D3) to (E3);
    \draw[-] (D3) to (F3);
    
    \draw[-] (A4) to (B4);
    \draw[-] (B4) to (C4);
    \draw[-] (C4) to (D4);
    \draw[-] (D4) to (E4);
    \draw[-] (C4) to (F4);    
    
    \draw[-] (A5) to (B5);
    \draw[-] (B5) to (C5);
    \draw[-] (C5) to (D5);
    \draw[-] (D5) to (E5);
    \draw[-] (E5) to (F5);
    \draw[-] (C5) to (G5);
    
        \draw[-] (A6) to (B6);
    \draw[-] (B6) to (C6);
    \draw[-] (C6) to (D6);
    \draw[-] (D6) to (E6);
    \draw[-] (E6) to (F6);
    \draw[-] (C6) to (G6);
    \draw[-] (F6) to (H6);

    \draw[-] (A7) to (B7);
    \draw ([xshift=0.5]B7.south) to ([xshift=0.5]C7.south);
    \draw[-] ([xshift=0.5]B7.north) to ([xshift=0.5]C7.north);
    \draw[-] (C7) to (D7);
    
    \draw[-] (A8) to (B8);
    \draw ([xshift=0.5]A8.south) to ([xshift=0.5]B8.south);
    \draw[-] ([xshift=0.5]A8.north) to ([xshift=0.5]B8.north);

\end{tikzpicture}
\caption{Dynkin diagrams.} \label{fig:Dynkin}
\end{figure}
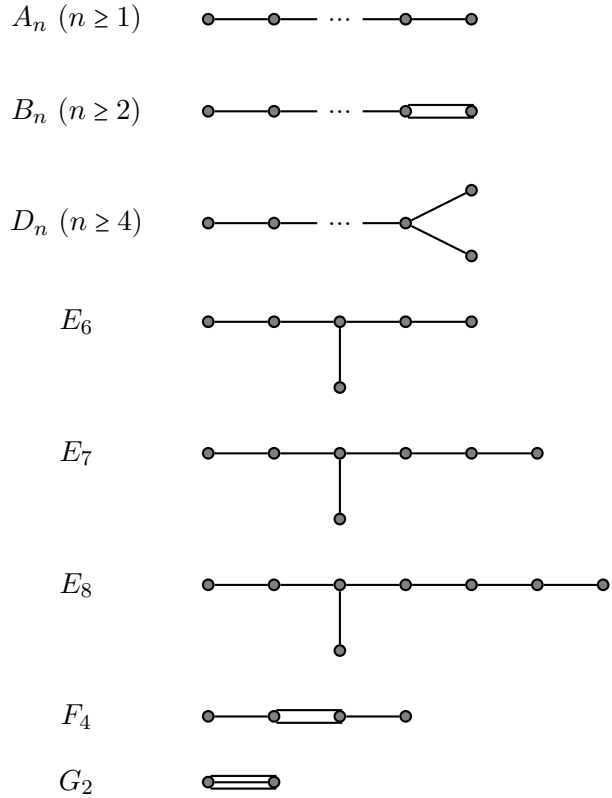

\subsection{Reflection Length and Reflection Factorizations} 
Let $W=W_{\Phi}$ be a finite reflection group with root system $\Phi$ and set of reflections $T$. Each $w \in W$ is a product of reflections in $T$. We define 
$$
\ell_T(w):= \min \{ k \in \ZZ_{\geq 0} \mid w=s_{\beta_1} \cdots s_{\beta_k}, ~\beta_i \in \Phi \}
$$
and call $\ell_T(w)$ the \defn{reflection length} of $w$. If $w= s_{\beta_1} \cdots s_{\beta_k}$ with $\beta_i \in \Phi$, we call $(s_{\beta_1}, \ldots , s_{\beta_k})$ a \defn{reflection factorization} for $w$. If in addition $k=\ell_T(w)$, we call $(s_{\beta_1}, \ldots , s_{\beta_k})$ to be \defn{reduced}. We denote by $\Red_T(w)$ the set of all reduced reflection factorizations for $w$. 

There is a geometric criterion for Weyl groups to decide whether a reflection factorization is reduced.
\begin{theorem}[{Carter's Lemma, \cite[Lemma 3]{Car72}}] \label{thm:CartersLemma}
Let $W_{\Phi}$ be a Weyl group and $\beta_1, \ldots , \beta_k \in \Phi$. Then $\ell_T(s_{\beta_1} \cdots s_{\beta_k})=k$ if and only if $\beta_1, \ldots , \beta_k \in \Phi$ are linearly independent. 
\end{theorem}

Since the set $T$ of reflections is closed under conjugation, there is a natural way to obtain new reflection factorizations from a given (not necessarily reduced) reflection factorization. The \defn{braid group} on $n$ strands, denoted $\mathcal{B}_n$, is the group with generators 
$\sigma_1,\ldots , \sigma_{n-1}$ subject to the relations
\begin{align*}
 \sigma_i \sigma_j & = \sigma_j \sigma_i \quad \text{for } |i-j| > 1,\\
\sigma_i \sigma_{i+1} \sigma_i & = \sigma_{i+1} \sigma_i \sigma_{i+1} \quad \text{for } i=1,\dots,n-2.
\end{align*}
It acts on the set $T^n$ of $n$-tuples of reflections as
\begin{align*}
\sigma_i (t_1 ,\ldots , t_n ) &= (t_1 ,\ldots , t_{i-1} , \hspace*{5pt} t_i t_{i+1} t_i,
\hspace*{5pt} \phantom{t_{i+1}}t_i\phantom{t_{i+1}}, \hspace*{5pt} t_{i+2} ,
\ldots , t_n), \\
\sigma_i^{-1} (t_1 ,\ldots , t_n ) &= (t_1 ,\ldots , t_{i-1} , \hspace*{5pt} \phantom
{t_i}t_{i+1}\phantom{t_i}, \hspace*{5pt} t_{i+1}t_it_{i+1}, \hspace*{5pt} t_{i+2} ,
\ldots , t_n).
\end{align*}
\noindent We call this action of $\mathcal{B}_n$ on $T^n$ the \defn{Hurwitz action} and an orbit of this action a \defn{Hurwitz orbit}. It is easy to see that this action restricts to the reduced reflection factorizations of a given element $w \in W$, that is, it restricts to the set $\Red_T(w)$ for each $w \in W$.

\subsection{Admissible diagrams and Carter diagrams} Let $W_{\Phi}$ be a Weyl group. It is a nontrivial result of Carter \cite[Theorem C]{Car72} that each $w \in W_{\Phi}$ can be written as $w= w_1w_2$ with $w_1, w_2 \in W_{\Phi}$ involutions. Furthermore we can write
$$
w_1= s_{\beta_1} \cdots s_{\beta_k}, ~w_2=s_{\beta_{k+1}} \cdots s_{\beta_{k+h}}
$$
such that $\{ \beta_1 \ldots ,\beta_k \}$ and $\{ \beta_{k+1} \ldots ,\beta_{k+h} \}$ are sets of mutually orthogonal roots and $\ell_T(w)= k+h$; see \cite[Lemma 5]{Car72}. Corresponding to such a factorization, Carter defines a graph $\Gamma$: 
\begin{itemize}
\item[(A)] The vertices of $\Gamma$ correspond to the roots $\beta_{1} \ldots ,\beta_{k+h}$.
\item[(B)] Two distinct roots $\beta_i, \beta_j$ are joined by 
$$
\frac{2(\beta_i \mid \beta_j)}{(\beta_i \mid \beta_i)} \cdot \frac{2(\beta_j \mid \beta_i)}{(\beta_j \mid \beta_j)}.
$$
edges.
\end{itemize}
Furthermore, the graph $\Gamma$ is called an \defn{admissible diagram} (for $w$) if each subgraph of $\Gamma$ which is a cycle contains an even number of vertices. Admissible diagrams were classified by Carter; see \cite[Section 5]{Car72}.

We extend Carter's definition as follows to all sets of linearly independent roots resp. to all reduced reflection factorizations.

\begin{Definition} \label{def:CarterDiagram}
Let $W_{\Phi}$ be a Weyl group with crystallographic root system $\Phi$. To each set of linearly independent roots $\{\beta_1, \ldots , \beta_m \} \subseteq \Phi$ we associate a diagram $\Gamma$ given by the conditions (A) and (B) above. We call $\Gamma$ a \defn{Carter diagram}. The \defn{type} of $\Gamma$ is defined to be the (Dynkin-)type of the smallest root subsystem $\Phi' \subseteq \Phi$ which contains $\{\beta_1, \ldots , \beta_m \}$.
\end{Definition}

\begin{remark}
$\phantom{4}$
\begin{enumerate}
\item[(a)] For a set of roots $R \subseteq \Phi$ the smallest root subsystem $\Phi' \subseteq \Phi$ containing $R$ is given by $W_R(R)$, where $W_R:= \langle s_{\alpha} \mid \alpha \in R \rangle$. If  $\Phi'$ is irreducible, then the Carter diagram associated to $R$ is connected. 
\item[(b)] A Carter diagram does not change if we replace a root by its corresponding negative root.
\end{enumerate}
\end{remark}

\begin{example}
Each Dynkin diagram is a Carter diagram. In fact, a Carter diagram without cycles is a Dynkin diagram. This is proved in \cite[Lemma 8]{Car72} for admissible diagrams, but the proof given there also works for arbitrary Carter diagrams.
\end{example}

\begin{example}
We consider a root system $\Phi$ of type $D_4$. This can be realized as 
$$
\Phi = \{ \pm e_i \pm e_j \mid 1 \leq i < j \leq 4 \},
$$
where $\{e_1,e_2,e_3,e_4\}$ is the canonical base of $\RR^4$. The Carter diagram $\Gamma_1$ attached to the set of roots 
$$
R_1 := \{ e_1-e_2, e_1+e_2, e_2-e_3, e_4-e_1 \} \subseteq \Phi
$$
is given by the following diagram.
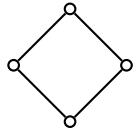
\begin{figure}[H]
\centering
\begin{tikzpicture}[scale=0.75, thick,>=latex]

  \node (21) at (3,0) [circle, draw, inner sep=0pt, minimum width=4pt]{};
  \node (31) at (3,-2) [circle, draw, inner sep=0pt, minimum width=4pt]{};
  \node (41) at (2,-1) [circle, draw, inner sep=0pt, minimum width=4pt]{};
  \node (51) at (4,-1) [circle, draw, inner sep=0pt, minimum width=4pt]{};

  \draw[-] (21) to (41);
  \draw[-] (31) to (41);
  
  \draw[-] (21) to (51);
  \draw[-] (31) to (51);

\end{tikzpicture}
\end{figure}
The smallest root subsystem of $\Phi$ containing $R_1$ is the root system $\Phi$ itself. Therefore $\Gamma_1$ is a Carter diagram of type $D_4$.

On the other hand, the Carter diagram $\Gamma_2$ attached to the set of roots 
$$
R_2 := \{ e_1-e_3, e_2-e_3, e_3-e_4 \} \subseteq \Phi
$$
is given by the following diagram.
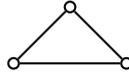
\begin{figure}[H]
\centering
\begin{tikzpicture}[scale=0.75, thick,>=latex]

  \node (21) at (3,0) [circle, draw, inner sep=0pt, minimum width=4pt]{};
  \node (41) at (2,-1) [circle, draw, inner sep=0pt, minimum width=4pt]{};
  \node (51) at (4,-1) [circle, draw, inner sep=0pt, minimum width=4pt]{};

  \draw[-] (21) to (41);
  
  \draw[-] (21) to (51);
  \draw[-] (41) to (51);

\end{tikzpicture}
\end{figure}
The smallest root subsystem of $\Phi$ containing $R_2$ is
$$
\{e_i-e_j \mid i \neq j, ~1 \leq i,j \leq 4\} \subseteq \Phi.
$$
This system is of type $A_3$ and so is $\Gamma_2$. 
\end{example}

\begin{remark} \label{rem:MinCarter}
In the definition of a Carter diagram we demand the set of roots $R:=\{\beta_1, \ldots , \beta_m \}$ to be linearly independent. By Carter's Lemma \ref{thm:CartersLemma} this is equivalent to $s_{\beta_1} \cdots s_{\beta_m}$ being reduced. In particular, to each $w \in W_{\Phi}$ and each reduced reflection factorization $(s_{\beta_1}, \ldots , s_{\beta_m})$ of $w$, we have an associated Carter diagram $\Gamma$ induced by the set $\{\beta_1, \ldots , \beta_m \}$. 
Therefore we can describe Carter diagrams entirely by reduced reflection factorizations. If $(s_{\beta_1}, \ldots , s_{\beta_m}) \in \Red_T(w)$, then the vertices of the diagram $\Gamma$ correspond to the reflections $s_{\beta_i}$ ($1 \leq i \leq m$). By Carter's Lemma \ref{thm:CartersLemma} there is an edge between $s_{\beta_i}$ and $s_{\beta_j}$ ($i \neq j$) if and only if $s_{\beta_i}$ and $s_{\beta_j}$ do not commute. The number of edges between $s_{\beta_i}$ and $s_{\beta_j}$ is given by $m_{ij}-2$, where $m_{ij}$ is the order of $s_{\beta_i} s_{\beta_j}$. In this case we also call $\Gamma$ a \defn{Carter diagram associated to $w$}, or more precisely a \defn{Carter diagram associated to $(s_{\beta_1}, \ldots , s_{\beta_m})$}. The linear independence of $R$ also implies that the root subsystem $\Phi':=W_R(R)$ is of rank $m$ and by \cite[Theorem 1.1]{BW18} we have $W_{\Phi'} = \langle s_{\beta_1}, \ldots , s_{\beta_m} \rangle$.
\end{remark}

\medskip
Carter diagrams are invariant under conjugation in the following sense.

\begin{Lemma} \label{lem:ConjCarterDiag}
Let $W_{\Phi}$ be a Weyl group with crystallographic root system $\Phi$, $w \in W_{\Phi}$ and $(s_{\beta_1}, \ldots , s_{\beta_m}) \in \Red_T(w)$. Then $(s_{\beta_1}, \ldots , s_{\beta_m})$ (resp. $\{\beta_1, \ldots , \beta_m \}$) and $(s_{\beta_1}^x, \ldots , s_{\beta_m}^x) \in \Red_T(w^x)$ (resp. $\{x(\beta_1), \ldots , x(\beta_m) \}$) yield the same Carter diagram for all $x \in W_{\Phi}$. 
\end{Lemma}

\begin{proof}
We have $s_{\beta_i}^x= s_{x(\beta_i)}$ and $(\beta_i \mid \beta_j) = (x(\beta_i)\mid x(\beta_j))$, hence
$$
\frac{2(\beta_i \mid \beta_j)}{(\beta_i \mid \beta_i)} \cdot \frac{2(\beta_j \mid \beta_i)}{(\beta_j \mid \beta_j)} = 
\frac{2(x(\beta_i) \mid x(\beta_j))}{(x(\beta_i) \mid x(\beta_i))} \cdot \frac{2(x(\beta_j) \mid x(\beta_i))}{(x(\beta_j) \mid x(\beta_j))}
$$
\end{proof}

\begin{Definition}
Let $W=W_{\Phi}$ be a finite reflection group with root system $\Phi$ of rank $n$. An element $w \in W$ is called \defn{quasi-Coxeter element} if there exists a reduced reflection factorization $(s_{\beta_1}, \ldots, s_{\beta_n}) \in \Red_T(w)$ such that $\langle s_{\beta_1}, \ldots, s_{\beta_n} \rangle = W$. 

A very important example of a quasi-Coxeter element is the following: An element $c \in W$ is called \defn{Coxeter element} if there exists a simple system $\{ \alpha_1 , \ldots , \alpha_n \} \subseteq \Phi$ such that $c = s_{\alpha_1} \cdots s_{\alpha_n}$.  
\end{Definition}

\begin{remark}
$\phantom{4}$
\begin{enumerate}
\item[(a)] Given a crystallographic root system $\Phi$, the Carter diagrams of the same type as $\Phi$ are precisely the Carter diagrams attached to reduced reflection factorizations of quasi-Coxeter elements in $W_{\Phi}$ (see also Remark \ref{rem:MinCarter}).
\item[(b)] Baumeister and the author give in \cite{BW18} another characterization of quasi-Coxeter elements in terms of bases of the (co-)root lattice.
\end{enumerate}
\end{remark}

\begin{example}
Let $\Phi$ be a root system of type $A_n$ and $\{\alpha_1 , \ldots, \alpha_n\} \subseteq \Phi$ a simple system. Then $c= s_{\alpha_1} \cdots s_{\alpha_n}$ is a Coxeter element. The Carter diagram associated to $\{\alpha_1 , \ldots, \alpha_n\}$ resp. to $(s_{\alpha_1}, \ldots , s_{\alpha_n}) \in \Red_T(c)$ is the Dynkin diagram of type $A_n$. Applying the Hurwitz action yields
\begin{align*}
(\sigma_1 \cdots \sigma_{n-1})(\sigma_2 \cdots \sigma_{n-1}) \cdots \sigma_{n-1} (s_{\alpha_1}, \ldots , s_{\alpha_n}) & =  (s_{\alpha_n}^{s_{\alpha_1} \cdots s_{\alpha_{n-1}}}, \ldots ,  s_{\alpha_2}^{s_{\alpha_1}}, s_{\alpha_1})\\
{} & = (s_{\alpha_1 + \cdots + \alpha_n}, \ldots , s_{\alpha_1 + \alpha_2}, s_{\alpha_1}) \in \Red_T(c).
\end{align*}
The Carter diagram associated to the set $\{ \alpha_1 + \cdots + \alpha_n, \ldots , \alpha_1 + \alpha_2, \alpha_1 \}$ resp. to the reduced reflection factorization $(s_{\alpha_n}^{s_{\alpha_1} \cdots s_{\alpha_{n-1}}}, \ldots ,s_{\alpha_2}^{s_{\alpha_1}} ,s_{\alpha_1} ) \in \Red_T(c)$ is the complete graph on $n$ vertices.
\end{example}

The following lemma allows us to investigate Carter diagrams via quasi-Coxeter elements in irreducible Weyl groups. 
\begin{Lemma} \label{lem:CarterIrredQuasiCox}
Each Carter diagram is the disjoint union of connected Carter diagrams associated to reduced reflection factorizations of quasi-Coxeter elements in irreducible Weyl groups. 
\end{Lemma}

\begin{proof}
Let $W=W_{\Phi}$ be a Weyl group with crystallographic root system $\Phi$. Consider a set of linearly independet roots $R:=\{ \beta_1 , \ldots, \beta_m \} \subseteq \Phi$ and put $\Phi':=W_R(R)$. By \cite[Theorem 1.1]{BW18} $s_{\beta_1} \cdots s_{\beta_m}$ is a quasi-Coxeter element in $W_{\Phi'}$. We can decompose $\Phi'$ as $\Phi'= \Phi_1 \cupdot \cdots \cupdot \Phi_k$ such that $W_{\Phi_i}$ is irreducible and such that the $\Phi_i$ are pairwise orthogonal. Again by \cite[Theorem 1.1]{BW18}, we have $\Phi_i = W_{R_i}(R_i)$ with $R_i:=\Phi_i \cap R$. Therefore each connected component of the Carter diagram attached to $\{ \beta_1 , \ldots, \beta_m \}$ is given by a Carter diagram associated to the reduced reflection factorization of a quasi-Coxeter element in $W_{\Phi_i}$ for some $i \in [k]$. 
\end{proof}

\subsection{Construction of Carter diagrams}
The aim of this section is to give a procedure of constructing all possible Carter diagrams of types $A_n$, $B_n$ and $D_n$. For the exceptional types we will provide a complete list based on computations.

\subsubsection{{\bf Carter diagrams of type $A_n$}} \label{sec:CarterDiagA}
Let $n \geq 1$ be an integer. For the type $A_n$ we will use a result of Kluitmann \cite{Klu88}. It is well known that the Weyl group $W$ of type $A_n$ can be identified with the symmetric group $\Sym([n+1])$. In this setting the set of reflections can be identified with the set of transpositions. As in \cite{Klu88} we define for $m \geq n$ and $w \in W$ the following (possibly empty) sets
\begin{align*}
\Xi^{n,m} & := \{ (t_1, \ldots , t_m) \mid t_i \in \Sym([n+1])~\text{a transposition}, ~ \langle t_1, \ldots , t_m \rangle = \Sym([n+1]) \}\\
\Xi^{n,m}_w & := \{ (t_1, \ldots , t_m) \in \Xi^{n,m} \mid w = t_1 \cdots t_m  \} .
\end{align*}

\begin{remark}
In type $A_n$ each quasi-Coxeter element is a Coxeter element and Coxeter elements correspond to $(n+1)$-cycles in $\Sym([n+1])$ \cite[Remark 6.6]{BGRW}. If $c$ is a $(n+1)$-cycle in $\Sym([n+1])$, then $\Xi^{n,n}_c= \Red_T(c)$, where $T$ is the set of transpositions in $\Sym([n+1])$. Moreover, we have 
$$
\Xi^{n,n} = \bigcup_{(n+1)-\text{cycle }c} \Red_T(c).
$$
\end{remark}

\medskip
For each element $(t_1, \ldots , t_m) \in \Xi^{n,m}$, Kluitmann defines a graph as follows:
\begin{itemize}
\item The vertices correspond to the set $\{ t_1, \ldots , t_m \}$.
\item Two vertices corresponding to $t_i$ and $t_j$ are connected by an edge if $t_it_j \neq t_j t_i$. 
\end{itemize}
We call such a diagram a \defn{Kluitmann diagram}. Denote by $\mathcal{A}^{n,m}$ (resp. by $\mathcal{A}^{n,m}_w$ for $w \in \Sym([n+1])$) the set of Kluimann diagrams given by the elements of $\Xi^{n,m}$ (resp. the elements of $\Xi^{n,m}_w$). As a direct consequence of our definitions of Carter diagrams and Kluitmann diagrams, we obtain:
\begin{Proposition}
The set of Carter diagrams of type $A_n$ is given by the set $\mathcal{A}^{n,n}$.
\end{Proposition}

\medskip
Kluitmann provides a procedure to construct all diagrams in $\mathcal{A}^{n,m}$.

\begin{theorem}[{Kluitmann, \cite[Theorem 1]{Klu88}}] \label{thm:Kluitmann}
The following construction yields all elements in $\mathcal{A}^{n,m}$: 
\begin{itemize}
\item[(a)] Choose $m' \in \NN$ with $n \leq m' \leq m$. Take any connected graph $\Gamma$ on $m'$ vertices with $\Gamma$ the union of subgraphs $\Gamma_1, \ldots , \Gamma_k$ such that:
\begin{itemize}
\item[(i)] $\Gamma_1, \ldots , \Gamma_k$ are complete graphs.
\item[(ii)] $\Gamma_i$ and $\Gamma_j$ ($i \neq j$) intersect in at most one vertex.
\item[(iii)] Every vertex of $\Gamma$ belongs to at most two of the subgraphs $\Gamma_i$.
\item[(iv)] $\sum_{i=1}^k(|\Gamma_i|-1) = (m'-1)+(m'-n)$.
\end{itemize}
\item[(b)] Attach additional vertices $v_1 , v_2, \ldots , v_{m-m'}$ to $\Gamma$ by the following procedure:

Suppose that $\Psi:= \Gamma \cup \{v_1 , \ldots, v_{\ell}\}$, with $v_1 , \ldots, v_{\ell}$ already attached to $\Gamma$. Choose a vertex $v$ of $\Psi$, and define $v_{k+1}$ to be its ``duplication'', that is $v$ and $v_{k+1}$ are connected with exactly the same vertices; there is no edge between $v$ and $v_{k+1}$. 
\end{itemize}

\end{theorem}

\begin{corollary}[{Carter diagrams of type $A_n$}] \label{cor:Kluitmann}
The Carter diagrams of type $A_n$ are given by the set $\mathcal{A}^{n.n}$. More precisely, a Carter diagram of type $A_n$ is a connected graph $\Gamma$ on $n$ vertices which is the union of subgraphs $\Gamma_1, \ldots , \Gamma_k$ such that:
\begin{itemize}
\item[(i)] $\Gamma_1, \ldots , \Gamma_k$ are complete graphs.
\item[(ii)] $\Gamma_i$ and $\Gamma_j$ ($i \neq j$) intersect in at most one vertex.
\item[(iii)] Every vertex of $\Gamma$ belongs to at most two of the subgraphs $\Gamma_i$.
\item[(iv)] $\sum_{i=1}^k(|\Gamma_i|-1) = n-1$.
\end{itemize}
\end{corollary}

\begin{remark}
Note that a complete graph cannot be written as the union of proper subgraphs such that the properties (i)-(iv) of Corollary \ref{cor:Kluitmann} hold. As a consequence, we obtain that the descomposition $\Gamma = \Gamma_1 \cup \ldots \cup \Gamma_k$ given by Corollary \ref{cor:Kluitmann} is unique up to permutation of the factors. 
\end{remark}

\begin{example} \label{ex:TypeADiag}
By Corollary \ref{cor:Kluitmann}, the graph
\begin{figure}[H]
\centering
\begin{tikzpicture}[scale=0.75, thick,>=latex]

  \node (1) at (2,0) [circle, draw, inner sep=0pt, minimum width=4pt]{};
  \node (3) at (3,-1) [circle, draw, inner sep=0pt, minimum width=4pt]{};
  \node (3b) at (3,-0.65) []{$v$};
  \node (2) at (2,-2) [circle, draw, inner sep=0pt, minimum width=4pt]{};
  \node (4) at (4,0) [circle, draw, inner sep=0pt, minimum width=4pt]{};
  \node (5) at (4,-2) [circle, draw, inner sep=0pt, minimum width=4pt]{};
  \node (6) at (5,-1) [circle, draw, inner sep=0pt, minimum width=4pt]{};
  \node (6b) at (5,-0.65) []{$w$};
  \node (7) at (6,0) [circle, draw, inner sep=0pt, minimum width=4pt]{};
  \node (8) at (6,-2) [circle, draw, inner sep=0pt, minimum width=4pt]{};

  \draw[-] (1) to (2);
  \draw[-] (1) to (3);
  \draw[-] (2) to (3);
  \draw[-] (3) to (4);
  \draw[-] (3) to (5);
  \draw[-] (3) to (6);
  \draw[-] (4) to (5);
  \draw[-] (4) to (6);
  \draw[-] (5) to (6);
  \draw[-] (6) to (7);
  \draw[-] (6) to (8);
  \draw[-] (7) to (8);

\end{tikzpicture}
\end{figure}
is a Carter diagram of type $A_8$. To see this, let $\Gamma_1$ and $\Gamma_2$ be the complete graphs on three vertices intersecting the complete graph $\Gamma_3$ on four vertices in the center of the picture in the vertices $v$ and $w$, respectively.
\end{example}

\medskip
\subsubsection{{\bf Carter diagrams of type $B_n$}} \label{sec:CarterDiagB}
Let $W_{\Phi}$ be a Weyl group with $\Phi$ of type $B_n$. It is well known (see for instance \cite[Chapter 8.1]{BB05}) that $W_{\Phi}$ can be realized by signed permutations, that is, as the group 
$$
S_n^B:= \{ \pi \in \Sym([\pm n]) \mid \pi(-i)= -\pi(i) \},
$$
where the group operation is given by composition. The set of reflections can be identified with the set 
\begin{align} \label{equ:RefB}
T_n^B:= \{ (i,j)(-i,-j) \mid 1 \leq i < |j| \leq n \} \cup \{ (i,-i) \mid i \in [n] \}.
\end{align}
The group homomorphism 
\begin{align*}
\theta: \Sym([n]) & \rightarrow \Aut(\ZZ_2^n)\\
\pi & \mapsto [\ZZ_2^n \rightarrow \ZZ_2^n, (d_1, \ldots , d_n) \mapsto (d_{\pi(1)}, \ldots , d_{\pi(n)})]
\end{align*}
yields an isomorphism 
$$
S_n^B \cong \ZZ_2^n \rtimes_{\theta} \Sym([n]).
$$
For our purposes and for later use we state this isomorphism explicitely on the generating set $T_n^B$ of $S_n^B$ given in (\ref{equ:RefB}). Let $i,j \in \NN$ with $i <j$:
\begin{align} \label{equ:IsoB}
\varphi_B: S_n^B & \rightarrow  \ZZ_2^n \rtimes_{\theta} \Sym([n])\\
(i,j)(-i,-j) & \mapsto (0, (i,j)) \notag \\
(i,-j)(-i,j) & \mapsto (e_i + e_j, (i,j)) \notag \\
(i,-i) & \mapsto (e_i, \idop) \notag
\end{align}

We will now describe how to obtain all Carter diagrams of type $B_n$ by describing all Carter diagrams of reduced reflection factorizations of quasi-Coxeter elements in $S_n^B$ (see also Remark \ref{rem:MinCarter}). Note that in type $B_n$ each quasi-Coxeter element is a Coxeter element \cite[Lemma 6.4]{BGRW}. Therefore let $(t_1, \ldots, t_n)$ be a reduced reflection factorization of a Coxeter element. By Lemma \ref{lem:ConjCarterDiag} and by the proof of \cite[Lemma 6.4]{BGRW} we can assume that 
$$
\{t_1, \ldots , t_n \} = \{ (1,-1), (2, i_2)(-2,-i_2), \dots , (n,i_n)(-n,-i_n) \} =:R,
$$
where $i_j \in \{ 1, \ldots , j-1 \}$ for each $j \in \{ 2, \ldots, n \}$. Put
$$
R' := \{  (2, i_2)(-2,-i_2), \dots , (n,i_n)(-n,-i_n) \} \subseteq R.
$$

\begin{Proposition} \label{prop:CarterDB1}
The Carter diagram on vertex set $R'$ is (isomorphic to) a Carter diagram of type $A_{n-1}$.
\end{Proposition}

\begin{proof}
Two elements $(j,i_j)(-j, -i_j), (k, i_k)(-k,-i_k) \in R'$ commute in $S_n^B$ if and only if $(j,i_j)$ and $(k,i_k)$ commute in $\Sym([n])$. The set $\{  (2, i_2), \dots , (n,i_n) \}$ generates $\Sym([n])$ and therefore yields a Carter diagram of type $A_{n-1}$. 
\end{proof}

Denote by $\Gamma'$ the Carter diagram induced by $R'$ and let $\{ k_1, \ldots, k_{\ell} \}$ be the subset of $\{2, \ldots, n\}$ such that $i_{k_j}=1$ for all $j \in [\ell]$. By property (i) in Corollary \ref{cor:Kluitmann}, the graph $\Gamma'$ decomposes as the union of complete graphs.   

\begin{Proposition} \label{prop:CarterDB2}
The full subgraph $\Gamma''$ of $\Gamma'$ on vertex set 
$$
R'' := \{ (k_1,1)(-k_1,-1), \ldots ,(k_{\ell},1)(-k_{\ell},-1) \} \subseteq R'
$$
is one of the complete graphs in the decomposition of $\Gamma'$.
\end{Proposition}

\begin{proof}
It is easy to see that $\Gamma''$ is a complete graph. Therefore the assertion is true if $R' = R''$. Thus assume that $R'' \subsetneq R'$. To show that $\Gamma''$ respects the properties (i)-(iv) of Corollary \ref{cor:Kluitmann} it is enough to show that $\Gamma''$ is connected by exactly one edge with the rest of the graph $\Gamma'$. Since $\Gamma'$ is connected there exists at least one of those edges. Now assume that there exists $(i,i_j)(-i,-i_j) \in R' \setminus R''$ and $p,q \in [\ell], ~p \neq q$ such that $(i,i_j)(-i,-i_j)$ is connected to both $(k_p,1)(-k_p,-1) \in R''$ and $(k_q,1)(-k_q,-1) \in R''$ by an edge. This implies that $\{i, i_j\} = \{k_p,k_q\}$ and therefore $(k_p,1)(-k_p,-1)^{(k_q,1)(-k_q,-1)} = (i,i_j)(-i,-i_j)$, contradicting the fact that $\langle R \rangle = S_n^B$. 
\end{proof}

\medskip
Let $\Gamma$ be the Carter diagram on vertex set $R$. By Proposition \ref{prop:CarterDB1} the induced subgraph $\Gamma'$ on vertex set $R'$ is a Carter diagram of type $A_{n-1}$. This graph is described by the construction in Section \ref{sec:CarterDiagA}. By Proposition \ref{prop:CarterDB2} the induced subgraph $\Gamma''$ of $\Gamma'$ on vertex set $R''$ is a complete graph in the decomposition of $\Gamma'$. The element $(1,-1) \in R$ commutes with all elements in $R' \setminus R''$ and does not commute with any element in $R''$. The full subgraph of $\Gamma$ on vertex set $R'' \cup \{ (1,-1) \}$ is a complete graph on $\ell + 1$ vertices. Denote by $\Gamma^{(1)}$ the graph obtained from $\Gamma$ by suppressing the weights of the edges (that is, all weights are equal to $1$). Our previous arguments show that $\Gamma^{(1)}$ is a Carter diagram of type $A_n$. The element $(1,-1) \in R$ corresponds to a reflection $s_{\alpha}$ in $W_{\Phi}$ with $\alpha \in \Phi$ a short root. All elements in $R \setminus \{ (1,-1) \}$ correspond to reflections $s_{\beta}$ with $\beta \in \Phi$ a long root. Therefore all edges in $\Gamma$ adjacent with the vertex corresponding to $(1,-1)$ have the weight $2$. 

We summarize our construction:

\begin{theorem}[{Carter diagrams of type $B_n$}] \label{thm:CarterTypeB}
All Carter diagrams of type $B_n$ are obtained by the following construction:
\begin{itemize}
\item Take any possible Carter diagram $\Gamma$ of type $A_n$.
\item Take any vertex $v$ of $\Gamma$ such that $\Gamma \setminus \{ v \}$ is still connected.
\item All edges adjacent with $v$ have weight $2$.
\end{itemize}
We call $v$ the \defn{distinguished vertex} of $\Gamma$.
\end{theorem}

\begin{example} \label{ex:FromAtoB}
Consider the type $A_n$ Carter diagram from Example \ref{ex:TypeADiag}.
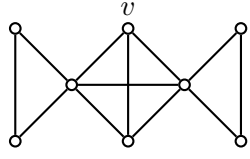
\begin{figure}[H]
\centering
\begin{tikzpicture}[scale=0.75, thick,>=latex]

  \node (1) at (2,0) [circle, draw, inner sep=0pt, minimum width=4pt]{};
  \node (3) at (3,-1) [circle, draw, inner sep=0pt, minimum width=4pt]{};
  \node (2) at (2,-2) [circle, draw, inner sep=0pt, minimum width=4pt]{};
  \node (4) at (4,0) [circle, draw, inner sep=0pt, minimum width=4pt]{};
  \node (4b) at (4,0.35) []{$v$};
  \node (5) at (4,-2) [circle, draw, inner sep=0pt, minimum width=4pt]{};
  \node (6) at (5,-1) [circle, draw, inner sep=0pt, minimum width=4pt]{};
  \node (7) at (6,0) [circle, draw, inner sep=0pt, minimum width=4pt]{};
  \node (8) at (6,-2) [circle, draw, inner sep=0pt, minimum width=4pt]{};

  \draw[-] (1) to (2);
  \draw[-] (1) to (3);
  \draw[-] (2) to (3);
  \draw[-] (3) to (4);
  \draw[-] (3) to (5);
  \draw[-] (3) to (6);
  \draw[-] (4) to (5);
  \draw[-] (4) to (6);
  \draw[-] (5) to (6);
  \draw[-] (6) to (7);
  \draw[-] (6) to (8);
  \draw[-] (7) to (8);

\end{tikzpicture}
\end{figure}
If we remove the vertex $v$, then the diagram is still connected. We choose $v$ to be the distinguished vertex. The resulting type $B_n$ Carter diagram is the following one. 
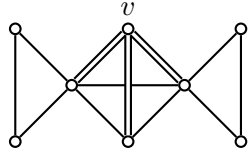
\begin{figure}[H]
\centering
\begin{tikzpicture}[scale=0.75, thick,>=latex, double equal sign distance]
   \tikzset{triple/.style={double,postaction={draw,-}}}

  \node (1) at (2,0) [circle, draw, inner sep=0pt, minimum width=4pt]{};
  \node (3) at (3,-1) [circle, draw, inner sep=0pt, minimum width=4pt]{};
  \node (2) at (2,-2) [circle, draw, inner sep=0pt, minimum width=4pt]{};
  \node (4) at (4,0) [circle, draw, inner sep=0pt, minimum width=4pt]{};
  \node (4b) at (4,0.35) []{$v$};
  \node (5) at (4,-2) [circle, draw, inner sep=0pt, minimum width=4pt]{};
  \node (6) at (5,-1) [circle, draw, inner sep=0pt, minimum width=4pt]{};
  \node (7) at (6,0) [circle, draw, inner sep=0pt, minimum width=4pt]{};
  \node (8) at (6,-2) [circle, draw, inner sep=0pt, minimum width=4pt]{};

  \draw[-] (1) to (2);
  \draw[-] (1) to (3);
  \draw[-] (2) to (3);
  \draw[double] (3) to (4);
  \draw[-] (3) to (5);
  \draw[-] (3) to (6);
  \draw[double] (4) to (5);
  \draw[double] (4) to (6);
  \draw[-] (5) to (6);
  \draw[-] (6) to (7);
  \draw[-] (6) to (8);
  \draw[-] (7) to (8);

\end{tikzpicture}
\end{figure}
\end{example}

\medskip
\subsubsection{{\bf Carter diagrams of type $D_n$}} \label{sec:CarterDiagD}
Let $W_{\Phi}$ be a Weyl group of type $D_n$. It is well known (see for instance \cite[Chapter 8.2]{BB05}) that $W_{\Phi}$ can be realized as the group of even signed permutations. We denote this group by $S_n^D$. For our purposes it will be convenient to see $S_n^D$ as the subgroup of $S_n^B$ generated by the set 
\begin{align}
T_n^D := \{ (i,j)(-i,-j) \mid 1 \leq i < |j| \leq n \}.
\end{align}
The set $T_n^D$ can also be identified with the set of reflections for $S_n^D$. We denote by $\varphi_D$ the restriction of the map $\varphi_B$ given in (\ref{equ:IsoB}) to the subgroup $S_n^D \leq S_n^B$. Further we put $\varphi:= \pr_2 \circ \varphi_D$, where
$$
\pr_2: \ZZ_2^n \rtimes_{\theta} \Sym([n]) \rightarrow \Sym([n]), (d,\pi) \mapsto \pi,
$$
and 
$$
\Xi_{D_n} := \{ (t_1, \ldots , t_n) \mid t_i \in T_n^D, ~\langle t_1, \ldots , t_n \rangle =S_n^D \}.
$$
Note that for $i,j \in \NN$ with $i < j$ we have $\varphi((i,j)(-i,-j))= \varphi((i,-j)(-i,j))=(i,j)$. Therefore and since $\varphi$ is surjective, $\varphi$ induces a map
$$
\overline{\varphi}: \Xi_{D_n} \rightarrow \Xi^{n-1,n}, ~(t_1, \ldots , t_n) \mapsto (\varphi(t_1), \ldots , \varphi(t_n)).
$$
This map is well defined. Since the tuple $(t_1, \ldots , t_n)$ generates the group $S_n^D$ and since $\varphi$ is surjective, the tuple $(\varphi(t_1), \ldots , \varphi(t_n))$ generates the group $\Sym([n])$. Therefore $\varphi(t_1) \cdots \varphi(t_n)$ is quasi-Coxeter, hence an $n$-cycle (see \cite[Lemma 6.3]{BGRW}).

The aim of this section is to show the following. 
\begin{theorem}[{Carter diagrams of type $D_n$}] \label{thm:CarterTypeD}
All Carter diagrams of type $D_n$ are given by the set $\mathcal{A}^{n-1,n}$. 

In particular, Theorem \ref{thm:Kluitmann} provides a procedure to construct all Carter diagrams of type $D_n$. 
\end{theorem}

The strategy of the proof is the following: Each Carter diagram of type $D_n$ is the Carter diagram of a reduced reflection factorization in $\Xi_{D_n}$. We first show that Carter diagrams are preserved under the map $\overline{\varphi}$. We then complete the proof of Theorem \ref{thm:CarterTypeD} by showing that $\overline{\varphi}$ is surjective. We state these observations separately.

\begin{Proposition} \label{prop:ProofCarterD1}
Carter diagrams are preserved under the map $\overline{\varphi}$, that is, if $(t_1, \ldots , t_n) \in \Xi_{D_n}$ then the Carter diagram associated to $(t_1, \ldots , t_n)$ is the same as the Kluitmann diagram associated to $(\varphi(t_1), \ldots , \varphi(t_n)) \in \Xi^{n-1,n}$. 
\end{Proposition}

\begin{Proposition} \label{prop:ProofCarterD2}
The map $\overline{\varphi}$ is surjective. 
\end{Proposition}

\medskip
\begin{Definition}
For a reflection $t=(i,j)(-i,-j) \in T_n^B$ (resp. $t=(i,-i) \in T_n^B$) we define the \defn{support} of $t$ as 
$$
\supp(t) := \{ |i|, |j| \} ~(\text{resp. } \supp(t) := \{ |i| \}.).
$$
\end{Definition}

\medskip
\begin{proof}[{\bf Proof of Proposition \ref{prop:ProofCarterD1}}]
Let $(t_1, \ldots , t_n) \in \Xi_{D_n}$ and $i,j \in [n]$ with $i<j$. We have to show that there is an edge between $t_i$ and $t_j$ in the Carter diagram associated to $(t_1, \ldots , t_n)$ if and only if there is an edge between $\varphi(t_i)$ and $\varphi(t_j)$ in the Kluitmann diagram associated to $(\varphi(t_1), \ldots , \varphi(t_n))$.

Let us first assume that there is an edge between $t_i$ and $t_j$. By Carter's Lemma \ref{thm:CartersLemma} this is equivalent to $t_i t_j \neq t_j t_i$. We conclude
\begin{align*}
t_i t_j \neq t_j t_i ~\text{and } t_i \neq t_j &~\Leftrightarrow ~ 
|\supp(t_i) \cap \supp(t_j)|=1 \\
{} &~\Leftrightarrow ~ |\supp(\varphi(t_i)) \cap \supp(\varphi(t_j))|=1\\
{} &~\Leftrightarrow ~ \varphi(t_i) \varphi(t_j) \neq \varphi(t_j) \varphi(t_i) ~\text{and } \varphi(t_i) \neq \varphi(t_j),
\end{align*}
that is, $\varphi(t_i)$ and $\varphi(t_j)$ are connected by an edge in the Kluitmann diagram. 

Now consider the case that there is no edge between $t_i$ and $t_j$. This is equivalent to 
$$
|\supp(t_i) \cap \supp(t_j)| \in \{ 0,2 \}.
$$
By the definition of $\varphi$ we have 
$$
\supp(t_i) \cap \supp(t_j) = \supp(\varphi(t_i)) \cap \varphi(\supp(t_j)).
$$
Hence there is no edge between $t_i$ and $t_j$ if and only if there is no edge between $\varphi(t_i)$ and $\varphi(t_j)$.
\end{proof}

As a preparation for the proof of Proposition \ref{prop:ProofCarterD2} we show:
\begin{Lemma} \label{lem:HurwitzEquiD}
The map $\overline{\varphi}$ is equivariant with respect to the Hurwitz action.
\end{Lemma}

\begin{proof}
Let $\underline{t}:= (t_1, \ldots ,t_i, t_{i+1}, \ldots , t_n ) \in \Xi_{D_n}$ and $i \in [n-1]$. It is enough to show that 
$$
\sigma_i(\overline{\varphi}(\underline{t}))= \overline{\varphi}(\sigma_i(\underline{t})).
$$
Let $t_i= (k_1, \ell_1)(-k_1, -\ell_1)$ and $t_{i+1}= (k_2, \ell_2)(-k_2, - \ell_2)$ with $1 \leq k_j < |\ell_j| \leq n$ for $j \in \{1,2\}$. 

Let us first assume that $t_it_{i+1}= t_{i+1}t_i$. Then $\sigma_i(\underline{t})= (\ldots, t_{i+1}, t_i, \ldots )$, thus $\overline{\varphi}(\sigma_i(\underline{t})) = (\ldots, (k_2, |\ell_2|), (k_1, |\ell_1|), \ldots )$. On the other hand we have 
$$
\sigma_i(\overline{\varphi}(\underline{t}))= \sigma_i (\ldots, (k_1, |\ell_1|), (k_2, |\ell_2|), \ldots ) = (\ldots, (k_2, |\ell_2|)^{(k_1, |\ell_1|)}, (k_1, |\ell_1|), \ldots ).
$$
But $(k_2, |\ell_2|)^{(k_1, |\ell_1|)} = (k_2, |\ell_2|)$, since $|\supp(t_i)\cap \supp(t_{i+1})| \in \{ 0,2 \}$ (see also the proof of Proposition \ref{prop:ProofCarterD1}). Hence $\sigma_i(\overline{\varphi}(\underline{t}))= \overline{\varphi}(\sigma_i(\underline{t}))$.

Now assume that $t_i$ and $t_{i+1}$ do not commute. This implies $|\supp(t_i)\cap \supp(t_{i+1})|=1$. Let us assume that $\ell_1= k_2$ (the other cases can be treated analogously). Then 
\begin{align*}
\overline{\varphi}(\sigma_i(\underline{t})) & = 
\overline{\varphi}(\ldots, (k_2, \ell_2)(-k_2, -\ell_2)^{(k_1, \ell_1)(-k_1, -\ell_1)}, (k_1, \ell_1)(-k_1, -\ell_1), \ldots)\\
{} & = 
\overline{\varphi}(\ldots, (k_2, \ell_2)(-k_2, -\ell_2)^{(k_1, k_2)(-k_1, -k_2)}, (k_1, k_2)(-k_1, -k_2), \ldots)\\
{} & = 
\overline{\varphi}(\ldots, (k_1, \ell_2)(-k_1, -\ell_2), (k_1, k_2)(-k_1, -k_2), \ldots)\\
{} & = (\ldots, (k_1, |\ell_2|), (k_1, k_2), \ldots)
\end{align*}
and
\begin{align*}
\sigma_i(\overline{\varphi}(\underline{t})) & = 
\sigma_i (\ldots, (k_1, |\ell_1|), (k_2, |\ell_2|), \ldots)\\
{} & = 
\sigma_i (\ldots, (k_1, k_2), (k_2, |\ell_2|), \ldots)\\
{} & = 
(\ldots, (k_2, |\ell_2|)^{(k_1, k_2)}, (k_1, k_2), \ldots)\\
{} & = 
(\ldots, (k_1, |\ell_2|), (k_1, k_2), \ldots)
\end{align*}

\end{proof}

\medskip
\begin{proof}[{\bf Proof of Proposition \ref{prop:ProofCarterD2}}]
Let $(\tau_1, \ldots , \tau_n) \in \Xi^{n-1, n }$ be arbitrary. By Lemma \ref{lem:HurwitzEquiD} and by \cite[Corollary 1.4]{LR16}, we can apply the Hurwitz action and assume that 
\begin{align} \label{equ:EquivD1}
\tau_{n-1}= \tau_n ~\text{and } \langle \tau_1, \ldots , \tau_{n-1} \rangle = \Sym([n]).
\end{align}
Let $\tau_i = (k_i, \ell_i)$ with $k_i , \ell_i \in [n]$ and $k_i < \ell_i$ for all $i \in [n-1]$. Put
\begin{align*}
\overline{\tau}_i & := (k_i, \ell_i)(-k_i, -\ell_i) ~\text{for }1 \leq i \leq n-1,\\
\overline{\tau}_n & := (k_{n-1}, -\ell_{n-1})(-k_{n-1}, \ell_{n-1}).
\end{align*}
Since $\varphi(\overline{\tau}_i) = \tau_i$ for all $i \in [n]$, it remains to show that $(\overline{\tau}_1, \ldots,  \overline{\tau}_{n-1}, \overline{\tau}_n) \in \Xi_{D_n}$, that is 
$$
W := \langle \overline{\tau}_1, \ldots,  \overline{\tau}_{n-1}, \overline{\tau}_n \rangle  = S_n^D.
$$
Equivalentely, we have to show that $\varphi_D(W) \cong \ZZ_2^{n-1} \rtimes_{\theta} \Sym([n])$, where we identify $\ZZ_2^{n-1}$ with the subgroup $\{ (d_1, \ldots d_n) \in \ZZ_2^n \mid d_1 + \ldots +d_n ~\text{is even} \} \leq \ZZ_2^{n}$. We have 
\begin{align*}
\varphi_D(\overline{\tau}_i) & = (0, \tau_i) ~\text{for }1 \leq i \leq n-1,\\
\varphi_D(\overline{\tau}_n) & = (e_{k_{n-1}}+ e_{\ell_{n-1}}, \tau_{n-1}).
\end{align*}
Choose $\pi \in \Sym([n])$ with $\pi(k_{n-1})=1$ and $\pi(\ell_{n-1})=2$. By (\ref{equ:EquivD1}) there exist $i_1, \ldots ,i_m \in [n-1]$ with $\pi = \tau_{i_1} \cdots \tau_{i_m}$. 
\begin{align*}
((0, \tau_{i_1}), \ldots , (0, \tau_{i_m})) \cdot (e_{k_{n-1}}+ e_{\ell_{n-1}}, \tau_{n-1}) & = (0, \pi)\cdot (e_{k_{n-1}}+ e_{\ell_{n-1}}, \tau_{n-1})\\
{} & = (e_{\pi(k_{n-1})}+ e_{\pi(\ell_{n-1})}, \pi \tau_{n-1})\\
{} & = (e_1+e_2, \pi \tau_{n-1}) \in \varphi_D(W).
\end{align*}
By (\ref{equ:EquivD1}) we have $(0, (\pi \tau_{n-1})^{-1}(1,2)) \in \varphi_D(W)$. 
Therefore we obtain
\begin{align*}
(e_1+e_2, \pi \tau_{n-1}) \cdot (0, (\pi \tau_{n-1})^{-1}(1,2)) & = ((e_1+e_2)+ \theta(\pi \tau_{n-1})(0), (\pi \tau_{n-1})(\pi \tau_{n-1})^{-1}(1,2))\\
{} & = (e_1 + e_2, (1,2)) \in \varphi_D(W).
\end{align*}
By (\ref{equ:EquivD1}) we have $(0, (i,i+1)) \in \varphi_D(W)$ for all $i \in [n-1]$. These elements together with the element $(e_1 + e_2, (1,2)) \in \varphi_D(W)$ are a generating set of $\ZZ_2^{n-1} \rtimes_{\theta} \Sym([n])$.
\end{proof}


\medskip
\subsubsection{{\bf Carter diagrams of exceptional types}}
We list all non-admissible Carter diagrams of the exceptional types $E_6$ and $F_4$ in Figure \ref{fig:CarterE6} and Figure \ref{fig:CarterF4}, respectively. The admissible ones can be found in \cite{Car72}. The only Carter diagram of type $G_2$ is the corresponding Dynkin diagram. For type $E_7$ there are up to isomorphism $233$ Carter diagrams, some examples are shown in Figure \ref{fig:CarterE7}. For type $E_8$ there are up to isomorphism $1242$ Carter diagrams, one example is shown in Figure \ref{fig:CarterE8}. The complete lists for the types $E_7$ and $E_8$ can be found on the author's webpage:
\begin{center}
\url{http://www-user.rhrk.uni-kl.de/~wegener/}
\end{center}

\begin{figure}[H]
\centering
\begin{tikzpicture}[scale=0.5, thick,>=latex]

  \node (a1) at (0,0) [circle, draw, inner sep=0pt, minimum width=4pt]{};
  \node (a2) at (2,0) [circle, draw, inner sep=0pt, minimum width=4pt]{};
  \node (a3) at (4,0) [circle, draw, inner sep=0pt, minimum width=4pt]{};
  \node (a4) at (0,-2) [circle, draw, inner sep=0pt, minimum width=4pt]{};
  \node (a5) at (2,-2) [circle, draw, inner sep=0pt, minimum width=4pt]{};
  \node (a6) at (4,-2) [circle, draw, inner sep=0pt, minimum width=4pt]{};

  \draw[-] (a1) to (a2);
  \draw[-] (a2) to (a3);
  \draw[-] (a1) to (a4);
  \draw[-] (a2) to (a5);
  \draw[-] (a2) to (a4);
  \draw[-] (a4) to (a5);
  \draw[-] (a5) to (a6);

  \node (b1) at (6,0) [circle, draw, inner sep=0pt, minimum width=4pt]{};
  \node (b2) at (8,0) [circle, draw, inner sep=0pt, minimum width=4pt]{};
  \node (b3) at (10,0) [circle, draw, inner sep=0pt, minimum width=4pt]{};
  \node (b4) at (6,-2) [circle, draw, inner sep=0pt, minimum width=4pt]{};
  \node (b5) at (8,-2) [circle, draw, inner sep=0pt, minimum width=4pt]{};
  \node (b6) at (10,-2) [circle, draw, inner sep=0pt, minimum width=4pt]{};

  \draw[-] (b1) to (b2);
  \draw[-] (b2) to (b3);
  \draw[-] (b1) to (b4);
  \draw[-] (b2) to (b4);
  \draw[-] (b2) to (b5);
  \draw[-] (b3) to (b6);

  \node (c1) at (12,0) [circle, draw, inner sep=0pt, minimum width=4pt]{};
  \node (c2) at (14,0) [circle, draw, inner sep=0pt, minimum width=4pt]{};
  \node (c3) at (16,0) [circle, draw, inner sep=0pt, minimum width=4pt]{};
  \node (c4) at (12,-2) [circle, draw, inner sep=0pt, minimum width=4pt]{};
  \node (c5) at (14,-2) [circle, draw, inner sep=0pt, minimum width=4pt]{};
  \node (c6) at (16,-2) [circle, draw, inner sep=0pt, minimum width=4pt]{};

  \draw[-] (c1) to (c2);
  \draw[-] (c2) to (c3);
  \draw[-] (c1) to (c4);
  \draw[-] (c2) to (c5);
  \draw[-] (c2) to (c6);
  \draw[-] (c4) to (c5);
  \draw[-] (c5) to (c6);

  \node (d1) at (18,0) [circle, draw, inner sep=0pt, minimum width=4pt]{};
  \node (d2) at (20,0) [circle, draw, inner sep=0pt, minimum width=4pt]{};
  \node (d3) at (22,0) [circle, draw, inner sep=0pt, minimum width=4pt]{};
  \node (d4) at (24,-2) [circle, draw, inner sep=0pt, minimum width=4pt]{};
  \node (d5) at (20,-2) [circle, draw, inner sep=0pt, minimum width=4pt]{};
  \node (d6) at (22,-2) [circle, draw, inner sep=0pt, minimum width=4pt]{};

  \draw[-] (d1) to (d2);
  \draw[-] (d2) to (d3);
  \draw[-] (d2) to (d5);
  \draw[-] (d3) to (d5);
  \draw[-] (d3) to (d6);
  \draw[-] (d5) to (d6);
  \draw[-] (d6) to (d4);

  \node (e1) at (0,-4) [circle, draw, inner sep=0pt, minimum width=4pt]{};
  \node (e2) at (2,-4) [circle, draw, inner sep=0pt, minimum width=4pt]{};
  \node (e3) at (4,-4) [circle, draw, inner sep=0pt, minimum width=4pt]{};
  \node (e4) at (0,-6) [circle, draw, inner sep=0pt, minimum width=4pt]{};
  \node (e5) at (2,-6) [circle, draw, inner sep=0pt, minimum width=4pt]{};
  \node (e6) at (4,-6) [circle, draw, inner sep=0pt, minimum width=4pt]{};

  \draw[-] (e1) to (e2);
  \draw[-] (e2) to (e3);
  \draw[-] (e1) to (e4);
  \draw[-] (e2) to (e4);
  \draw[-] (e2) to (e5);
  \draw[-] (e2) to (e6);
  \draw[-] (e5) to (e6);

  \node (f1) at (6,-4) [circle, draw, inner sep=0pt, minimum width=4pt]{};
  \node (f2) at (8,-4) [circle, draw, inner sep=0pt, minimum width=4pt]{};
  \node (f3) at (10,-4) [circle, draw, inner sep=0pt, minimum width=4pt]{};
  \node (f4) at (6,-6) [circle, draw, inner sep=0pt, minimum width=4pt]{};
  \node (f5) at (8,-6) [circle, draw, inner sep=0pt, minimum width=4pt]{};
  \node (f6) at (10,-6) [circle, draw, inner sep=0pt, minimum width=4pt]{};

  \draw[-] (f1) to (f2);
  \draw[-] (f2) to (f3);
  \draw[-] (f1) to (f4);
  \draw[-] (f1) to (f5);
  \draw[-] (f2) to (f4);
  \draw[-] (f2) to (f5);
  \draw[-] (f3) to (f5);
  \draw[-] (f3) to (f6);
  \draw[-] (f4) to (f5);

  \node (g1) at (12,-4) [circle, draw, inner sep=0pt, minimum width=4pt]{};
  \node (g2) at (14,-4) [circle, draw, inner sep=0pt, minimum width=4pt]{};
  \node (g3) at (16,-4) [circle, draw, inner sep=0pt, minimum width=4pt]{};
  \node (g4) at (12,-6) [circle, draw, inner sep=0pt, minimum width=4pt]{};
  \node (g5) at (14,-6) [circle, draw, inner sep=0pt, minimum width=4pt]{};
  \node (g6) at (16,-6) [circle, draw, inner sep=0pt, minimum width=4pt]{};

  \draw[-] (g1) to (g2);
  \draw[-] (g2) to (g3);
  \draw[-] (g1) to (g4);
  \draw[-] (g2) to (g4);
  \draw[-] (g2) to (g5);
  \draw[-] (g3) to (g5);
  \draw[-] (g4) to (g5);
  \draw[-] (g5) to (g6);

  \node (h1) at (18,-4) [circle, draw, inner sep=0pt, minimum width=4pt]{};
  \node (h2) at (20,-4) [circle, draw, inner sep=0pt, minimum width=4pt]{};
  \node (h3) at (22,-4) [circle, draw, inner sep=0pt, minimum width=4pt]{};
  \node (h4) at (18,-6) [circle, draw, inner sep=0pt, minimum width=4pt]{};
  \node (h5) at (20,-6) [circle, draw, inner sep=0pt, minimum width=4pt]{};
  \node (h6) at (22,-6) [circle, draw, inner sep=0pt, minimum width=4pt]{};

  \draw[-] (h1) to (h2);
  \draw[-] (h2) to (h3);
  \draw[-] (h1) to (h4);
  \draw[-] (h1) to (h5);
  \draw[-] (h2) to (h4);
  \draw[-] (h2) to (h5);
  \draw[-] (h3) to (h5);
  \draw[-] (h5) to (h6);
  \draw[-] (h4) to (h5);

  \node (i1) at (0,-8) [circle, draw, inner sep=0pt, minimum width=4pt]{};
  \node (i2) at (2,-8) [circle, draw, inner sep=0pt, minimum width=4pt]{};
  \node (i3) at (4,-8) [circle, draw, inner sep=0pt, minimum width=4pt]{};
  \node (i4) at (0,-10) [circle, draw, inner sep=0pt, minimum width=4pt]{};
  \node (i5) at (2,-10) [circle, draw, inner sep=0pt, minimum width=4pt]{};
  \node (i6) at (4,-10) [circle, draw, inner sep=0pt, minimum width=4pt]{};

  \draw[-] (i1) to (i2);
  \draw[-] (i2) to (i3);
  \draw[-] (i1) to (i4);
  \draw[-] (i3) to (i5);
  \draw[-] (i3) to (i6);
  \draw[-] (i4) to (i5);

  \node (j1) at (6,-8) [circle, draw, inner sep=0pt, minimum width=4pt]{};
  \node (j2) at (8,-8) [circle, draw, inner sep=0pt, minimum width=4pt]{};
  \node (j3) at (10,-8) [circle, draw, inner sep=0pt, minimum width=4pt]{};
  \node (j4) at (6,-10) [circle, draw, inner sep=0pt, minimum width=4pt]{};
  \node (j5) at (8,-10) [circle, draw, inner sep=0pt, minimum width=4pt]{};
  \node (j6) at (10,-10) [circle, draw, inner sep=0pt, minimum width=4pt]{};

  \draw[-] (j1) to (j2);
  \draw[-] (j2) to (j3);
  \draw[-] (j1) to (j6);
  \draw[-] (j2) to (j5);
  \draw[-] (j2) to (j6);
  \draw[-] (j3) to (j5);
  \draw[-] (j3) to (j6);
  \draw[-] (j4) to (j5);
  \draw[-] (j5) to (j6);

  \node (k1) at (12,-8) [circle, draw, inner sep=0pt, minimum width=4pt]{};
  \node (k2) at (14,-8) [circle, draw, inner sep=0pt, minimum width=4pt]{};
  \node (k3) at (16,-8) [circle, draw, inner sep=0pt, minimum width=4pt]{};
  \node (k4) at (12,-10) [circle, draw, inner sep=0pt, minimum width=4pt]{};
  \node (k5) at (14,-10) [circle, draw, inner sep=0pt, minimum width=4pt]{};
  \node (k6) at (16,-10) [circle, draw, inner sep=0pt, minimum width=4pt]{};

  \draw[-] (k1) to (k2);
  \draw[-] (k2) to (k3);
  \draw[-] (k1) to (k4);
  \draw[-] (k1) to (k5);
  \draw[-] (k2) to (k4);
  \draw[-] (k2) to (k6);
  \draw[-] (k3) to (k6);
  \draw[-] (k4) to (k5);
  \draw[-] (k5) to (k6);

  \node (l1) at (18,-8) [circle, draw, inner sep=0pt, minimum width=4pt]{};
  \node (l2) at (20,-8) [circle, draw, inner sep=0pt, minimum width=4pt]{};
  \node (l3) at (22,-8) [circle, draw, inner sep=0pt, minimum width=4pt]{};
  \node (l4) at (18,-10) [circle, draw, inner sep=0pt, minimum width=4pt]{};
  \node (l5) at (20,-10) [circle, draw, inner sep=0pt, minimum width=4pt]{};
  \node (l6) at (22,-10) [circle, draw, inner sep=0pt, minimum width=4pt]{};

  \draw[-] (l1) to (l2);
  \draw[-] (l2) to (l3);
  \draw[-] (l1) to (l4);
  \draw[-] (l2) to (l4);
  \draw[-] (l2) to (l5);
  \draw[-] (l2) to (l6);
  \draw[-] (l5) to (l6);
  \draw[-] (l4) to (l5);

  \node (m1) at (0,-12) [circle, draw, inner sep=0pt, minimum width=4pt]{};
  \node (m2) at (2,-12) [circle, draw, inner sep=0pt, minimum width=4pt]{};
  \node (m3) at (4,-12) [circle, draw, inner sep=0pt, minimum width=4pt]{};
  \node (m4) at (0,-14) [circle, draw, inner sep=0pt, minimum width=4pt]{};
  \node (m5) at (2,-14) [circle, draw, inner sep=0pt, minimum width=4pt]{};
  \node (m6) at (4,-14) [circle, draw, inner sep=0pt, minimum width=4pt]{};

  \draw[-] (m1) to (m2);
  \draw[-] (m2) to (m3);
  \draw[-] (m1) to (m4);
  \draw[-] (m2) to (m5);
  \draw[-] (m2) to (m6);
  \draw[-] (m3) to (m6);
  \draw[-] (m4) to (m5);
  \draw[-] (m5) to (m6);

  \node (n1) at (6,-12) [circle, draw, inner sep=0pt, minimum width=4pt]{};
  \node (n2) at (8,-12) [circle, draw, inner sep=0pt, minimum width=4pt]{};
  \node (n3) at (10,-12) [circle, draw, inner sep=0pt, minimum width=4pt]{};
  \node (n4) at (6,-14) [circle, draw, inner sep=0pt, minimum width=4pt]{};
  \node (n5) at (8,-14) [circle, draw, inner sep=0pt, minimum width=4pt]{};
  \node (n6) at (10,-14) [circle, draw, inner sep=0pt, minimum width=4pt]{};

  \draw[-] (n1) to (n2);
  \draw[-] (n2) to (n3);
  \draw[-] (n1) to (n4);
  \draw[-] (n2) to (n4);
  \draw[-] (n2) to (n5);
  \draw[-] (n2) to (n6);
  \draw[-] (n3) to (n6);
  \draw[-] (n4) to (n5);
  \draw[-] (n5) to (n6);

  \node (o1) at (12,-12) [circle, draw, inner sep=0pt, minimum width=4pt]{};
  \node (o2) at (14,-12) [circle, draw, inner sep=0pt, minimum width=4pt]{};
  \node (o3) at (16,-12) [circle, draw, inner sep=0pt, minimum width=4pt]{};
  \node (o4) at (12,-14) [circle, draw, inner sep=0pt, minimum width=4pt]{};
  \node (o5) at (14,-14) [circle, draw, inner sep=0pt, minimum width=4pt]{};
  \node (o6) at (16,-14) [circle, draw, inner sep=0pt, minimum width=4pt]{};

  \draw[-] (o1) to (o2);
  \draw[-] (o2) to (o3);
  \draw[-] (o1) to (o4);
  \draw[-] (o1) to (o6);
  \draw[-] (o2) to (o5);
  \draw[-] (o2) to (o6);
  \draw[-] (o3) to (o5);
  \draw[-] (o4) to (o5);
  \draw[-] (o5) to (o6);

  \node (p1) at (18,-12) [circle, draw, inner sep=0pt, minimum width=4pt]{};
  \node (p2) at (20,-12) [circle, draw, inner sep=0pt, minimum width=4pt]{};
  \node (p3) at (22,-12) [circle, draw, inner sep=0pt, minimum width=4pt]{};
  \node (p4) at (18,-14) [circle, draw, inner sep=0pt, minimum width=4pt]{};
  \node (p5) at (20,-14) [circle, draw, inner sep=0pt, minimum width=4pt]{};
  \node (p6) at (22,-14) [circle, draw, inner sep=0pt, minimum width=4pt]{};

  \draw[-] (p1) to (p2);
  \draw[-] (p2) to (p3);
  \draw[-] (p1) to (p4);
  \draw[-] (p1) to (p5);
  \draw[-] (p2) to (p4);
  \draw[-] (p2) to (p5);
  \draw[-] (p2) to (p6);
  \draw[-] (p3) to (p5);
  \draw[-] (p3) to (p6);
  \draw[-] (p5) to (p6);
  \draw[-] (p4) to (p5);

  \node (q1) at (0,-16) [circle, draw, inner sep=0pt, minimum width=4pt]{};
  \node (q2) at (2,-16) [circle, draw, inner sep=0pt, minimum width=4pt]{};
  \node (q3) at (4,-16) [circle, draw, inner sep=0pt, minimum width=4pt]{};
  \node (q4) at (0,-18) [circle, draw, inner sep=0pt, minimum width=4pt]{};
  \node (q5) at (2,-18) [circle, draw, inner sep=0pt, minimum width=4pt]{};
  \node (q6) at (4,-18) [circle, draw, inner sep=0pt, minimum width=4pt]{};

  \draw[-] (q1) to (q2);
  \draw[-] (q2) to (q3);
  \draw[-] (q1) to (q4);
  \draw[-] (q1) to (q5);
  \draw[-] (q2) to (q4);
  \draw[-] (q3) to (q6);
  \draw[-] (q4) to (q5);
  \draw[-] (q5) to (q6);

  \node (r1) at (6,-16) [circle, draw, inner sep=0pt, minimum width=4pt]{};
  \node (r2) at (8,-16) [circle, draw, inner sep=0pt, minimum width=4pt]{};
  \node (r3) at (10,-16) [circle, draw, inner sep=0pt, minimum width=4pt]{};
  \node (r4) at (6,-18) [circle, draw, inner sep=0pt, minimum width=4pt]{};
  \node (r5) at (8,-18) [circle, draw, inner sep=0pt, minimum width=4pt]{};
  \node (r6) at (10,-18) [circle, draw, inner sep=0pt, minimum width=4pt]{};

  \draw[-] (r1) to (r2);
  \draw[-] (r2) to (r3);
  \draw[-] (r1) to (r5);
  \draw[-] (r1) to (r6);
  \draw[-] (r2) to (r4);
  \draw[-] (r2) to (r5);
  \draw[-] (r2) to (r6);
  \draw[-] (r3) to (r4);
  \draw[-] (r3) to (r5);
  \draw[-] (r3) to (r6);
  \draw[-] (r5) to (r6);

  \node (s1) at (12,-16) [circle, draw, inner sep=0pt, minimum width=4pt]{};
  \node (s2) at (14,-16) [circle, draw, inner sep=0pt, minimum width=4pt]{};
  \node (s3) at (16,-16) [circle, draw, inner sep=0pt, minimum width=4pt]{};
  \node (s4) at (12,-18) [circle, draw, inner sep=0pt, minimum width=4pt]{};
  \node (s5) at (14,-18) [circle, draw, inner sep=0pt, minimum width=4pt]{};
  \node (s6) at (16,-18) [circle, draw, inner sep=0pt, minimum width=4pt]{};

  \draw[-] (s1) to (s2);
  \draw[-] (s2) to (s3);
  \draw[-] (s1) to (s4);
  \draw[-] (s2) to (s4);
  \draw[-] (s2) to (s5);
  \draw[-] (s3) to (s5);
  \draw[-] (s3) to (s6);
  \draw[-] (s4) to (s5);
  \draw[-] (s5) to (s6);

  \node (t1) at (18,-16) [circle, draw, inner sep=0pt, minimum width=4pt]{};
  \node (t2) at (20,-16) [circle, draw, inner sep=0pt, minimum width=4pt]{};
  \node (t3) at (22,-16) [circle, draw, inner sep=0pt, minimum width=4pt]{};
  \node (t4) at (18,-18) [circle, draw, inner sep=0pt, minimum width=4pt]{};
  \node (t5) at (20,-18) [circle, draw, inner sep=0pt, minimum width=4pt]{};
  \node (t6) at (22,-18) [circle, draw, inner sep=0pt, minimum width=4pt]{};

  \draw[-] (t1) to (t2);
  \draw[-] (t2) to (t3);
  \draw[-] (t1) to (t4);
  \draw[-] (t2) to (t5);
  \draw[-] (t3) to (t4);
  \draw[-] (t3) to (t5);
  \draw[-] (t2) to (t6);
  \draw[-] (t4) to (t5);

  \node (u1) at (0,-20) [circle, draw, inner sep=0pt, minimum width=4pt]{};
  \node (u2) at (2,-20) [circle, draw, inner sep=0pt, minimum width=4pt]{};
  \node (u3) at (4,-20) [circle, draw, inner sep=0pt, minimum width=4pt]{};
  \node (u4) at (0,-22) [circle, draw, inner sep=0pt, minimum width=4pt]{};
  \node (u5) at (2,-22) [circle, draw, inner sep=0pt, minimum width=4pt]{};
  \node (u6) at (4,-22) [circle, draw, inner sep=0pt, minimum width=4pt]{};

  \draw[-] (u1) to (u2);
  \draw[-] (u2) to (u3);
  \draw[-] (u1) to (u4);
  \draw[-] (u1) to (u5);
  \draw[-] (u2) to (u4);
  \draw[-] (u2) to (u5);
  \draw[-] (u4) to (u3);
  \draw[-] (u3) to (u5);
  \draw[-] (u4) to (u5);
  \draw[-] (u3) to (u6);

  \node (v1) at (6,-20) [circle, draw, inner sep=0pt, minimum width=4pt]{};
  \node (v2) at (8,-20) [circle, draw, inner sep=0pt, minimum width=4pt]{};
  \node (v3) at (10,-20) [circle, draw, inner sep=0pt, minimum width=4pt]{};
  \node (v4) at (6,-22) [circle, draw, inner sep=0pt, minimum width=4pt]{};
  \node (v5) at (8,-22) [circle, draw, inner sep=0pt, minimum width=4pt]{};
  \node (v6) at (10,-22) [circle, draw, inner sep=0pt, minimum width=4pt]{};

  \draw[-] (v1) to (v2);
  \draw[-] (v2) to (v3);
  \draw[-] (v1) to (v6);
  \draw[-] (v2) to (v5);
  \draw[-] (v2) to (v6);
  \draw[-] (v1) to (v4);
  \draw[-] (v3) to (v6);
  \draw[-] (v3) to (v5);
  \draw[-] (v4) to (v5);
  \draw[-] (v5) to (v6);

  \node (w1) at (12,-20) [circle, draw, inner sep=0pt, minimum width=4pt]{};
  \node (w2) at (14,-20) [circle, draw, inner sep=0pt, minimum width=4pt]{};
  \node (w3) at (16,-20) [circle, draw, inner sep=0pt, minimum width=4pt]{};
  \node (w4) at (12,-22) [circle, draw, inner sep=0pt, minimum width=4pt]{};
  \node (w5) at (14,-22) [circle, draw, inner sep=0pt, minimum width=4pt]{};
  \node (w6) at (16,-22) [circle, draw, inner sep=0pt, minimum width=4pt]{};

  \draw[-] (w1) to (w2);
  \draw[-] (w2) to (w3);
  \draw[-] (w1) to (w4);
  \draw[-] (w1) to (w5);
  \draw[-] (w1) to (w6);
  \draw[-] (w2) to (w4);
  \draw[-] (w2) to (w5);
  \draw[-] (w2) to (w6);
  \draw[-] (w3) to (w5);
  \draw[-] (w3) to (w6);
  \draw[-] (w4) to (w5);
  \draw[-] (w5) to (w6);
  \draw[bend right=20] (w4) to (w6);

  \node (x1) at (18,-20) [circle, draw, inner sep=0pt, minimum width=4pt]{};
  \node (x2) at (20,-20) [circle, draw, inner sep=0pt, minimum width=4pt]{};
  \node (x3) at (22,-20) [circle, draw, inner sep=0pt, minimum width=4pt]{};
  \node (x4) at (18,-22) [circle, draw, inner sep=0pt, minimum width=4pt]{};
  \node (x5) at (20,-22) [circle, draw, inner sep=0pt, minimum width=4pt]{};
  \node (x6) at (22,-22) [circle, draw, inner sep=0pt, minimum width=4pt]{};

  \draw[-] (x1) to (x2);
  \draw[-] (x2) to (x3);
  \draw[-] (x1) to (x4);
  \draw[-] (x1) to (x5);
  \draw[-] (x2) to (x4);
  \draw[-] (x3) to (x6);
  \draw[-] (x4) to (x5);
  \draw[-] (x5) to (x6);
  \draw[bend right=20] (x4) to (x6);

  \node (y1) at (0,-24) [circle, draw, inner sep=0pt, minimum width=4pt]{};
  \node (y2) at (2,-24) [circle, draw, inner sep=0pt, minimum width=4pt]{};
  \node (y3) at (4,-24) [circle, draw, inner sep=0pt, minimum width=4pt]{};
  \node (y4) at (0,-26) [circle, draw, inner sep=0pt, minimum width=4pt]{};
  \node (y5) at (2,-26) [circle, draw, inner sep=0pt, minimum width=4pt]{};
  \node (y6) at (4,-26) [circle, draw, inner sep=0pt, minimum width=4pt]{};

  \draw[-] (y1) to (y2);
  \draw[-] (y2) to (y3);
  \draw[-] (y1) to (y4);
  \draw[-] (y1) to (y5);
  \draw[-] (y2) to (y5);
  \draw[-] (y2) to (y6);
  \draw[-] (y3) to (y4);
  \draw[-] (y3) to (y6);
  \draw[-] (y4) to (y5);
  \draw[-] (y5) to (y6);

  \node (z1) at (6,-24) [circle, draw, inner sep=0pt, minimum width=4pt]{};
  \node (z2) at (8,-24) [circle, draw, inner sep=0pt, minimum width=4pt]{};
  \node (z3) at (10,-24) [circle, draw, inner sep=0pt, minimum width=4pt]{};
  \node (z4) at (6,-26) [circle, draw, inner sep=0pt, minimum width=4pt]{};
  \node (z5) at (8,-26) [circle, draw, inner sep=0pt, minimum width=4pt]{};
  \node (z6) at (10,-26) [circle, draw, inner sep=0pt, minimum width=4pt]{};

  \draw[-] (z1) to (z2);
  \draw[-] (z2) to (z3);
  \draw[-] (z1) to (z4);
  \draw[-] (z1) to (z5);
  \draw[-] (z2) to (z4);
  \draw[-] (z2) to (z5);
  \draw[-] (z2) to (z6);
  \draw[-] (z3) to (z5);
  \draw[-] (z3) to (z6);
  \draw[-] (z4) to (z5);
  \draw[-] (z5) to (z6);
  \draw[bend right=20] (z4) to (z6);

  \node (zz1) at (12,-24) [circle, draw, inner sep=0pt, minimum width=4pt]{};
  \node (zz2) at (14,-24) [circle, draw, inner sep=0pt, minimum width=4pt]{};
  \node (zz3) at (16,-24) [circle, draw, inner sep=0pt, minimum width=4pt]{};
  \node (zz4) at (12,-26) [circle, draw, inner sep=0pt, minimum width=4pt]{};
  \node (zz5) at (14,-26) [circle, draw, inner sep=0pt, minimum width=4pt]{};
  \node (zz6) at (16,-26) [circle, draw, inner sep=0pt, minimum width=4pt]{};

  \draw[-] (zz1) to (zz2);
  \draw[-] (zz2) to (zz3);
  \draw[-] (zz1) to (zz4);
  \draw[-] (zz1) to (zz5);
  \draw[-] (zz1) to (zz6);
  \draw[-] (zz2) to (zz4);
  \draw[-] (zz3) to (zz5);
  \draw[-] (zz3) to (zz6);
  \draw[-] (zz4) to (zz5);
  \draw[-] (zz5) to (zz6);
  \draw[bend right=20] (zz4) to (zz6);

  \node (xx1) at (18,-24) [circle, draw, inner sep=0pt, minimum width=4pt]{};
  \node (xx2) at (20,-24) [circle, draw, inner sep=0pt, minimum width=4pt]{};
  \node (xx3) at (22,-24) [circle, draw, inner sep=0pt, minimum width=4pt]{};
  \node (xx4) at (18,-26) [circle, draw, inner sep=0pt, minimum width=4pt]{};
  \node (xx5) at (20,-26) [circle, draw, inner sep=0pt, minimum width=4pt]{};
  \node (xx6) at (22,-26) [circle, draw, inner sep=0pt, minimum width=4pt]{};

  \draw[-] (xx1) to (xx2);
  \draw[-] (xx2) to (xx6);
  \draw[-] (xx1) to (xx4);
  \draw[-] (xx1) to (xx5);
  \draw[-] (xx2) to (xx4);
  \draw[-] (xx3) to (xx6);
  \draw[-] (xx4) to (xx3);
  \draw[-] (xx5) to (xx6);
  \draw[-] (xx2) to (xx5);
  \draw[-] (xx3) to (xx5);

  \node (uu1) at (0,-28) [circle, draw, inner sep=0pt, minimum width=4pt]{};
  \node (uu2) at (2,-28) [circle, draw, inner sep=0pt, minimum width=4pt]{};
  \node (uu3) at (4,-28) [circle, draw, inner sep=0pt, minimum width=4pt]{};
  \node (uu4) at (0,-30) [circle, draw, inner sep=0pt, minimum width=4pt]{};
  \node (uu5) at (2,-30) [circle, draw, inner sep=0pt, minimum width=4pt]{};
  \node (uu6) at (4,-30) [circle, draw, inner sep=0pt, minimum width=4pt]{};

  \draw[-] (uu1) to (uu2);
  \draw[-] (uu5) to (uu6);
  \draw[-] (uu1) to (uu4);
  \draw[-] (uu1) to (uu5);
  \draw[-] (uu2) to (uu5);
  \draw[-] (uu4) to (uu3);
  \draw[-] (uu3) to (uu5);
  \draw[-] (uu4) to (uu5);
  \draw[-] (uu3) to (uu6);
  \draw[-] (uu2) to (uu6);

\end{tikzpicture}
\caption{Non-admissible Carter diagrams of type $E_6$} \label{fig:CarterE6}
\end{figure}
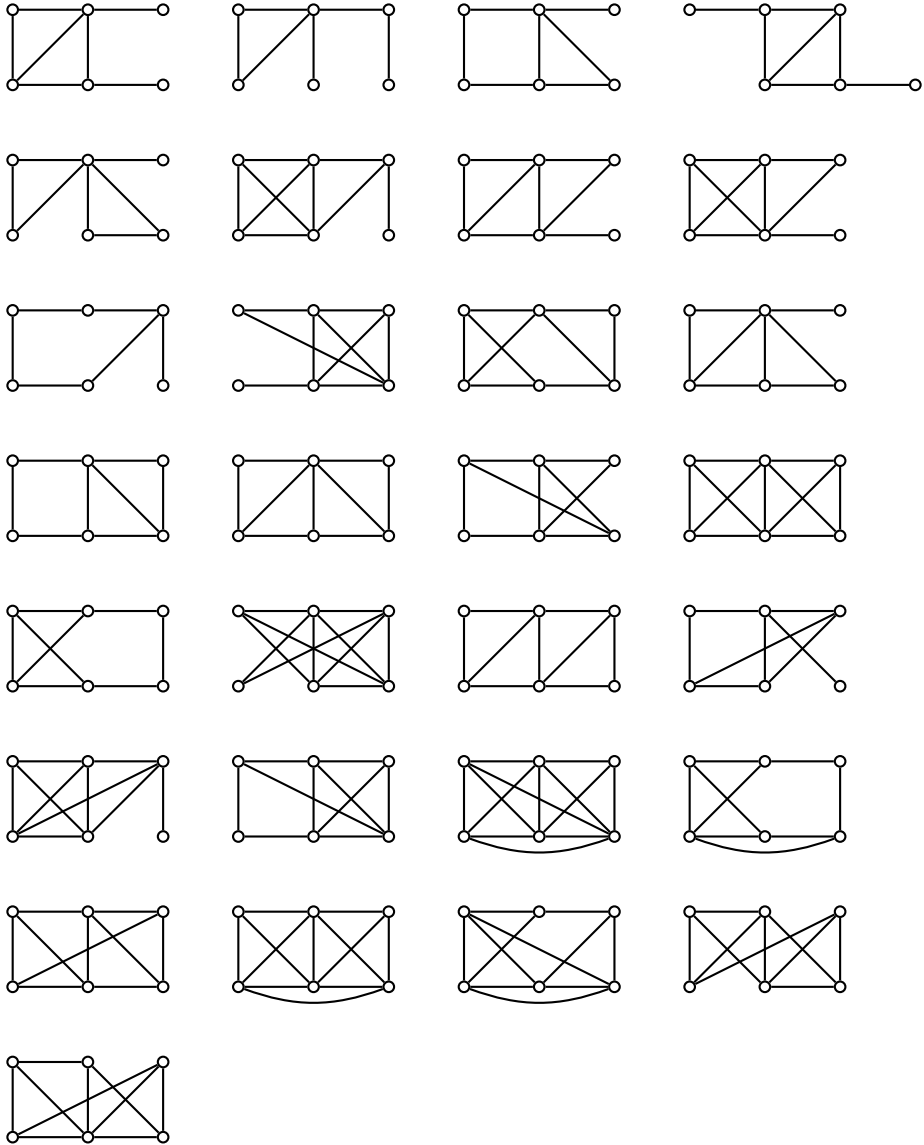

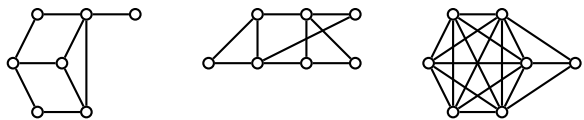
\begin{figure}[H] 
\centering
\begin{tikzpicture}[scale=0.65, thick,>=latex]

  \node (A1) at (0,0) [circle, draw, inner sep=0pt, minimum width=4pt]{};
  \node (A2) at (0.5,1) [circle, draw, inner sep=0pt, minimum width=4pt]{};
  \node (A3) at (0.5,-1) [circle, draw, inner sep=0pt, minimum width=4pt]{};
  \node (A4) at (1,0) [circle, draw, inner sep=0pt, minimum width=4pt]{};
  \node (A5) at (1.5,1) [circle, draw, inner sep=0pt, minimum width=4pt]{};
  \node (A6) at (1.5,-1) [circle, draw, inner sep=0pt, minimum width=4pt]{};
  \node (A7) at (2.5,1) [circle, draw, inner sep=0pt, minimum width=4pt]{};

  \draw[-] (A1) to (A2);
  \draw[-] (A1) to (A3);
  \draw[-] (A1) to (A4);
  \draw[-] (A2) to (A5);
  \draw[-] (A3) to (A6);
  \draw[-] (A4) to (A5);
  \draw[-] (A4) to (A6);
  \draw[-] (A5) to (A6);
  \draw[-] (A5) to (A7);

  \node (B1) at (4,0) [circle, draw, inner sep=0pt, minimum width=4pt]{};
  \node (B2) at (5,1) [circle, draw, inner sep=0pt, minimum width=4pt]{};
  \node (B3) at (5,0) [circle, draw, inner sep=0pt, minimum width=4pt]{};
  \node (B4) at (6,1) [circle, draw, inner sep=0pt, minimum width=4pt]{};
  \node (B5) at (6,0) [circle, draw, inner sep=0pt, minimum width=4pt]{};
  \node (B6) at (7,1) [circle, draw, inner sep=0pt, minimum width=4pt]{};
  \node (B7) at (7,0) [circle, draw, inner sep=0pt, minimum width=4pt]{};
  
  \draw[-] (B1) to (B2);
  \draw[-] (B1) to (B3);
  \draw[-] (B2) to (B3);
  \draw[-] (B2) to (B4);
  \draw[-] (B3) to (B5);
  \draw[-] (B3) to (B6);
  \draw[-] (B4) to (B5);
  \draw[-] (B4) to (B6);
  \draw[-] (B4) to (B7);
  \draw[-] (B5) to (B7);

  \node (C1) at (8.5,0) [circle, draw, inner sep=0pt, minimum width=4pt]{};
  \node (C2) at (9,1) [circle, draw, inner sep=0pt, minimum width=4pt]{};
  \node (C3) at (9,-1) [circle, draw, inner sep=0pt, minimum width=4pt]{};
  \node (C4) at (10.5,0) [circle, draw, inner sep=0pt, minimum width=4pt]{};
  \node (C5) at (10,1) [circle, draw, inner sep=0pt, minimum width=4pt]{};
  \node (C6) at (10,-1) [circle, draw, inner sep=0pt, minimum width=4pt]{};
  \node (C7) at (11.5,0) [circle, draw, inner sep=0pt, minimum width=4pt]{};

  \draw[-] (C1) to (C2);
  \draw[-] (C1) to (C3);
  \draw[-] (C1) to (C4);
  \draw[-] (C1) to (C5);
  \draw[-] (C1) to (C6);
  \draw[-] (C2) to (C3);
  \draw[-] (C2) to (C4);
  \draw[-] (C2) to (C5);
  \draw[-] (C2) to (C6);
  \draw[-] (C3) to (C4);
  \draw[-] (C3) to (C5);
  \draw[-] (C3) to (C6);
  \draw[-] (C4) to (C5);
  \draw[-] (C4) to (C6);
  \draw[-] (C5) to (C6);
  \draw[-] (C4) to (C7);
  \draw[-] (C5) to (C7);
  \draw[-] (C6) to (C7);

\end{tikzpicture}
\caption{Examples of Carter diagrams of type $E_7$} \label{fig:CarterE7}
\end{figure}

\begin{figure}[H] 
\centering
\begin{tikzpicture}[scale=0.65, thick,>=latex]

  \node (B1) at (4,0) [circle, draw, inner sep=0pt, minimum width=4pt]{};
  \node (B2) at (5,1) [circle, draw, inner sep=0pt, minimum width=4pt]{};
  \node (B3) at (5,0) [circle, draw, inner sep=0pt, minimum width=4pt]{};
  \node (B4) at (6,1) [circle, draw, inner sep=0pt, minimum width=4pt]{};
  \node (B5) at (6,0) [circle, draw, inner sep=0pt, minimum width=4pt]{};
  \node (B6) at (7,1) [circle, draw, inner sep=0pt, minimum width=4pt]{};
  \node (B7) at (7,0) [circle, draw, inner sep=0pt, minimum width=4pt]{};
  \node (B8) at (6,-1) [circle, draw, inner sep=0pt, minimum width=4pt]{};
  
  \draw[-] (B1) to (B3);
  \draw[-] (B2) to (B3);
  \draw[-] (B2) to (B5);
  \draw[-] (B2) to (B4);
  \draw[-] (B3) to (B5);
  \draw[-] (B3) to (B4);
  \draw[-] (B4) to (B5);
  \draw[-] (B4) to (B6);
  \draw[-] (B4) to (B7);
  \draw[-] (B5) to (B7);
  \draw[-] (B6) to (B7);
  \draw[-] (B5) to (B8);

\end{tikzpicture}
\caption{A Carter diagram of type $E_8$} \label{fig:CarterE8}
\end{figure}
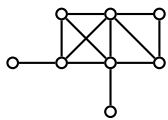

\begin{figure}[H] 
\centering
\begin{tikzpicture}[scale=0.65, thick, >=latex, double equal sign distance]
   \tikzset{triple/.style={double,postaction={draw,-}}}

  \node (A11) at (3,0) [circle, draw, inner sep=0pt, minimum width=4pt]{};
  \node (21) at (1,-1) [circle, draw, inner sep=0pt, minimum width=4pt]{};
  \node (A31) at (3,-2) [circle, draw, inner sep=0pt, minimum width=4pt]{};
  \node (A41) at (2,-1) [circle, draw, inner sep=0pt, minimum width=4pt]{};

  \draw[-] (21) to (A41);
  \draw[double] (A31) to (A41);
  \draw[double] (A41) to (A11);
  \draw[-] (A31) to (A11);

  \node (b1) at (5,0) [circle, draw, inner sep=0pt, minimum width=4pt]{};
  \node (b2) at (7,0) [circle, draw, inner sep=0pt, minimum width=4pt]{};
  \node (b3) at (7,-2) [circle, draw, inner sep=0pt, minimum width=4pt]{};
  \node (b4) at (5,-2) [circle, draw, inner sep=0pt, minimum width=4pt]{};
  
  \draw[double] (b1) to (b2);
  \draw[-] (b1) to (b3);
  \draw[double] (b3) to (b4);
  \draw[double] (b4) to (b1);
  \draw[-] (b2) to (b4);

  \node (c1) at (9,0) [circle, draw, inner sep=0pt, minimum width=4pt]{};
  \node (c2) at (11,0) [circle, draw, inner sep=0pt, minimum width=4pt]{};
  \node (c3) at (11,-2) [circle, draw, inner sep=0pt, minimum width=4pt]{};
  \node (c4) at (9,-2) [circle, draw, inner sep=0pt, minimum width=4pt]{};
  
  \draw[double] (c1) to (c2);
  \draw[double] (c2) to (c3);
  \draw[double] (c3) to (c4);
  \draw[double] (c4) to (c1);
  \draw[-] (c2) to (c4);
  \draw[-] (c1) to (c3);

\end{tikzpicture}
\caption{Non-admissible Carter diagrams of type $F_4$} \label{fig:CarterF4}
\end{figure}
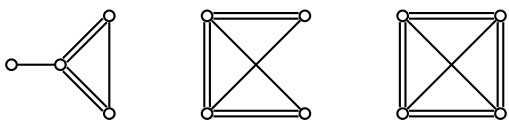

\begin{remark}
Let us explain how we carried out the computations to obtain all Carter diagrams (up to isomorphism) of an exceptional type $X_n \in \{ E_6, E_7, E_8, F_4 \}$. By Remark \ref{rem:MinCarter} we just have to consider Carter diagrams associated to reduced reflection factorizations of quasi Coxeter elements in a Coxeter group of type $X_n$. By Lemma \ref{lem:ConjCarterDiag}, we just have to consider one fixed quasi-Coxeter element for each conjugacy class. Given a quasi-Coxeter element $w$, all reduced reflection factorizations of $w$ are precisely given by the Hurwitz orbit of one given reduced reflection factorization \cite[Theorem 1.1]{BGRW}. The programs to carry out these computations in GAP \cite{GAP2017} can be found at:

\url{https://www.math.uni-bielefeld.de/~baumeist/Dual-Coxeter/dual-Coxeter.html}

\noindent Given a reduced reflection factorization $(t_1, \ldots , t_n)$, the computation of the Carter diagram associated to this factorization is easily done by computing the order of $t_it_j$ for all $1 \leq i < j \leq n$. Finally we used Sage \cite{Sage} to obtain the complete list of Carter diagrams up to isomorphism. Sage provides the command ``$G$.is$_{-}$isomorphic($H$)'' to check whether two graphs $G$ and $H$ are isomorphic. 
\end{remark}

\section{Quiver Mutation} \label{sec:QuiverMutation}

The aim of this section is to establish a connection between Carter diagrams and the mutation classes of Dynkin quivers. Namely we prove Theorem \ref{thm:Main1} in this section

Quiver mutation appears as an important concept in the theory of cluster algebras. We will shortly review the necessary definitions. A \defn{quiver} $Q=(Q_0, Q_1, s,t)$ is a directed graph on vertex set $Q_0$, edges given by the set $Q_1$ and maps $s,t: Q_1 \rightarrow Q_0$ such that $s(\alpha)=i$ and $t(\alpha)=j$ whenever $\alpha \in Q_1$ is an arrow from $i \in Q_0$ to $j \in Q_0$. 

We will assume throughout that quivers do not have loops or $2$-cycles. 

\begin{Definition}
Let $Q=(Q_0, Q_1, s,t)$ be a quiver. The \defn{mutation at vertex $k \in Q_0$} is the quiver $\mu_k(Q)= (Q_0^*, Q_1^*, s^*, t^*)$ obtained as follows:
\begin{itemize}
\item all arrows incident with k are reversed (that is, $s^*(\alpha)= t(\alpha)$ and $t^*(\alpha)=s(\alpha)$ for all $\alpha \in Q_1$ incident with $k$);
\item whenever $i, j\in Q_0$ are such that there are $m >0$ arrows from $i$ to $k$ (in $Q$) and $n>0$ arrows from $k$ to $j$ (in $Q$), first add $mn$ arrows from $i$ to $j$. Then remove a maximal number of $2$-cycles. 
\end{itemize}

\end{Definition}

\begin{Definition}
Let $\Gamma$ be an undirected graph.
\begin{enumerate}
\item[(a)] A \defn{cycle} in $\Gamma$ is a subgraph which is isomorphic to the graph on vertex set $[n]$ whose edges are $(1,2), \ldots, (n-1, n), (n,1)$. 
\item[(b)] A full subgraph of $\Gamma$ which is a cycle is called \defn{chordless cycle}.
\item[(c)] The graph $\Gamma$ is called \defn{cyclically orientable} if it admits an orientation in which every chordless cycle of $\Gamma$ is cyclically oriented.
\end{enumerate}
\end{Definition}

\begin{example}
None of the graphs in Figure \ref{fig:CarterE7} is cyclically orientable. For instance, this can be easily seen by using a criterion provided by Gurvich \cite{Gur08}.
\end{example}

We divide the proof of Theorem \ref{thm:Main1} into three parts. Since the types $A_n$ and $B_n$ are closely related, we prove them together. Then we will prove Theorem \ref{thm:Main1} for the type $D_n$. Lastly we treat the exceptional types. 

\begin{Definition}
Let $\Gamma$ be a graph (possibly directed or with weighted edges). Let $v$ be a vertex of $\Gamma$. We define the \defn{valency} of $v$, denoted by $\val(v)$, to be the number of vertices of $\Gamma$ which are connected by a (directed or weighted) edge with $v$.

For example, in both graphs shown in Example \ref{ex:FromAtoB} we have $\val(v)=3$.
\end{Definition}

\begin{Proposition} \label{prop:ClusterIsCarterTypeA}
Theorem \ref{thm:Main1} is true for type $A_n$ (resp. type $B_n$). 

More precisley, let $Q$ be a quiver which is mutation-equivalent to an orientation of the Dynkin diagram of type $A_n$ (resp. $B_n$). Then the underlying undirected graph $\overline{Q}$ is a Carter diagram of type $A_n$ (resp. $B_n$).

Moreover, let $\Gamma$ be a Carter diagram of type $A_n$ (resp. $B_n$). Then there exists a quiver $Q$ which is mutation-equivalent to an orientation of the Dynkin diagram of type $A_n$ (resp. $B_n$) such that $\Gamma$ is isomorphic to $\overline{Q}$ if and only if $\Gamma$ is cyclically orientable.
\end{Proposition}

\begin{proof}
We begin with the $A_n$-case. By \cite{BV08} the muation class of quivers of type $A_n$ is given by the connected quivers on $n$ vertices such that the following properties hold:
\begin{itemize}
\item[(I)] All non-trivial cycles are oriented and of length $3$.
\item[(II)] A vertex has valency at most $4$.
\item[(III)] If a vertex has valency $4$, then two of its adjacent arrows belong to one $3$-cycle, the other two belong to another $3$-cycle.
\item[(IV)] If a vertex has valency $3$, then two of its adjacent arrows belong to a $3$-cycle, the third arrow does not belong to another cycle.
\end{itemize}
Let $Q$ be a quiver on $n$ vertices fulfilling the properties (I)-(IV) above. By Corollary \ref{cor:Kluitmann} we have to show that $\overline{Q}$ belongs to $\mathcal{A}^{n,n}$, that is $\overline{Q}$ fulfills the properties (i)-(iv) of Corollary \ref{cor:Kluitmann}.
The claim is obvious for $n=1$. So let $n>1$. By (I) we conclude that we can write $\Gamma$ as the union
$$
\Gamma = \Gamma_1 \cup \cdots \cup \Gamma_k ~(k \in\NN),
$$ 
where each $\Gamma_i$ is a complete graph on two or three vertices, $\Gamma_i \not\subseteq \Gamma_j$ for $i \neq j$ and $k$ is minimal (that is, for instance, we exclude the possibility that the union of three complete graphs on two vertices is a complete graph on three vertices).

Let $i \neq j$ be such that $\Gamma_i \cap \Gamma_j \neq \varnothing$ and let $v$ be a vertex in $\Gamma_i \cap \Gamma_j$. By the previous arguments and by (II) we have $2 \leq \val(v) \leq 4$. By distinguishing all possible cases for the value of $\val(v)$, we can show that $v$ is the only vertex in $\Gamma_i \cap \Gamma_j$, hence (ii) holds. We exhibit the case $\val(v)=4$. In this case, condition (III) implies that $\Gamma_i$ and $\Gamma_j$ are both $3$-cycles intersecting just in the vertex $v$. 

Next we want to show (iii). Let $v$ be a vertex of $\Gamma$. If $\val(v) \leq 2$, then (iii) holds obviously for $v$. If $\val(v)=3$ (resp. $\val(v)=4$) then condition (IV) (resp. (III)) implies that $v$ belongs to exactly two of the subgraphs $\Gamma_i$. Thus (iii) holds. 

Finally we have to show (iv). Since (iii) holds, we know that $\cap_{i \in I} \Gamma_i = \varnothing$ for $I \subseteq [k]$ with $|I| \geq 3$. By the inclusion-exclusion principle we conclude that 
\begin{align} \label{equ:TypeAProof}
n = | \Gamma | = \sum_{i=1}^k | \Gamma_i | - \sum_{\substack{\{i,j\} \subseteq [k]\\ i \neq j}} | \Gamma_i \cap \Gamma_j |.
\end{align}
If $k=1$, we are done. Therefore let $k >1$. Since $\Gamma$ is connected, the graph $\Gamma_1$ has non-trivial intersection with one of the graphs $\Gamma_i$ with $i > 1$. After possible renumbering we can assume that $|\Gamma_1 \cap \Gamma_2|=1$ by (ii). Again, if $k=2$, we are done. Therefore let $k>2$. Since $\Gamma$ is connected, one of the graphs $\Gamma_i$ with $i >2$ has non-trivial intersection with $\Gamma_1$ or $\Gamma_2$. Again, after possible renumbering, we can assume that $i=3$ and $|\Gamma_2 \cap \Gamma_3 |=1$ by (ii). Using the conditions (I)-(IV) and the minimality of $k$, it is straightforward to see that $\Gamma_1 \cap \Gamma_3 = \varnothing$. Proceeding in this manner we obtain that $|\Gamma_i \cap \Gamma_{i+1}|=1$ for $i \in [k-1]$, while $\Gamma_i \cap \Gamma_j = \varnothing$ for $j\neq i, i+1$. We leave the details to the reader. Hence 
$$
\sum_{\substack{\{i,j\} \subseteq [k]\\ i \neq j}} | \Gamma_i \cap \Gamma_j | = k-1
$$
and we conclude by (\ref{equ:TypeAProof}) that (iv) holds.

Now let $Q$ be a quiver which is mutation-equivalent to an orientation of the Dynkin diagram of type $B_n$. By \cite[Proposition 3.2]{NS14} and what we have shown above for the $A_n$-case, $\overline{Q}$ is given by one of the graphs in Figure \ref{fig:MutClassB}, where $\Gamma_1$ and $\Gamma_2$ are cyclically orientable Carter diagrams of type $A_{k_1}$ and $A_{k_2}$ for some $k_1, k_2 \geq 1$.

\begin{figure}[H]
\centering
\begin{tikzpicture}[thick,>=latex, double equal sign distance]
   \tikzset{triple/.style={double,postaction={draw,-}}}

	\coordinate(0) at (0,0);
  \node (1) at (2,0) [circle, draw, inner sep=0pt, minimum width=4pt]{};
	\node (2) at (3,0) [circle, draw, inner sep=0pt, minimum width=4pt]{};
  \coordinate(3) at (2.3,0);
  \node(b1) at (1,0) []{$\Gamma_1$};

	\draw[bend right=65] (0) to (3);
  \draw[bend left=65] (0) to (3);
  \draw[double] (1) to (2);

  \coordinate(01) at (4,0);
  \node (11) at (6,0) [circle, draw, inner sep=0pt, minimum width=4pt]{};
	\node (21) at (7,-0.7) [circle, draw, inner sep=0pt, minimum width=4pt]{};
  \coordinate(31) at (6.3,0);
  \node(b11) at (5,0) []{$\Gamma_1$};

	\draw[bend right=65] (01) to (31);
  \draw[bend left=65] (01) to (31);
  \draw[double] (11) to (21);

  \coordinate(011) at (4,-1.4);
  \node (111) at (6,-1.4) [circle, draw, inner sep=0pt, minimum width=4pt]{};
  \coordinate(311) at (6.3,-1.4);
  \node(b111) at (5,-1.4) []{$\Gamma_2$};

	\draw[bend right=65] (011) to (311);
  \draw[bend left=65] (011) to (311);
  \draw[double] (111) to (21);
  \draw[-] (11) to (111);

\end{tikzpicture}
\caption{Quiver-mutation class $B_n$.} \label{fig:MutClassB}
\end{figure}
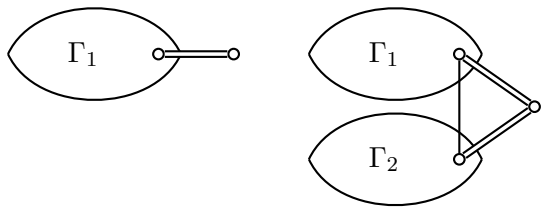
It follows immediately from our description of type $B_n$ Carter diagrams that $\overline{Q}$ is a Carter diagram of type $B_n$.

\medskip
To show the other assertion first note, that if $Q$ is a quiver which is mutation-equivalent to an orientation of the Dynkin diagram of type $A_n$ (resp. $B_n$), then in particular $Q$ admits a cyclic orientation. By what we have shown above, $\overline{Q}$ is a cyclically orientable Carter diagram.

For the other direction let us begin with a Carter diagram $\Gamma$ of type $A_n$, that is $\Gamma$ lies in $\mathcal{A}^{n,n}$ by Theorem \ref{thm:Kluitmann}, and assume that $\Gamma$ is cyclically orientable. We choose a cyclic orientation of $\Gamma$ and show that the conditions (I)-(IV) hold.

The complete graph on four vertices is not cyclically orientable. In particular, the complete graph on $\ell \geq 4$ vertices is not cyclically orientable. Therefore condition (i) from Theorem \ref{thm:Kluitmann} yields that $\Gamma$ is the union of complete graphs on two or three vertices. Hence (I) holds. Conditions (ii) and (iii) imply (II), while conditions (ii)-(iv) imply (III) and (IV). We leave the details again to the reader. 

Now let $\Gamma$ be a Carter diagram of type $B_n$ which is cyclically orientable. By our description of the type $B_n$ Carter diagrams, there is an unique vertex $v$ such that all edges adjacent with $v$ have weight $2$ and such that $\Gamma \setminus \{v\}$ is connected. Let us denote by $\Gamma_0$ the diagram obtained from $\Gamma$ by replacing edges of weight two with edges of weight one. By our description of the type $B_n$ Carter diagrams, the graph $\Gamma_0$ is a cyclically orientable Carter diagram of type $A_n$. In particular, $\Gamma$ is cyclically orientable. 

\end{proof}


\medskip
$\phantom{4}$

Before we show that Theorem \ref{thm:Main1} holds for type $D_n$, let us recall the description of the quiver mutation-class of type $D_n$ given by Vatne.

\begin{Lemma}[{\cite[Theorem 3.1]{Vat10}}] \label{lem:MutClassD}
For $n \geq 4$, a quiver $Q$ is mutation equivalent to an orientation of the Dynkin diagram of type $D_n$ if and only if $Q$ is one of the types (D1)-(D4) shown in Figure \ref{fig:MutClassD}
\end{Lemma}

\begin{figure}[H]
\centering
\begin{tikzpicture}[thick,>=latex]
  \node (N1) at (-0.35,0.5) []{(D1)};
	\coordinate(0) at (0,0);
  \node (1) at (2,0) [circle, draw, inner sep=0pt, minimum width=4pt]{};
	\node (2) at (2.5,0.5) [circle, draw, inner sep=0pt, minimum width=4pt]{};
  \node (2b) at (2.5,-0.5) [circle, draw, inner sep=0pt, minimum width=4pt]{};
  \coordinate(3) at (2.3,0);
  \node(b1) at (0.8,0) []{$\Gamma_1$};
  \node(b1D1) at (1.75,0) []{$v$};

	\draw[bend right=65] (0) to (3);
  \draw[bend left=65] (0) to (3);
  \draw[-] (1) to (2);
  \draw[-] (1) to (2b);

  \node (N2) at (5.65,0.5) []{(D2)};
  \coordinate(01) at (6,0);
  \node (11) at (8,0) [circle, draw, inner sep=0pt, minimum width=4pt]{};
  \coordinate(31) at (8.3,0);
  \node(b11) at (6.8,0) []{$\Gamma_1$};
  \node(b11D2) at (7.65,0) []{$v_1$};

	\draw[bend right=65] (01) to (31);
  \draw[bend left=65] (01) to (31);
  
    \coordinate(011) at (8.7,0);
  \node (111) at (9,0) [circle, draw, inner sep=0pt, minimum width=4pt]{};
  \coordinate(311) at (11,0);
  \node(b111) at (10.15,0) []{$\Gamma_2$};
  \node(b111D2) at (9.4,0) []{$v_2$};
  \node (411) at (8.5,0.5) [circle, draw, inner sep=0pt, minimum width=4pt]{};
  \node (511) at (8.5,-0.5) [circle, draw, inner sep=0pt, minimum width=4pt]{};
  
  \draw[<-] (11) to (411);
  \draw[->] (111) to (411);
  \draw[->] (111) to (511);
  \draw[->] (511) to (11);
  \draw[<-] (111) to (11);

	\draw[bend right=65] (011) to (311);
  \draw[bend left=65] (011) to (311);

  \node (N4) at (5.65,-1.5) []{(D4)};
  \coordinate(B01) at (6,-2);
  \node (B11) at (8,-2) [circle, draw, inner sep=0pt, minimum width=4pt]{};
  \coordinate(B31) at (8.3,-2);
  \node(Bb11) at (7,-2) []{$\Gamma_{ij}$};

	\draw[bend right=65] (B01) to (B31);
  \draw[bend left=65] (B01) to (B31);
  
  \coordinate(B011) at (9.2,-2);
  \node (B111) at (9.5,-2) [circle, draw, inner sep=0pt, minimum width=4pt]{};
  \coordinate(B311) at (11.5,-2);
  \node (B411) at (8.75,-2) [circle, draw, inner sep=0pt, minimum width=4pt]{};
  
  \node (B511) at (8,-2.65) [circle, draw, inner sep=0pt, minimum width=4pt]{};
  \node (B611) at (9.5,-2.65) [circle, draw, inner sep=0pt, minimum width=4pt]{};
  
  \draw[<-] (B11) to (B411);
  \draw[<-] (B411) to (B111);
  \draw[<-] (B111) to (B611);
  \draw[->] (B511) to (B411);
  \draw[<-] (B611) to (B411);
  \draw[->] (B11) to (B511);

	\draw[bend right=65] (B011) to (B311);
  \draw[bend left=65] (B011) to (B311);
  
  \coordinate(B711) at (8.75,-3.5);
  
  \draw[dotted, bend right=45] (B511) to (B711);
  \draw[dotted, bend left=45] (B611) to (B711);

  \node (N3) at (-0.35,-1.5) []{(D3)};
  \coordinate(aB01) at (0,-2);
  \node (aB11) at (2,-2) [circle, draw, inner sep=0pt, minimum width=4pt]{};
  \coordinate(aB31) at (2.3,-2);
  \node(aBb11) at (0.8,-2) []{$\Gamma_1$};
  \node(aBb11D3) at (1.7,-2) []{$v_1$};

	\draw[bend right=65] (aB01) to (aB31);
  \draw[bend left=65] (aB01) to (aB31);
  
  \coordinate(aB011) at (2.7,-2);
  \node (aB111) at (3,-2) [circle, draw, inner sep=0pt, minimum width=4pt]{};
  \coordinate(aB311) at (5,-2);
  \node(aBb111) at (4.3,-2) []{$\Gamma_2$};
  \node(aBb111D3) at (3.37,-2) []{$v_2$};
  \node (aB411) at (2.5,-1.5) [circle, draw, inner sep=0pt, minimum width=4pt]{};
  \node (aB511) at (2.5,-2.5) [circle, draw, inner sep=0pt, minimum width=4pt]{};
  
  \draw[<-] (aB11) to (aB411);
  \draw[<-] (aB411) to (aB111);
  \draw[<-] (aB111) to (aB511);
  \draw[<-] (aB511) to (aB11);

	\draw[bend right=65] (aB011) to (aB311);
  \draw[bend left=65] (aB011) to (aB311);

\end{tikzpicture}
\caption{Quiver-mutation class $D_n$.} \label{fig:MutClassD}
\end{figure}
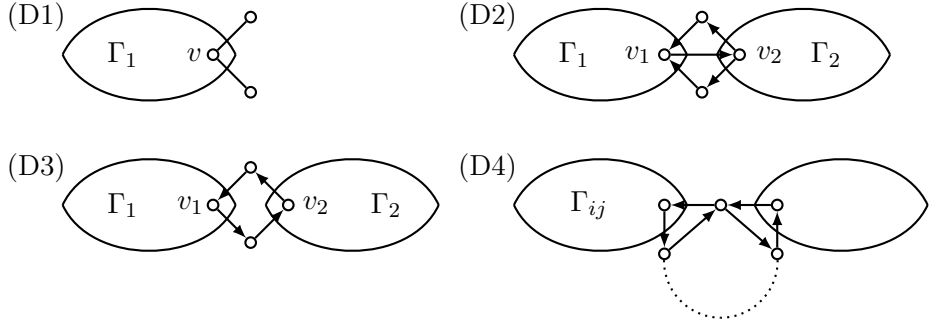

Let us make this description more precise (see \cite[Chapter 2]{Vat10}). Therefore denote by $\mathcal{M}_{A_n}$ the type $A_n$ mutation-class described by the properties (I)-(IV) in the proof of Proposition \ref{prop:ClusterIsCarterTypeA}. Call a vertex $v$ of a quiver $Q$ a \defn{connecting vertex} if $v$ has valency at most $2$ and, moreover, if $v$ has valency $2$, then $v$ is a vertex in a $3$-cycle in $Q$.
\begin{itemize}
\item[(D1):] $\Gamma_1$ is in $\mathcal{M}_{A_{n-2}}$ and $v$ is a connecting vertex for $\Gamma_1$
\item[(D2):] $\Gamma_1$ (resp. $\Gamma_2$) is in $\mathcal{M}_{A_{n_1}}$ (resp. $\mathcal{M}_{A_{n_2}}$) for some $n_1 \in \NN$ (resp. $n_2 \in \NN$) and $v_1$ (resp. $v_2$) is a connecting vertex for $\Gamma_1$ (resp. $\Gamma_2$).
\item[(D3):] See (D2).
\item[(D4):] The quiver $Q$ described by this type has a full subquiver which is a directed $k$-cycle ($k \geq 3$), called \defn{central cycle}. For each arrow $\alpha: i \rightarrow j$ in $Q$, there may (and may not) be a vertex $c_{ij}$ which is not in the central cycle, such that there is an oriented $3$-cycle $i \stackrel{\alpha}{\rightarrow} j \rightarrow v_{ij} \rightarrow i$. This $3$-cycle is a full subquiver. It is called \defn{spike}. There are no more arrows starting or ending in vertices on the central cycle. To each spike $i \stackrel{\alpha}{\rightarrow} j \rightarrow v_{ij} \rightarrow i$ there is a quiver $\Gamma_{ij}$ from $\mathcal{M}_{A_{n_{ij}}}$ attached, for some $n_{ij} \in \NN$. The vertex $v_{ij}$ is a connecting vertex for $\Gamma_{ij}$.
\end{itemize}
Note that in types (D2)-(D4), the subquivers $\Gamma_i$ resp. $\Gamma_{ij}$ might be in $\mathcal{M}_{A_1}$.

\begin{Proposition} \label{prop:ClusterIsCarterTypeD}
Theorem \ref{thm:Main1} is true for type $D_n$. 
\end{Proposition}

\begin{proof}
Let us first recall that by Theorem \ref{thm:CarterTypeD}, all Carter diagrams of type $D_n$ are given by the set $\mathcal{A}^{n-1,n}$. By Theorem \ref{thm:Kluitmann} the shape of a diagram $\Gamma$ in $\mathcal{A}^{n-1,n}$ depends on the choice of an integer $m'$, which is either equal to $n-1$ or to $n$. If $m'=n-1$, we call $\Gamma$ to be of type (D.I). In this type the diagram $\Gamma$ is given by a Carter diagram of type $A_{n-1}$ to which we attach a ``duplicated'' vertex. If $m'=n$, we call $\Gamma$ to be of type (D.II). 

\medskip
Let us start with a quiver $Q$ which is mutation-equivalent to an orientation of the Dynkin diagram of type $D_n$. We have to show that the underlying undirected graph $\overline{Q}$ is a Carter diagram of type $D_n$. We do this by showing that all four possible types of underlying undirected graphs in the mutation-class given by Lemma \ref{lem:MutClassD} can be realized by a Carter diagram of type $D_n$.

Let $Q$ be of type (D1). Then $\overline{Q}$ is given by the following picture. 
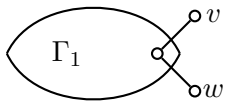
\begin{figure}[H]
\centering
\begin{tikzpicture}[thick,>=latex]
%
	\coordinate(0) at (0,0);
  \node (1) at (2,0) [circle, draw, inner sep=0pt, minimum width=4pt]{};
	\node (2) at (2.5,0.5) [circle, draw, inner sep=0pt, minimum width=4pt]{};
  \node (2b) at (2.5,-0.5) [circle, draw, inner sep=0pt, minimum width=4pt]{};
  \coordinate(3) at (2.3,0);
  \node(b1) at (0.8,0) []{$\Gamma_1$};
  \node(b1D1) at (2.75,0.5) []{$v$};
  \node(b2D1) at (2.75,-0.5) []{$w$};

	\draw[bend right=65] (0) to (3);
  \draw[bend left=65] (0) to (3);
  \draw[-] (1) to (2);
  \draw[.] (1) to (2b);

\end{tikzpicture}
\end{figure}

By Proposition \ref{prop:ClusterIsCarterTypeA}, the graph $\Gamma_1$ is a Carter diagram of type $A_{n-2}$. In particular, the full subgraph $\Gamma_1 \cup \{v \}$ is a Carter diagram of type $A_{n-1}$. Duplication of the vertex $v$ yields the above graph, which is therefore of type (D.I). 

Let $Q$ be of type (D2) or (D3). Then $\overline{Q}$ is given by the following picture (for type (D2) the vertices $v_1$ and $v_2$ are not connected, while for type (D3) they are).

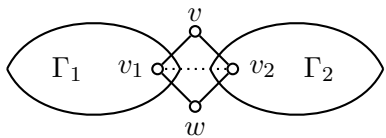
\begin{figure}[H]
\centering
\begin{tikzpicture}[thick,>=latex]

  \coordinate(01) at (6,0);
  \node (11) at (8,0) [circle, draw, inner sep=0pt, minimum width=4pt]{};
  \coordinate(31) at (8.3,0);
  \node(b11) at (6.8,0) []{$\Gamma_1$};
  \node(b11D2) at (7.65,0) []{$v_1$};

	\draw[bend right=65] (01) to (31);
  \draw[bend left=65] (01) to (31);
  
    \coordinate(011) at (8.7,0);
  \node (111) at (9,0) [circle, draw, inner sep=0pt, minimum width=4pt]{};
  \coordinate(311) at (11,0);
  \node(b111) at (10.15,0) []{$\Gamma_2$};
  \node(b111D2) at (9.4,0) []{$v_2$};
  \node (411) at (8.5,0.5) [circle, draw, inner sep=0pt, minimum width=4pt]{};
  \node (411DD) at (8.5,0.72) []{$v$};
  \node (511) at (8.5,-0.5) [circle, draw, inner sep=0pt, minimum width=4pt]{};
  \node (411DDD) at (8.5,-0.8) []{$w$};
  
  \draw[-] (11) to (411);
  \draw[-] (111) to (411);
  \draw[-] (111) to (511);
  \draw[-] (511) to (11);
  \draw[dotted] (111) to (11);

	\draw[bend right=65] (011) to (311);
  \draw[bend left=65] (011) to (311);

\end{tikzpicture}
\end{figure}
By Proposition \ref{prop:ClusterIsCarterTypeA}, the full subgraph $\Gamma_1 \cup \Gamma_2 \cup \{ v \}$ is a Carter diagram of type $A_{n-1}$. Duplication of the vertex $v$ yields the above graph, which is therefore again of type (D.I).

Let $Q$ be of type (D4). We want to show that $\overline{Q}$ is of type (D.II). By \cite[Proposition 5 (iii) a)]{Klu88} any connected graph on $n$ vertices that is the union of complete graphs which are arranged in the form of a circle with some side branches, is a graph of type (D.II). In particular, if we assume that all of these complete graphs are complete graphs on two or three vertices, eventually we will end up with the graph $\overline{Q}$. 

\medskip
Finally, let $\Gamma$ be a Carter diagram of type $D_n$ which is cyclically orientable. We have to show that $\Gamma$ is ismorphic to $\overline{Q}$ for some quiver $Q$ which is mutation-equivalent to an orientation of the Dynkin diagram of type $D_n$. 

Let us first assume that $\Gamma$ is of type (D.I), that is $\Gamma = \Gamma' \cup \{ w \}$, where $\Gamma'$ is a Carter diagram of type $A_{n-1}$ and $w$ is the ``duplication'' of a vertex $v \in \Gamma'$. In particular, $\Gamma'$ is cyclically orientable and therefore $\val(v) \leq 4$. If $\val(v)=1$, then it is easy to see that $\Gamma$ is isomorphic to $\overline{Q}$ for some quiver $Q$ of type (D1). Similarly, if $\val(v)=2$, then $\Gamma$ is isomorphic to $\overline{Q}$ for some quiver $Q$ of type (D2) or (D3). Let $\val(v)=3$. All three vertices $v_1, v_2, v_3$ adjacent to $v$ have to be vertices of $\Gamma'$. Since $\Gamma'$ is cyclically orientable of type $A_{n-1}$, property (IV) from the proof of Proposition \ref{prop:ClusterIsCarterTypeA} holds. Let us assume that we have the following situation in $\Gamma'$:

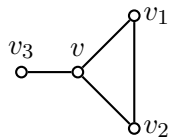
\begin{figure}[H]
\centering
\begin{tikzpicture}[scale=0.75, thick,>=latex]

  \node (11) at (3,0) [circle, draw, inner sep=0pt, minimum width=4pt]{};
  \node (a11) at (3.4,0) []{$v_1$};
  \node (21) at (1,-1) [circle, draw, inner sep=0pt, minimum width=4pt]{};
  \node (a21) at (1,-0.6) []{$v_3$};
  \node (31) at (3,-2) [circle, draw, inner sep=0pt, minimum width=4pt]{};
  \node (a31) at (3.4,-2) []{$v_2$};
  \node (41) at (2,-1) [circle, draw, inner sep=0pt, minimum width=4pt]{};
  \node (a41) at (2,-0.6) []{$v$};

  \draw[-] (21) to (41);
  \draw[-] (31) to (41);
  \draw[-] (41) to (11);
  \draw[-] (31) to (11);

\end{tikzpicture}
\end{figure}
Duplication of $v$ yields the following full subgraph of $\Gamma$:
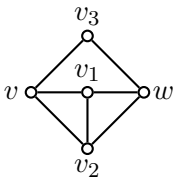
\begin{figure}[H]
\centering
\begin{tikzpicture}[scale=0.75, thick,>=latex]

  \node (11) at (3,-1) [circle, draw, inner sep=0pt, minimum width=4pt]{};
  \node (a11) at (3,-0.65) []{$v_1$};
  \node (21) at (3,0) [circle, draw, inner sep=0pt, minimum width=4pt]{};
  \node (a21) at (3,0.35) []{$v_3$};
  \node (31) at (3,-2) [circle, draw, inner sep=0pt, minimum width=4pt]{};
  \node (a31) at (3,-2.35) []{$v_2$};
  \node (41) at (2,-1) [circle, draw, inner sep=0pt, minimum width=4pt]{};
  \node (a41) at (1.65,-1) []{$v$};
  \node (51) at (4,-1) [circle, draw, inner sep=0pt, minimum width=4pt]{};
  \node (a51) at (4.35,-1) []{$w$};

  \draw[-] (21) to (41);
  \draw[-] (31) to (41);
  \draw[-] (41) to (11);
  \draw[-] (31) to (11);
  
  \draw[-] (21) to (51);
  \draw[-] (31) to (51);
  \draw[-] (51) to (11);

\end{tikzpicture}
\end{figure}
But this full subgraph is not cyclically orientable, hence $\Gamma$ cannot be cyclically orientable. Therefore the case that $v$ has valency $3$ does not occur. Similarly, we can show that the case that $v$ has valency $4$ does not occur.

Let us assume that $\Gamma$ is of type (D.II). By \cite[Proposition 5 (iii) a)]{Klu88} the graph $\Gamma$ is the union of complete graphs which are arranged in the form of a circle with some side branches. Each of these complete graphs has to be a complete graph on $2$ or $3$ vertices, since complete graphs on $\ell \geq 4$ vertices are not cyclically orientable and $\Gamma$ is cyclically orientable. Therefore $\Gamma$ is isomorphic to $\overline{Q}$ for some quiver $Q$ of type (D4).
\end{proof}

\begin{Proposition} \label{prop:ClusterIsCarterTypeX}
Theorem \ref{thm:Main1} is true for the exceptional types $E_6$, $E_7$, $E_8$, $F_4$ and $G_2$.


\end{Proposition}

\begin{proof}
This is done by inspection of our lists of Carter diagrams of exceptional type $X_n$ and the mutation classes of Dynkin quivers of type $X_n$. The latter ones can be computed using Keller's quiver mutation applet \cite{Kel}.
\end{proof}

\medskip
We have given a complete classification of Carter diagrams of Dynkin types. Thereafter we have seen that the cyclically orientable Carter diagrams are precisely the underlying undirected graphs of quivers which appear in the seeds of finite type cluster algebras. This leads to the following question:
\begin{question}
Let $\Phi$ be a crystallographic root system. Given a set of linearly independent roots $\{\beta_1, \ldots , \beta_m \} \subseteq \Phi$ and let $\Gamma$ be its associated Carter diagram. Is there a criterion to determine whether $\Gamma$ is cyclically orientable purely in terms of the root system?
\end{question}

\section{Presentations of Reflection Groups} \label{sec:Presentations}
In \cite{BM15} it is shown that each quiver which is mutation-equivalent to an orientation of a Dynkin diagram encodes a natural presentation of the corresponding finite Coxeter group (like the Dynkin diagram does; see Section \ref{sec:CoxeterI}).

We have seen in the previous chapters that the cyclically orientable Carter diagrams exactly provide the underlying graphs of quivers which are mutation-equivalent to a Dynkin diagram. The aim of this section is to show that all Carter diagrams, even those which are not cyclically orientable or simply-laced (that is, of type $A_n, D_n, E_6, E_7$ or $E_8$), provide a natural presentation (as given by Barot--Marsh) of the corresponding finite Coxeter group.

Let us begin with the observation that Theorem \ref{thm:Main2} holds for admissible Carter diagrams.
\begin{Proposition} \label{prop:Main2Admissible}
Let $\Phi$ be a crystallographic root system and let $\Gamma$ be an admissible Carter diagram of the same Dynkin type as $\Phi$. Then $W(\Gamma)$ is isomorphic to the Weyl group $W_{\Phi}$.
\end{Proposition}

\begin{proof}
For the cyclically orientable Carter diagrams, the assertion is a consequence of Theorem \ref{thm:Main1} and \cite[Theorem A]{BM15}. Inspection of Carter's list of admissible diagrams in \cite{Car72} yields that all of them are cyclically orientable except for $E_7(a_4), E_8(a_7)$ and $E_8(a_8)$. For these three cases it has been checked using GAP \cite{GAP2017} that the assertion holds. 
\end{proof}

\medskip
Before we start to explain the idea and to carry out the proof of Theorem \ref{thm:Main2}, let us remark the following fact about the relations (R3) given in Section \ref{sec:intro}.

\begin{Proposition}[{\cite[Lemma 4.1, Proposition 4.6]{BM15}}]
For any chordless cycle $C$ in a Carter diagram $\Gamma$, all relations of type (R3) attached to $C$ are equivalent to one fixed relation of type (R3) attached to $C$ (in the presence of the relations (R1) and (R2)). In particular, relation (R3) does neither depend on the choice of the vertex $i_0$ in the cycle $C$ nor on the direction of the cycle.
\end{Proposition}

Let us make this statement more precise by considering the following example.
\begin{example}
Consider the following cycle.
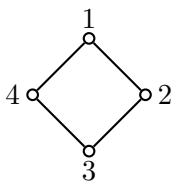
\begin{figure}[H]
\centering
\begin{tikzpicture}[scale=0.75, thick,>=latex]
%

  \node (21) at (3,0) [circle, draw, inner sep=0pt, minimum width=4pt]{};
  \node (a21) at (3,0.35) []{$1$};
  \node (31) at (3,-2) [circle, draw, inner sep=0pt, minimum width=4pt]{};
  \node (a31) at (3,-2.35) []{$3$};
  \node (41) at (2,-1) [circle, draw, inner sep=0pt, minimum width=4pt]{};
  \node (a41) at (1.65,-1) []{$4$};
  \node (51) at (4,-1) [circle, draw, inner sep=0pt, minimum width=4pt]{};
  \node (a51) at (4.35,-1) []{$2$};

  \draw[-] (21) to (41);
  \draw[-] (31) to (41);
  
  \draw[-] (21) to (51);
  \draw[-] (31) to (51);

%
\end{tikzpicture}
\end{figure}
Attach to it the relations (R1) and (R2) and the relation 
$$
(t_1t_2t_3t_4t_3t_2)^2=1
$$
of type (R3) attached to the following numbering and direction of the cycle:
$$
1 \edge 2 \edge 3 \edge 4 \edge 1.
$$
Using realtions (R1) and (R2), we obtain:
\begin{align*}
t_1t_2t_3t_4t_3t_2 & = t_2t_3t_4t_3t_2t_1\\
\Leftrightarrow ~ t_1t_2t_4t_3t_4t_2 & = t_2t_4t_3t_4t_2t_1\\
\Leftrightarrow ~ t_1t_4t_2t_3t_2t_4 & = t_2t_4t_3t_4t_2t_1\\
\Leftrightarrow ~ t_1t_4t_3t_2t_3t_4 & = t_2t_4t_3t_4t_2t_1.
\end{align*}
Conjugation by $t_4$ and the relation $t_2t_4 = t_4t_2$ yield 
\begin{align*}
t_4t_1t_4t_3t_2t_3  = t_2t_3t_2t_4t_1t_4.
\end{align*}
Conjugation by $t_3$ and the relation $t_3t_2t_3t_2 = t_2t_3$ yield 
\begin{align*}
t_3t_4t_1t_4t_3t_2  = t_2t_3t_4t_1t_4t_3.
\end{align*}
Finally, conjuagation by $t_2$ yields
\begin{align*}
t_2t_3t_4t_1t_4t_3  = t_3t_4t_1t_4t_3t_2,
\end{align*}
that is, the relation (R3) attached to the following numbering and direction of the cycle:
$$
2 \edge 3 \edge 4 \edge 1 \edge 2.
$$
Similarly, starting with same relation of type (R3) as above and using the relations of types (R1) and (R2), we obtain:
\begin{align*}
t_1t_2t_3t_4t_3t_2 & = t_2t_3t_4t_3t_2t_1\\
\Leftrightarrow ~ t_1t_2t_4t_3t_4t_2 & = t_2t_4t_3t_4t_2t_1\\
\Leftrightarrow ~ t_1t_4t_2t_3t_2t_4 & = t_4t_2t_3t_2t_4t_1\\
\Leftrightarrow ~ t_1t_4t_3t_2t_3t_4 & = t_4t_3t_2t_3t_4t_1,
\end{align*}
that is, the relation (R3) attached to the following numbering and direction of the cycle:
$$
1 \edge 4 \edge 3 \edge 2 \edge 1.
$$
\end{example}

\medskip
The Hurwitz action will play an important role in our following arguments. Let us therefore give two examples which illustrate this fact.

\begin{example}
Consider a Coxeter system $(W,S)$ of type $D_6$ and let $(t_1,\ldots, t_6)$ be a tuple of reflections such that the corresponding Carter diagram is given by the left diagram in Figure \ref{fig:D6}. Note that this is in fact a Carter diagram of type $D_6$ by Theorem \ref{thm:CarterTypeD}. We apply the Hurwitz move 
$$
(t_1, \ldots , t_6) ~\stackrel{\sigma_3}{\sim} ~ (t_1,t_2,t_3t_4t_3,t_3,t_5,t_6)=:(r_1, \ldots , r_6).
$$
The resulting diagram is given by the diagram on the right side in Figure \ref{fig:D6}.
\begin{figure}[H]
\centering
\begin{tikzpicture}[scale=1.05, thick,>=latex]
    \node (A1) at (-3,1) [circle, draw, inner sep=0pt, minimum width=4pt]{};
    \node (A11) at (-3,1.3) [] {$t_1$};
    
    \node (B1) at (-3,0) [circle, draw, inner sep=0pt, minimum width=4pt]{};
    \node (B11) at (-3,-0.3) [] {$t_2$};
    
    \node (C1) at (-2,1) [circle, draw, inner sep=0pt, minimum width=4pt]{};
    \node (C11) at (-2,1.3) [] {$t_3$};
    
    \node (D1) at (-2,0) [circle, draw, inner sep=0pt, minimum width=4pt]{};
    \node (D11) at (-2,-0.3) [] {$t_4$};
    
    \node (E1) at (-1,1) [circle, draw, inner sep=0pt, minimum width=4pt]{};
    \node (E11) at (-1,1.3) [] {$t_5$};
    
    \node (F1) at (-1,0) [circle, draw, inner sep=0pt, minimum width=4pt]{};
    \node (F11) at (-1,-0.3) [] {$t_6$};

    \node (1) at (-0.5,0.5) []{};
    \node (2) at (0.5,0.5) []{};
    \node (3) at (0,0.65) []{$\sigma_3$};

    \node (A2) at (3,1) [circle, draw, inner sep=0pt, minimum width=4pt]{};
    \node (A22) at (3,1.3) [] {$r_5$};
    
    \node (B2) at (3,0) [circle, draw, inner sep=0pt, minimum width=4pt]{};
    \node (B22) at (3,-0.3) [] {$r_6$};
    
    \node (C2) at (2,1) [circle, draw, inner sep=0pt, minimum width=4pt]{};
    \node (C22) at (2,1.3) [] {$r_3$};
    
    \node (D2) at (2,0) [circle, draw, inner sep=0pt, minimum width=4pt]{};
    \node (D22) at (2,-0.3) [] {$r_4$};
    
    \node (E2) at (1,1) [circle, draw, inner sep=0pt, minimum width=4pt]{};
    \node (E22) at (1,1.3) [] {$r_1$};
    
    \node (F2) at (1,0) [circle, draw, inner sep=0pt, minimum width=4pt]{};
    \node (F22) at (1,-0.3) [] {$r_2$};

    \draw[-] (A1) to (B1);
    \draw[-] (A1) to (C1);
    \draw[-] (B1) to (C1);
    \draw[-] (C1) to (D1);
    \draw[-] (C1) to (E1);
    \draw[-] (D1) to (F1);   
    \draw[-] (E1) to (F1);

    \draw[-] (A2) to (B2);
    \draw[-] (A2) to (C2);
    \draw[-] (B2) to (D2);
    \draw[-] (C2) to (D2);
    \draw[-] (C2) to (E2);
    \draw[-] (C2) to (F2);
    \draw[-] (D2) to (E2);
    \draw[-] (C2) to (B2);
    \draw[-] (D2) to (F2);
    \draw[-] (E2) to (F2);
    
    \draw[->] (1) to (2);

\end{tikzpicture}
\caption{Hurwitz move applied to a Carter diagram} \label{fig:D6}
\end{figure}
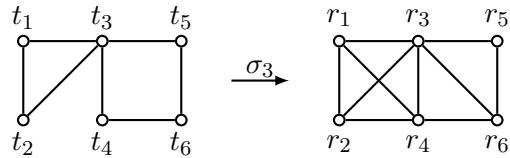
Note that the diagram on the left side is cyclically orientable, while the diagram on the right side is not. Therefore the arguments of Barot--Marsh just yield a presentation attached to the diagram on the left.

Let $\langle t_1, \ldots , t_6 \mid R \rangle$ (resp. $\langle r_1, \ldots , r_6 \mid R' \rangle$) be the presentation attached to the diagram on the left side (resp. to the diagram on the right side) as described in Section \ref{sec:intro} and consider the map

\begin{align*}
\varphi: \langle t_1, \ldots , t_6 \mid R \rangle & \rightarrow \langle r_1, \ldots , r_6 \mid R' \rangle \\
t_j & \mapsto \begin{cases} 
r_j &\mbox{if } j \neq 3,4 \\
r_4 &\mbox{if } j = 3 \\
r_4r_3r_4 &\mbox{if } j=4.
\end{cases}
\end{align*}
The relations $R$ are preserved under the map $\varphi$. We exhibit some examples:
\begin{itemize}
\item We have the relation $(t_2t_4)^2=1$ in $R$, but also 
$$
(\varphi(t_2) \varphi(t_4))^2= (r_1r_4r_3r_4)^2 = 1,
$$
since we have the chordless cycle $r_1 \edge r_4 \edge r_3 \edge r_1$.
\item We have the relation $(t_4t_3t_5t_6t_5t_3)^2=1$ in $R$, which is induced by the cycle of length $4$. But also 
\begin{align*}
(\varphi(t_4) \varphi(t_3) \varphi(t_5) \varphi(t_6) \varphi(t_5) \varphi(t_3))^2 & = 
(r_4r_3r_4r_4r_5r_6r_5r_4)^2\\
{} & = (r_4r_3r_5r_6r_5r_4)^2\\
{} & = r_4(r_3r_5r_6r_5)^2 r_4 = r_4^2=1,
\end{align*}
where we used in the last line the relation in $R'$ given by the cycle $r_3 \edge r_5 \edge r_6 \edge r_3$. 
\end{itemize}
By \cite[Ch. 4, Proposition 3]{Joh90} we obtain that the map $\varphi$ extends to a surjective group homomorphism. Likewise we see that 
\begin{align*}
\psi: \langle r_1, \ldots , r_6 \mid R' \rangle & \rightarrow \langle t_1, \ldots , t_6 \mid R \rangle \\
r_j & \mapsto \begin{cases} 
t_j &\mbox{if } j \neq 3,4 \\
t_3t_4t_3 &\mbox{if } j =3 \\
t_3 &\mbox{if } j=4
\end{cases}
\end{align*}
extends to a surjective group homomorphism. Since $\psi(\varphi(t_j))=t_j$ and $\varphi(\psi(r_j))=r_j$ for all $j \in \{1, \ldots , 6\}$, we obtain 
$$
\langle t_1, \ldots , t_6 \mid R \rangle \cong \langle r_1, \ldots , r_6 \mid R' \rangle.
$$
\end{example}

\medskip
\begin{example} \label{ex:CarterTypeB}
We consider a root system $\Phi$ of type $B_3$ and use the description of $W_{\Phi}$ as group of signed permutations $S_3^B$ given in Section \ref{sec:CarterDiagB}. Consider 
$$
(t_1,t_2,t_3):= ((1,-1), (1,2)(-1,-2),  (1,3)(-1,-3)).
$$
This is a reduced reflection factorization of a Coxeter element. We have 
\begin{align*}
\sigma_1 ((1,-1), (1,2)(-1,-2),  (1,3)(-1,-3)) & = ((1,-2)(-1,2), (1,-1), (1,3)(-1,-3))\\
{} & =:(r_1, r_2, r_3).
\end{align*}
The effect of this Hurwitz move on the Carter diagram is as follows.
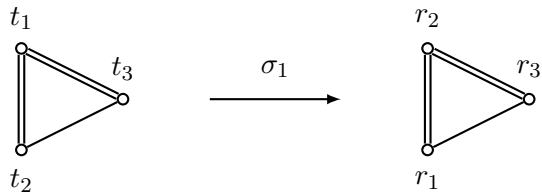
\begin{figure}[H]
\centering
\begin{tikzpicture}[scale=1.35, thick,>=latex, double equal sign distance]
   \tikzset{triple/.style={double,postaction={draw,-}}}

  \node (a1) at (2,0.5) [circle, draw, inner sep=0pt, minimum width=4pt]{};
  \node (a11) at (2,0.8) []{$t_{1}$};
  \node (a2) at (2,-0.5) [circle, draw, inner sep=0pt, minimum width=4pt]{};
  \node (a21) at (2,-0.8) []{$t_{2}$};
  \node (a3) at (3,0) [circle, draw, inner sep=0pt, minimum width=4pt]{};
  \node (a31) at (3,0.3) []{$t_3$};

  \draw[double] (a1) to (a2);
  \draw[-] (a2) to (a3);
  \draw[double] (a1) to (a3);
  
  \node (A) at (3.75,0) []{};
  \node (AB) at (4.5,0.3) []{$\sigma_1$};
  \node (B) at (5.25,0) []{};

  \draw[->] (A) to (B);

  \node (A1) at (6,0.5) [circle, draw, inner sep=0pt, minimum width=4pt]{};
  \node (A11) at (6,0.8) []{$r_{2}$};
  \node (A2) at (6,-0.5) [circle, draw, inner sep=0pt, minimum width=4pt]{};
  \node (A21) at (6,-0.8) []{$r_{1}$};
  \node (A3) at (7,0) [circle, draw, inner sep=0pt, minimum width=4pt]{};
  \node (A31) at (7,0.3) []{$r_3$};

  \draw[double] (A1) to (A2);
  \draw[double] (A1) to (A3);
  \draw[-] (A2) to (A3);

\end{tikzpicture}
\end{figure}
That is, we obtain the same diagram. This comes from the fact that the distinguished vertex $t_1= (1,-1)=r_2$ has not changed under the Hurwitz move $\sigma_1$. 

On the other hand, we have 
\begin{align*}
\sigma_1^{-1} ((1,-1), (1,2)(-1,-2),  (1,3)(-1,-3)) & = ((1,2)(-1,-2), (2,-2), (1,3)(-1,-3))\\
{} & =:(r_1', r_2', r_3').
\end{align*}
The effect of this Hurwitz move on the Carter diagram is different since we change the distinguished vertex from $(1,-1)$ to $(2,-2)$.
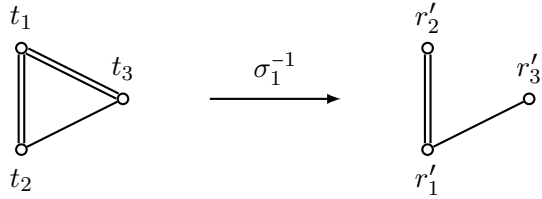
\begin{figure}[H]
\centering
\begin{tikzpicture}[scale=1.35, thick,>=latex, double equal sign distance]
   \tikzset{triple/.style={double,postaction={draw,-}}}

  \node (a1) at (2,0.5) [circle, draw, inner sep=0pt, minimum width=4pt]{};
  \node (a11) at (2,0.8) []{$t_{1}$};
  \node (a2) at (2,-0.5) [circle, draw, inner sep=0pt, minimum width=4pt]{};
  \node (a21) at (2,-0.8) []{$t_{2}$};
  \node (a3) at (3,0) [circle, draw, inner sep=0pt, minimum width=4pt]{};
  \node (a31) at (3,0.3) []{$t_3$};

  \draw[double] (a1) to (a2);
  \draw[-] (a2) to (a3);
  \draw[double] (a1) to (a3);
  
  \node (A) at (3.75,0) []{};
  \node (AB) at (4.5,0.3) []{$\sigma_1^{-1}$};
  \node (B) at (5.25,0) []{};

  \draw[->] (A) to (B);

  \node (A1) at (6,0.5) [circle, draw, inner sep=0pt, minimum width=4pt]{};
  \node (A11) at (6,0.8) []{$r_{2}'$};
  \node (A2) at (6,-0.5) [circle, draw, inner sep=0pt, minimum width=4pt]{};
  \node (A21) at (6,-0.8) []{$r_{1}'$};
  \node (A3) at (7,0) [circle, draw, inner sep=0pt, minimum width=4pt]{};
  \node (A31) at (7,0.3) []{$r_3'$};

  \draw[double] (A1) to (A2);
  \draw[-] (A2) to (A3);

\end{tikzpicture}
\end{figure}
\end{example}

\medskip
$\phantom{4}$

Let $W_{\Phi}$ be a Weyl group, $w \in W_{\Phi}$ and $(t_1, \ldots , t_m) \in \Red_T(w)$ with Carter diagram $\Gamma$. For every braid $\sigma \in \mathcal{B}_m$, the Hurwitz action yields a new reduced reflection factorization $\sigma(t_1, \ldots , t_m) \in \Red_T(w)$ with Carter diagram $\Gamma'$. We put $\sigma(\Gamma):=\Gamma'$ 

\begin{remark} \label{rem:RelationsSimplyLaced}
Let $\Phi$ be a simply-laced root system (that is, of type $A_n, D_n, E_6, E_7$ or $E_8$). Let $w \in W_{\Phi}$ and $(t_1, \ldots , t_m) \in \Red_T(w)$ with associated Carter diagram $\Gamma$. We can describe the impact of an elementary Hurwitz move on $\Gamma$, that is, how to obtain $\sigma_i(\Gamma)$ from $\Gamma$ as follows:

Let $\sigma_i(t_1, \ldots , t_m) = (r_1, \ldots , r_m)$. Hence we have $t_j = r_j$ for $j \in [m] \setminus \{i, i+1 \}$, $r_i = t_i t_{i+1} t_i$ and $r_{i+1}= t_i$. If $(t_i t_{i+1})^2=1$, then $\Gamma$ and $\sigma_i(\Gamma)$ are identical. Therefore let us assume that $(t_i t_{i+1})^2 \neq 1$. The diagram $\sigma_i(\Gamma)$ is obtained as follows:
\begin{itemize}
\item The vertices of $\sigma_i(\Gamma)$ correspond to $r_1, \ldots , r_m$.
\item For $j,k \in [m] \setminus \{i, i+1 \}$, there is an edge between $r_j$ and $r_k$ if and only if there is an edge between $t_j$ and $t_k$.
\item For $j \in [m] \setminus \{i, i+1 \}$, there is an edge between $r_j$ and $r_{i+1}$ if and only if there is an edge between $t_j$ and $t_i$.
\item There is an edge between $r_i$ and $r_{i+1}$.
\item For $j \in [m] \setminus \{i, i+1 \}$, there is an edge between $r_j$ and $r_i$ if and only if either $t_j$ and $t_i$ or $t_j$ and $t_{i+1}$ (but not both) are connected by an edge in $\Gamma$. 
\end{itemize}
\end{remark}

\begin{proof}
Most of the assertions can be checked directly. Let us assume that $(t_i t_{i+1})^2 \neq 1$. We only show that if both $t_j$ and $t_i$ as well as $t_j$ and $t_{i+1}$ are connected by an edge, then $r_j$ and $r_i$ are not connected by an edge. Since $\Phi$ is simply-laced, we assume that $(\beta \mid \beta)=2$ for all $\beta \in \Phi$. Put $t_k = s_{\beta_k}$ with $\beta_k \in \Phi$ ($k \in \{i, i+1,j\}$). The case 
$$
(\beta_i \mid \beta_j) = (\beta_{i+1} \mid \beta_j) = ( \beta_i \mid \beta_{i+1}) = -1
$$
can not occur. Otherwise we would have $s_{\beta_i}(\beta_{i+1}) = \beta_i + \beta_{i+1} \in \Phi$ and 
$$
(\beta_j \mid \beta_i + \beta_{i+1}) = (\beta_j \mid \beta_i) + (\beta_j \mid \beta_{i+1}) = -2,
$$
which is not possible. If we have 
$$
(\beta_i \mid \beta_j) = (\beta_{i+1} \mid \beta_j) = ( \beta_i \mid \beta_{i+1}) = 1,
$$
then $s_{\beta_i}(\beta_{i+1})= \beta_{i+1} - \beta_i$ and 
$$
(\beta_j \mid \beta_i - \beta_{i+1}) = (\beta_j \mid \beta_i) - (\beta_j \mid \beta_{i+1}) = 0.
$$
In particular, $s_{\beta_j}$ and $s_{\beta_i}s_{\beta_{i+1}}s_{\beta_i}$ commute as desired. Therefore let us assume that 
$$
(\beta_i \mid \beta_j) = 1, ~(\beta_{i+1} \mid \beta_j) = ( \beta_i \mid \beta_{i+1}) = -1.
$$
We have $s_{\beta_i}(\beta_{i+1})= \beta_{i+1} + \beta_i$ and 
$$
(\beta_j \mid \beta_i + \beta_{i+1}) = (\beta_j \mid \beta_i) + (\beta_j \mid \beta_{i+1}) = 0.
$$
Again, $s_{\beta_j}$ and $s_{\beta_i}s_{\beta_{i+1}}s_{\beta_i}$ commute as desired. If 
$$
(\beta_i \mid \beta_j) =(\beta_{i+1} \mid \beta_j)=1,~ ( \beta_i \mid \beta_{i+1}) = -1
$$
We still have $s_{\beta_i}(\beta_{i+1})= \beta_{i+1} + \beta_i \in \Phi$, but 
$$
(\beta_j \mid \beta_i + \beta_{i+1}) = (\beta_j \mid \beta_i) + (\beta_j \mid \beta_{i+1}) = 2.
$$
This is not possible since $\beta_j \neq \beta_i + \beta_{i+1}$ by Carter's Lemma \ref{thm:CartersLemma}. All other cases are analogous.
\end{proof}

\medskip
In the $B_n$-case, the Hurwitz action on Carter diagrams might behave different, see also Example \ref{ex:CarterTypeB}.

\begin{remark} \label{rem:RelationsNonSimplyLaced}
Let $\Phi$ be a root system of type $B_n$. Let $w \in W_{\Phi}$ and $(t_1, \ldots , t_m) \in \Red_T(w)$ with associated Carter diagram $\Gamma$. As in Remark \ref{rem:RelationsSimplyLaced} we describe how to obtain $\sigma_i(\Gamma)$ from $\Gamma$:

Let $\sigma_i(t_1, \ldots , t_m) = (r_1, \ldots , r_m)$. Hence we have $t_j = r_j$ for $j \in [m] \setminus \{i, i+1 \}$, $r_i = t_i t_{i+1} t_i$ and $r_{i+1}= t_i$. If $(t_i t_{i+1})^2=1$, then $\Gamma$ and $\sigma_i(\Gamma)$ are identical. Therefore let us assume that $(t_i t_{i+1})^2 \neq 1$. The diagram $\sigma_i(\Gamma)$ is obtained as follows:

\begin{enumerate} \label{rem:HurwitzCarterTypeB}
\item If $t_i$ does not correspond to the distinguished vertex of $\Gamma$, then the diagram $\sigma_i(\Gamma)$ is obtained from $\Gamma$ as decribed in Remark \ref{rem:RelationsSimplyLaced}. We just have to take care of the labeling. If $t_{i+1}$ is the distinguished vertex in $\Gamma$, then $r_i$ is the distinguished vertex in $\sigma_i(\Gamma)$. If $t_{j}$ is the distinguished vertex in $\Gamma$ for some $j \in [n] \setminus \{i, i+1 \}$, then $r_j$ is the distinguished vertex in $\sigma_i(\Gamma)$.
\item If $t_i$ is the distinguished vertex, then we obtain $\sigma_i(\Gamma)$ as follows:
\begin{itemize}
\item The vertices of $\sigma_i(\Gamma)$ correspond to $r_1, \ldots , r_m$.
\item For $j,k \in [m] \setminus \{i, i+1 \}$, there is an edge between $r_j$ and $r_k$ if and only if there is an edge between $t_j$ and $t_k$.
\item For $j \in [m] \setminus \{i, i+1 \}$, there is an edge between $r_j$ and $r_{i+1}$ if and only if there is an edge between $t_j$ and $t_i$.
\item There is an edge between $r_i$ and $r_{i+1}$.
\item For $j \in [m] \setminus \{i, i+1 \}$, there is an edge between $r_j$ and $r_i$ if and only if there is an edge between $t_j$ and $t_{i+1}$: Since $t_i$ is the distinguished vertex, we have 
$$
\supp(t_{i+1})= \supp(t_i t_{i+1} t_i)= \supp(r_i),
$$
and since $(t_1, \ldots , t_m)$ is reduced, we have $t_j \neq t_{i+1}$ for all $j \neq i+1$. 
\item The distinguished vertex of $\sigma_i(\Gamma)$ is $r_{i+1}$.
\end{itemize}
In particular, we see that $\Gamma$ and $\sigma_i(\Gamma)$ are isomorphic.
\end{enumerate}
\end{remark}

As a direct consequence of this remark we obtain:
\begin{Proposition} \label{prop:CaseDistVertex}
If $\Phi$ is of type $B_n$ and $t_i$ is the distinguished vertex of $\Gamma$, then $ W(\Gamma) \cong W(\sigma_i(\Gamma))$.
\end{Proposition}

\medskip
Before we begin with the proof of Theorem \ref{thm:Main2}, let us emphasize again that the simply-laced case is already covered by the work of Cameron, Seidel and Tsaranov \cite{CST94}. Therefore we just sketch how one may proceed in this case. The case that the Carter diagram is cyclically orientable is covered by the work of Barot and Marsh \cite{BM15}. Our proof will fill in the missing gaps and will prove the assertion in full generality for all Carter diagrams.
\begin{Lemma} \label{lem:SurjHomB1}
Let $\Gamma$ be a Carter diagram of simply-laced type or type $B_n$ with vertex set corresponding to $t_1, \ldots , t_m$ and $i \in [m-1]$. If $(t_i t_{i+1})^2 \neq 1$ (and, in the $B_n$-case, if $t_i$ is not the distinguished vertex of $\Gamma$), then the map 
$$
\varphi: W(\Gamma) \rightarrow W(\sigma_i(\Gamma)), ~ t_j \mapsto \begin{cases} 
r_j &\mbox{if } j \neq i, i+1 \\
r_{i+1} &\mbox{if } j=i\\
r_{i+1} r_{i} r_{i+1} &\mbox{if } j=i+1
\end{cases}
$$
extends to a (surjective) group homomorphism.
\end{Lemma}

\medskip

\begin{proof}
We argue with the ``substitution test'' \cite[Ch. 4, Proposition 3]{Joh90}. We therefore investigate all possible constellations in $\Gamma$ and how these might change under the action of $\sigma_i$. We show that the relations in $W(\Gamma)$ are preserved under the map $\varphi$. We denote by $W(\sigma_i(\Gamma))$ the group with generators $r_1, \ldots r_m$ subject to the relations induced by the Carter diagram $\sigma_i(\Gamma)$. 

By assumption, the vertices corresponding to $t_i$ and $t_{i+1}$ will be connected by an edge in $\Gamma$. Depending on whether the weight is $2$ or not, we have $(t_i t_{i+1})^3=1$ or $(t_i t_{i+1})^4=1$. By Remark \ref{rem:RelationsSimplyLaced} and Remark \ref{rem:RelationsNonSimplyLaced} the relation $(\varphi(t_i)\varphi(t_{i+1}))^3=1$ or $(\varphi(t_i)\varphi(t_{i+1}))^4=1$ will always hold. Also by these remarks it is enough to just consider those relations which actually involve $t_i$ and $t_{i+1}$. Therefore one would have to consider all possible constellations of $t_i$ and $t_{i+1}$ in $\Gamma$. To give an idea for the proof of the simply-laced case, we carry out two possible cases.

1) The vertex $t_{i}$ is part of a cycle, while $t_{i+1}$ is not part of that cycle.

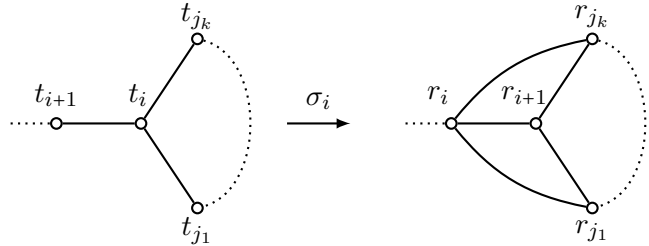
\begin{figure}[H]
\centering
\begin{tikzpicture}[scale=0.75, thick,>=latex]

  \node (a1) at (-0.5,0) []{};
  \node (a2) at (0.5,0) [circle, draw, inner sep=0pt, minimum width=4pt]{};
  \node (a21) at (0.5,0.5) []{$t_{i+1}$};
  \node (a3) at (2,0) [circle, draw, inner sep=0pt, minimum width=4pt]{};
  \node (a31) at (1.95,0.5) []{$t_{i}$};
  \node (a4) at (3,1.5) [circle, draw, inner sep=0pt, minimum width=4pt]{};
  \node (a41) at (3,1.9) []{$t_{j_k}$};
  \node (a5) at (3,-1.5) [circle, draw, inner sep=0pt, minimum width=4pt]{};
  \node (a51) at (3,-1.9) []{$t_{j_1}$};

  \draw[dotted] (a1) to (a2);
  \draw[-] (a2) to (a3);
  \draw[-] (a3) to (a4);
  \draw[-] (a3) to (a5);
  \draw[dotted, bend left=75] (a4) to (a5);

  \node (A) at (4.4,0) []{};
  \node (AB) at (5.15,0.4) []{$\sigma_i$};
  \node (B) at (5.9,0) []{};

  \draw[->] (A) to (B);

  \node (A1) at (6.5,0) []{};
  \node (A2) at (7.5,0) [circle, draw, inner sep=0pt, minimum width=4pt]{};
  \node (A21) at (7.27,0.5) []{$r_{i}$};
  \node (A3) at (9,0) [circle, draw, inner sep=0pt, minimum width=4pt]{};
  \node (A31) at (8.8,0.5) []{$r_{i+1}$};
  \node (A4) at (10,1.5) [circle, draw, inner sep=0pt, minimum width=4pt]{};
  \node (A41) at (10,1.9) []{$r_{j_k}$};
  \node (A5) at (10,-1.5) [circle, draw, inner sep=0pt, minimum width=4pt]{};
  \node (A51) at (10,-1.9) []{$r_{j_1}$};

  \draw[dotted] (A1) to (A2);
  \draw[-] (A2) to (A3);
  \draw[-] (A3) to (A4);
  \draw[-] (A3) to (A5);
  \draw[dotted, bend left=75] (A4) to (A5);
  \draw[-, bend left=20] (A2) to (A4);
  \draw[-, bend right=20] (A2) to (A5);

\end{tikzpicture}
\end{figure}

\begin{itemize}
\item $(t_{i+1} t_{j_1})^2=1$ $\rightarrow$ $(\varphi(t_{i+1}) \varphi(t_{j_1}))^2= (r_{i+1} r_i r_{i+1} r_{j_1})^2=1$, since $r_i, r_{i+1}$ and $r_{j_1}$ are the vertices of a $3$-cycle. Analogously we can argue for $(t_{i+1} t_{j_k})^2=1$. 
\item $(t_{i}t_{j_1})^3=1$ $\rightarrow$ $(\varphi(t_{i}) \varphi(t_{j_1}))^3= (r_{i+1} r_{j_1})^3=1$. Analogously we can argue for $(t_{i} t_{j_k})^3=1$. 
\item $(t_i t_{j_1} \cdots t_{j_{k-1}} t_{j_k} t_{j_{k-1}} \cdots t_{j_1})^2=1$. Applying $\varphi$ yields 
\begin{align*}
(r_{i+1} r_{j_1} \cdots r_{j_{k-1}} r_{j_k} r_{j_{k-1}} \cdots r_{j_1})^2=1,
\end{align*}
which is exactly the realtion of type (R3) for the ``big'' cycle in $\sigma_i(\Gamma)$. 
\end{itemize}

2) $t_i$ and $t_{i+1}$ are both vertices of a full subgraph which is an $\ell$-cycle  ($\ell \geq 4$).

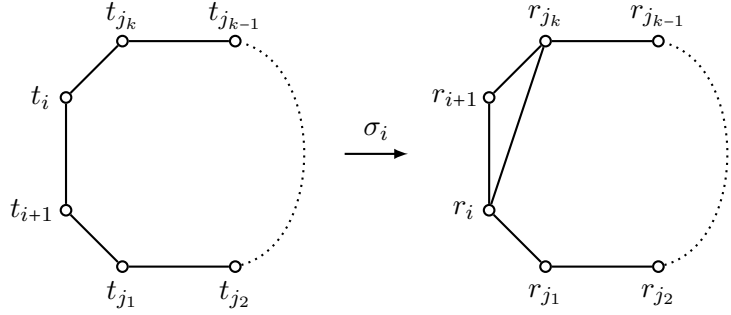
\begin{figure}[H]
\centering
\begin{tikzpicture}[scale=0.75, thick,>=latex]

  \node (a1) at (0,0) [circle, draw, inner sep=0pt, minimum width=4pt]{};
  \node (a11) at (-0.45,0) []{$t_{i}$};
  \node (a2) at (0,-2) [circle, draw, inner sep=0pt, minimum width=4pt]{};
  \node (a21) at (-0.6,-2) []{$t_{i+1}$};
  \node (a3) at (1,1) [circle, draw, inner sep=0pt, minimum width=4pt]{};
  \node (a31) at (1,1.45) []{$t_{j_k}$};
  \node (a4) at (1,-3) [circle, draw, inner sep=0pt, minimum width=4pt]{};
  \node (a41) at (1,-3.5) []{$t_{j_1}$};
  \node (a5) at (3,1) [circle, draw, inner sep=0pt, minimum width=4pt]{};
  \node (a51) at (3,1.45) []{$t_{j_{k-1}}$};
  \node (a6) at (3,-3) [circle, draw, inner sep=0pt, minimum width=4pt]{};
  \node (a61) at (3,-3.5) []{$t_{j_2}$};

  \draw[-] (a1) to (a2);
  \draw[-] (a2) to (a4);
  \draw[-] (a4) to (a6);
  \draw[-] (a1) to (a3);
  \draw[-] (a3) to (a5);
  \draw[dotted, bend right=75] (a6) to (a5);
  
  \node (A) at (4.75,-1) []{};
  \node (AB) at (5.5,-0.6) []{$\sigma_i$};
  \node (B) at (6.25,-1) []{};

  \draw[->] (A) to (B);

  \node (A1) at (7.5,0) [circle, draw, inner sep=0pt, minimum width=4pt]{};
  \node (A11) at (6.9,0) []{$r_{i+1}$};
  \node (A2) at (7.5,-2) [circle, draw, inner sep=0pt, minimum width=4pt]{};
  \node (A21) at (7.05,-2) []{$r_{i}$};
  \node (A3) at (8.5,1) [circle, draw, inner sep=0pt, minimum width=4pt]{};
  \node (A31) at (8.5,1.45) []{$r_{j_k}$};
  \node (A4) at (8.5,-3) [circle, draw, inner sep=0pt, minimum width=4pt]{};
  \node (A41) at (8.5,-3.5) []{$r_{j_1}$};
  \node (A5) at (10.5,1) [circle, draw, inner sep=0pt, minimum width=4pt]{};
  \node (A51) at (10.5,1.45) []{$r_{j_{k-1}}$};
  \node (A6) at (10.5,-3) [circle, draw, inner sep=0pt, minimum width=4pt]{};
  \node (A61) at (10.5,-3.5) []{$r_{j_2}$};

  \draw[-] (A1) to (A2);
  \draw[-] (A2) to (A4);
  \draw[-] (A2) to (A3);
  \draw[-] (A4) to (A6);
  \draw[-] (A1) to (A3);
  \draw[-] (A3) to (A5);
  \draw[dotted, bend right=75] (A6) to (A5);

\end{tikzpicture}
\end{figure}

\begin{itemize}
\item $(t_{i} t_{j_k})^3=1$ $\rightarrow$ $(\varphi(t_{i}) \varphi(t_{j_k}))^3= (r_{i+1}  r_{j_k})^3=1$, since $r_{i+1}$ and $r_{j_k}$ are connected by an edge. 
\item $(t_{i}t_{j_1})^2=1$ $\rightarrow$ $(\varphi(t_{i}) \varphi(t_{j_1}))^2= (r_{i+1} r_{j_1})^2=1$, since $r_{i+1}$ and $r_{j_1}$ are not connected by an edge. 
\item $(t_{i+1} t_{j_k})^2=1$ $\rightarrow$ $(\varphi(t_{i+1}) \varphi(t_{j_k}))^2=( r_{i+1} r_i r_{i+1} r_{j_k})^2=1$, since the vertices are part of a $3$-cycle.
\item $(t_{i+1} t_{j_1})^3=1$ $\rightarrow$ $(\varphi(t_{i+1}) \varphi(t_{j_1}))^3 =  (r_{i+1} r_i r_{i+1} r_{j_1})^3 =(r_{i+1} r_i r_{j_1} r_{i+1})^3 =$ \linebreak $r_{i+1} (r_i r_{j_1})^3 r_{i+1} =1$. 
\item $(t_i t_{i+1} t_{j_1} \cdots t_{j_{k-1}} t_{j_k} t_{j_{k-1}} \cdots t_{j_1} t_{i+1})^2=1$. Applying $\varphi$ yields 
\begin{align*}
\phantom{r_{i+1}r_i}(r_{i+1} r_{i+1} r_i r_{i+1} r_{j_1} \cdots r_{j_{k-1}} r_{j_k} r_{j_{k-1}} \cdots r_{j_1}  r_{i+1} r_i r_{i+1})^2 = & (r_i r_{i+1} r_{j_1} \cdots r_{j_{k-1}} r_{j_k} r_{j_{k-1}} \cdots r_{j_1} r_i r_{i+1} r_i)^2\\
{} = & r_i r_{i+1} (r_{j_1} \cdots r_{j_{k-1}} r_{j_k} r_{j_{k-1}} \cdots r_{j_1} r_i )^2 r_{i+1} r_i\\
{} = & r_i r_{i+1} r_{i+1} r_i =1,
\end{align*}
since $(r_{j_1} \cdots r_{j_{k-1}} r_{j_k} r_{j_{k-1}} \cdots r_{j_1} r_i )^2 =1$ is equivalent to $(r_i r_{j_1} \cdots r_{j_{k-1}} r_{j_k} r_{j_{k-1}} \cdots r_{j_1})^2 =1$, which is precisely the relation induced by the ``big'' cycle. 
\end{itemize}

\medskip
In the following we will consider all possible constellations in type $B_n$.

1) The vertices $t_i$ and $t_{i+1}$ are part of a full subgraph which is a line. We excluded the case that $t_i$ is the distinguished vertex. If $t_{i+1}$ is the distinguished vertex, then it has to be of valency one, because otherwise $\Gamma \setminus \{ t_i \}$ would be disconnected. 

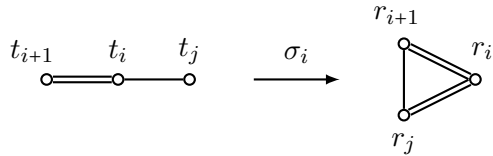
\begin{figure}[H]
\centering
\begin{tikzpicture}[scale=0.95, thick,>=latex, double equal sign distance]
   \tikzset{triple/.style={double,postaction={draw,-}}}

  \node (11) at (1,0) [circle, draw, inner sep=0pt, minimum width=4pt]{};
  \node (a11) at (0.8,0.4) []{$t_{i+1}$};
  \node (21) at (2,0) [circle, draw, inner sep=0pt, minimum width=4pt]{};
  \node (a21) at (2,0.4) []{$t_i$};
  \node (31) at (3,0) [circle, draw, inner sep=0pt, minimum width=4pt]{};
  \node (a31) at (3,0.4) []{$t_{j}$};

  \draw[double] (11) to (21);
  \draw[-] (21) to (31);
  
  \node (A) at (3.75,0) []{};
  \node (AB) at (4.5,0.35) []{$\sigma_i$};
  \node (B) at (5.25,0) []{};

  \draw[->] (A) to (B);

  \node (A1) at (6,0.5) [circle, draw, inner sep=0pt, minimum width=4pt]{};
  \node (A11) at (5.9,0.9) []{$r_{i+1}$};
  \node (A2) at (6,-0.5) [circle, draw, inner sep=0pt, minimum width=4pt]{};
  \node (A21) at (6,-0.9) []{$r_j$};
  \node (A3) at (7,0) [circle, draw, inner sep=0pt, minimum width=4pt]{};
  \node (A31) at (7.1,0.4) []{$r_i$};

  \draw[-] (A1) to (A2);
  \draw[double] (A1) to (A3);
  \draw[double] (A2) to (A3);

\end{tikzpicture}
\end{figure}

\begin{itemize}
\item $(t_i t_j)^3=1$ $\rightarrow$ $(\varphi(t_i) \varphi(t_j))^3= (r_{i+1}r_j)^3=1$,
\item $(t_{i+1}t_j)^2=1$ $\rightarrow$ $(\varphi(t_{i+1}) \varphi(t_j))^2= (r_{i+1}r_ir_{i+1} r_j)^2$. But $(r_{i+1}r_ir_{i+1} r_j)^2=1$ is equivalent to $(r_jr_{i+1}r_ir_{i+1})^2=1$ and this relation holds because of the cycle in $\sigma_i(\Gamma)$. 
\end{itemize}

By the description of Carter diagrams of type $B_n$ in Section \ref{sec:CarterDiagB}, we know that if $\Gamma$ contains a chordless cycle, then it has to be a $3$-cycle.

\medskip

2) The vertices $t_i$ and $t_{i+1}$ are part of a full subgraph which is a $3$-cycle. The vertex $t_{i+1}$ is the distinguished vertex.

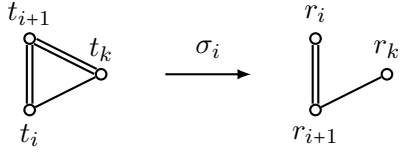
\begin{figure}[H]
\centering
\begin{tikzpicture}[scale=0.95, thick,>=latex, double equal sign distance]
   \tikzset{triple/.style={double,postaction={draw,-}}}

  \node (a1) at (2,0.5) [circle, draw, inner sep=0pt, minimum width=4pt]{};
  \node (a11) at (2,0.85) []{$t_{i+1}$};
  \node (a2) at (2,-0.5) [circle, draw, inner sep=0pt, minimum width=4pt]{};
  \node (a21) at (2,-0.85) []{$t_{i}$};
  \node (a3) at (3,0) [circle, draw, inner sep=0pt, minimum width=4pt]{};
  \node (a31) at (3,0.35) []{$t_{k}$};

  \draw[double] (a1) to (a2);
  \draw[-] (a2) to (a3);
  \draw[double] (a1) to (a3);
  
  \node (A) at (3.75,0) []{};
  \node (AB) at (4.5,0.35) []{$\sigma_i$};
  \node (B) at (5.25,0) []{};

  \draw[->] (A) to (B);

  \node (A1) at (6,0.5) [circle, draw, inner sep=0pt, minimum width=4pt]{};
  \node (A11) at (6,0.85) []{$r_i$};
  \node (A2) at (6,-0.5) [circle, draw, inner sep=0pt, minimum width=4pt]{};
  \node (A21) at (6,-0.85) []{$r_{i+1}$};
  \node (A3) at (7,0) [circle, draw, inner sep=0pt, minimum width=4pt]{};
  \node (A31) at (7,0.35) []{$r_k$};

  \draw[double] (A1) to (A2);
  \draw[-] (A2) to (A3);

\end{tikzpicture}
\end{figure}

\begin{itemize}
\item $(t_i t_{k})^3=1$ $\rightarrow$ $(\varphi(t_i) \varphi(t_k))^3= (r_{i+1}r_k)^3=1$. 
\item $(t_{i+1}t_k)^4=1$ $\rightarrow$ applying the map $\varphi$ yields:
\begin{align*}
(\varphi(t_{i+1}) \varphi(t_k))^4 & = (r_{i+1} r_i r_{i+1} r_k)^4\\
{} & = (r_{i+1}r_i \underbrace{r_{i+1} r_k)(r_{i+1}}_{=r_k r_{i+1}r_k}r_i \underbrace{r_{i+1} r_k)(r_{i+1}}_{=r_k r_{i+1}r_k}r_i \underbrace{r_{i+1} r_k)(r_{i+1}}_{=r_k r_{i+1}r_k}r_i r_{i+1} r_k))\\
{} & = r_{i+1} r_i r_k r_{i+1} \underbrace{r_k r_i r_k}_{=r_i} r_{i+1} \underbrace{r_k r_i r_k}_{=r_i} r_{i+1} r_k r_i r_{i+1} r_k\\
{} & = r_{i+1} r_i r_k r_{i+1} r_i r_{i+1} r_i r_{i+1} r_k r_i r_{i+1} r_k\\
{} & = r_{i+1} r_k \underbrace{r_i r_{i+1} r_i r_{i+1} r_i r_{i+1} r_i}_{r_{i+1}} r_k r_{i+1} r_k\\
{} & = (r_{i+1} r_k)^3=1.
\end{align*}
\item $(t_{i} t_k t_{i+1} t_k)^2=1$ $\rightarrow$ applying the map $\varphi$ yields:
\begin{align*}
(\varphi(t_{i}) \varphi(t_k) \varphi(t_{i+1}) \varphi(t_{k}))^2 & = (r_{i+1} r_k r_{i+1} r_i r_{i+1} r_k)^2\\
{} & = (r_{i+1} r_k r_{i+1} r_i \underbrace{r_{i+1} r_k)(r_{i+1} r_k r_{i+1}}_{=r_k} r_i r_{i+1} r_k)\\
{} & = r_{i+1} r_k r_{i+1} \underbrace{r_i r_k r_i}_{=r_k} r_{i+1} r_k\\
{} & = (r_{i+1} r_k)^3 = 1.
\end{align*}
\end{itemize}

3) The vertex $t_{i+1}$ is part of a full subgraph which is a $3$-cycle and which contains the distinguished vertex $t_j$, while $t_i$ is not a vertex of that $3$-cycle.

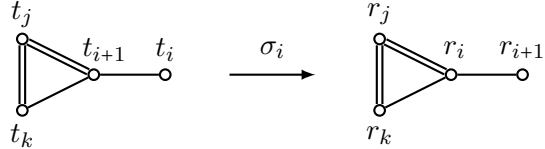
\begin{figure}[H]
\centering
\begin{tikzpicture}[scale=0.95, thick,>=latex, double equal sign distance]
   \tikzset{triple/.style={double,postaction={draw,-}}}

  \node (a1) at (1,0.5) [circle, draw, inner sep=0pt, minimum width=4pt]{};
  \node (a11) at (1,0.85) []{$t_{j}$};
  \node (a2) at (1,-0.5) [circle, draw, inner sep=0pt, minimum width=4pt]{};
  \node (a21) at (1,-0.85) []{$t_{k}$};
  \node (a3) at (2,0) [circle, draw, inner sep=0pt, minimum width=4pt]{};
  \node (a31) at (2.15,0.35) []{$t_{i+1}$};
  \node (a4) at (3,0) [circle, draw, inner sep=0pt, minimum width=4pt]{};
  \node (a41) at (3,0.35) []{$t_{i}$};

  \draw[double] (a1) to (a2);
  \draw[-] (a2) to (a3);
  \draw[double] (a1) to (a3);
  \draw[-] (a3) to (a4);
  
  \node (A) at (3.75,0) []{};
  \node (AB) at (4.5,0.35) []{$\sigma_i$};
  \node (B) at (5.25,0) []{};

  \draw[->] (A) to (B);

  \node (A1) at (6,0.5) [circle, draw, inner sep=0pt, minimum width=4pt]{};
  \node (A11) at (6,0.85) []{$r_j$};
  \node (A2) at (6,-0.5) [circle, draw, inner sep=0pt, minimum width=4pt]{};
  \node (A21) at (6,-0.85) []{$r_{k}$};
  \node (A3) at (7,0) [circle, draw, inner sep=0pt, minimum width=4pt]{};
  \node (A31) at (7.05,0.35) []{$r_i$};
  \node (A4) at (8,0) [circle, draw, inner sep=0pt, minimum width=4pt]{};
  \node (A41) at (8,0.35) []{$r_{i+1}$};

  \draw[double] (A1) to (A2);
  \draw[-] (A2) to (A3);
  \draw[double] (A1) to (A3);
  \draw[-] (A3) to (A4);

\end{tikzpicture}
\end{figure}

\begin{itemize}
\item $(t_i t_j)^2=1$ $\rightarrow$ $(\varphi(t_i) \varphi(t_j))^2= (r_{i+1}r_j)^2=1$,
\item $(t_{i+1}t_j)^4=1$ $\rightarrow$ $(\varphi(t_{i+1}) \varphi(t_j))^4= (r_i r_j)^4=1$.
\item $(t_{i+1}t_kt_jt_k)^2$ $\rightarrow$ $(\varphi(t_{i+1})\varphi(t_k)\varphi(t_j) \varphi(t_k))^2 = (r_{i+1}r_i r_{i+1}r_kr_jr_k)^2 = r_{i+1}(r_i r_kr_jr_k)^2 r_{i+1}=1$, where $(r_i r_kr_jr_k)^2=1$ holds because of the cycle in $\sigma_i(\Gamma)$.
\end{itemize}

4) The vertex $t_{i}$ is part of a full subgraph which is a $3$-cycle and which contains the distinguished vertex $t_j$, while $t_{i+1}$ is not a vertex of that $3$-cycle.

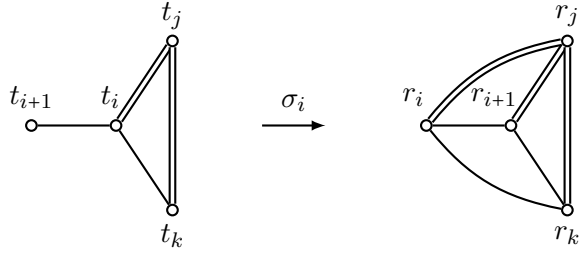
\begin{figure}[H]
\centering
\begin{tikzpicture}[scale=0.75, thick,>=latex, double equal sign distance]
   \tikzset{triple/.style={double,postaction={draw,-}}}

  \node (a2) at (0.5,0) [circle, draw, inner sep=0pt, minimum width=4pt]{};
  \node (a21) at (0.5,0.5) []{$t_{i+1}$};
  \node (a3) at (2,0) [circle, draw, inner sep=0pt, minimum width=4pt]{};
  \node (a31) at (1.9,0.5) []{$t_{i}$};
  \node (a4) at (3,1.5) [circle, draw, inner sep=0pt, minimum width=4pt]{};
  \node (a41) at (3,1.95) []{$t_{j}$};
  \node (a5) at (3,-1.5) [circle, draw, inner sep=0pt, minimum width=4pt]{};
  \node (a51) at (3,-1.95) []{$t_{k}$};

  \draw[-] (a2) to (a3);
  \draw[double] (a3) to (a4);
  \draw[-] (a3) to (a5);
  \draw[double] (a4) to (a5);

  \node (A) at (4.4,0) []{};
  \node (AB) at (5.15,0.4) []{$\sigma_i$};
  \node (B) at (5.9,0) []{};

  \draw[->] (A) to (B);

  \node (A2) at (7.5,0) [circle, draw, inner sep=0pt, minimum width=4pt]{};
  \node (A21) at (7.27,0.5) []{$r_{i}$};
  \node (A3) at (9,0) [circle, draw, inner sep=0pt, minimum width=4pt]{};
  \node (A31) at (8.7,0.5) []{$r_{i+1}$};
  \node (A4) at (10,1.5) [circle, draw, inner sep=0pt, minimum width=4pt]{};
  \node (A41) at (10,1.95) []{$r_{j}$};
  \node (A5) at (10,-1.5) [circle, draw, inner sep=0pt, minimum width=4pt]{};
  \node (A51) at (10,-1.95) []{$r_{k}$};

  \draw[-] (A2) to (A3);
  \draw[double] (A3) to (A4);
  \draw[-] (A3) to (A5);
  \draw[double] (A4) to (A5);
  \draw[double, bend left=20] (A2) to (A4);
  \draw[-, bend right=20] (A2) to (A5);

\end{tikzpicture}
\end{figure}

\begin{itemize}
\item $(t_i t_j)^4=1$ $\rightarrow$ $(\varphi(t_i) \varphi(t_j))^4= (r_{i+1}r_j)^4=1$. Analogously we can argue for $(t_{i} t_{k})^4=1$.
\item $(t_{j}t_k)^4=1$ $\rightarrow$ $(\varphi(t_{j}) \varphi(t_k))^4= (r_j r_k)^4=1$.
\item $(t_{i}t_kt_jt_k)^2$ $\rightarrow$ $(\varphi(t_{i})\varphi(t_k)\varphi(t_j) \varphi(t_k))^2 = (r_{i+1}r_kr_jr_k)^2 = 1$, where $(r_{i+1} r_kr_jr_k)^2=1$ by one of the cycles in $\sigma_i(\Gamma)$.
\item $(t_{i+1}t_j)^2=1$ $\rightarrow$ $(\varphi(t_{i+1}) \varphi(t_j))^2 = (r_{i+1}r_i r_{i+1}r_j)^2$, but the equation $(r_{i+1}r_i r_{i+1}r_j)^2=1$ is equivalent to $(r_ir_{i+1}r_j r_{i+1})^2=1$, which holds by one of the cycles in $\sigma_i(\Gamma)$. Analogously we can argue for $(t_{i+1}t_k)^2=1$.
\end{itemize}

5) The vertex $t_{i}$ is part of a full subgraph which is a $3$-cycle, while $t_{i+1}$ is not a vertex of that $3$-cycle but $t_{i+1}$ is the distinguished vertex.

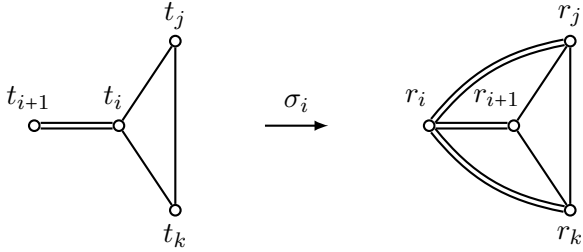
\begin{figure}[H]
\centering
\begin{tikzpicture}[scale=0.75, thick,>=latex, double equal sign distance]
   \tikzset{triple/.style={double,postaction={draw,-}}}

  \node (a2) at (0.5,0) [circle, draw, inner sep=0pt, minimum width=4pt]{};
  \node (a21) at (0.4,0.5) []{$t_{i+1}$};
  \node (a3) at (2,0) [circle, draw, inner sep=0pt, minimum width=4pt]{};
  \node (a31) at (1.9,0.5) []{$t_{i}$};
  \node (a4) at (3,1.5) [circle, draw, inner sep=0pt, minimum width=4pt]{};
  \node (a41) at (3,1.95) []{$t_{j}$};
  \node (a5) at (3,-1.5) [circle, draw, inner sep=0pt, minimum width=4pt]{};
  \node (a51) at (3,-1.95) []{$t_{k}$};

  \draw[double] (a2) to (a3);
  \draw[-] (a3) to (a4);
  \draw[-] (a3) to (a5);
  \draw[-] (a4) to (a5);

  \node (A) at (4.4,0) []{};
  \node (AB) at (5.15,0.4) []{$\sigma_i$};
  \node (B) at (5.9,0) []{};

  \draw[->] (A) to (B);

  \node (A2) at (7.5,0) [circle, draw, inner sep=0pt, minimum width=4pt]{};
  \node (A21) at (7.27,0.5) []{$r_{i}$};
  \node (A3) at (9,0) [circle, draw, inner sep=0pt, minimum width=4pt]{};
  \node (A31) at (8.7,0.5) []{$r_{i+1}$};
  \node (A4) at (10,1.5) [circle, draw, inner sep=0pt, minimum width=4pt]{};
  \node (A41) at (10,1.95) []{$r_{j}$};
  \node (A5) at (10,-1.5) [circle, draw, inner sep=0pt, minimum width=4pt]{};
  \node (A51) at (10,-1.95) []{$r_{k}$};

  \draw[double] (A2) to (A3);
  \draw[-] (A3) to (A4);
  \draw[-] (A3) to (A5);
  \draw[-] (A4) to (A5);
  \draw[double, bend left=20] (A2) to (A4);
  \draw[double, bend right=20] (A2) to (A5);

\end{tikzpicture}
\end{figure}

\begin{itemize}
\item $(t_i t_j)^3=1$ $\rightarrow$ $(\varphi(t_i) \varphi(t_j))^3= (r_{i+1}r_j)^3=1$. Analogously we can argue for $(t_{i} t_{k})^3=1$.
\item $(t_{j}t_k)^3=1$ $\rightarrow$ $(\varphi(t_{j}) \varphi(t_k))^3= (r_j r_k)^3=1$.
\item $(t_{i}t_kt_jt_k)^2$ $\rightarrow$ $(\varphi(t_{i})\varphi(t_k)\varphi(t_j) \varphi(t_k))^2 = (r_{i+1}r_kr_jr_k)^2 = 1$, where $(r_{i+1} r_kr_jr_k)^2=1$ by one of the cycles in $\sigma_i(\Gamma)$.
\item $(t_{i+1}t_j)^2=1$ $\rightarrow$ $(\varphi(t_{i+1}) \varphi(t_j))^2 = (r_{i+1}r_i r_{i+1}r_j)^2$, but the equation $(r_{i+1}r_i r_{i+1}r_j)^2=1$ is equivalent to $(r_ir_{i+1}r_j r_{i+1})^2=1$, which holds by one of the cycles in $\sigma_i(\Gamma)$. Analogously we can argue for $(t_{i+1}t_k)^2=1$.
\end{itemize}
\end{proof}


\medskip
\noindent We have $\sigma_i^{-1}(t_1, \ldots , t_m) = (r_1, \ldots , r_m)$, where $t_j = r_j$ for $j \in [m] \setminus \{i, i+1 \}$, $r_i = t_{i+1}$ and $r_{i+1}= t_{i+1} t_i t_{i+1}$. Similar to Remarks \ref{rem:RelationsSimplyLaced} and \ref{rem:RelationsNonSimplyLaced} we can describe how to obtain the Carter diagram $\sigma_i^{-1}(\Gamma)$ on vertex set $\{r_1, \ldots , r_m \}$ from $\Gamma$. The same arguments as in the proof of Lemma \ref{lem:SurjHomB1} lead to the following result.

\begin{Lemma} \label{lem:SurjHomB2}
Let $\Gamma$ be a Carter diagram of simply-laced type or type $B_n$ with vertex set corresponding to $t_1, \ldots , t_m$ and $i \in [m-1]$. If $(t_i t_{i+1})^2 \neq 1$ (and, in the $B_n$-case, if $t_{i+1}$ is not the distinguished vertex of $\Gamma$), then the map 
$$
\psi: W(\Gamma) \rightarrow W(\sigma_i^{-1}(\Gamma)), ~ t_j \mapsto \begin{cases} 
r_j &\mbox{if } j \neq i, i+1 \\
r_ir_{i+1}r_i &\mbox{if } j=i\\
r_{i} &\mbox{if } j=i+1
\end{cases}
$$
extends to a (surjective) group homomorphism.
\end{Lemma}

The maps $\varphi$ from Lemma \ref{lem:SurjHomB1} and $\psi$ from Lemma \ref{lem:SurjHomB2} are mutually inverse group homomorphisms. Therefore we obtain the following.
\begin{Proposition} \label{prop:HurwitzPresInv}
The groups $W(\Gamma)$ and $W(\sigma_i(\Gamma))$ (resp. $W(\sigma_i^{-1}(\Gamma))$) are isomorphic.
\end{Proposition}

\medskip
$\phantom{4}$

\begin{proof}[\textbf{Proof of Theorem \ref{thm:Main2}}]
Let us begin with the case that the root system $\Phi$ is simply-laced or of type $B_n$. Let $\Gamma'$ be a Carter diagram of the same Dynkin type as $\Phi$. By Lemma \ref{lem:CarterIrredQuasiCox}, the Carter diagram $\Gamma'$ can be realized by a reduced reflection factorization $(r_1, \dots , r_m) \in \Red_T(w)$ for some quasi-Coxeter element $w \in W_{\Phi}$. By \cite[Theorem 1.1 and Remark 8.3]{BGRW} we find $(t_1, \ldots , t_m)$ and $\sigma \in \mathcal{B}_m$ such that $\sigma(t_1, \ldots , t_m)= (r_1, \ldots , r_m)$ and the Carter diagram $\Gamma$ corresponding to $(t_1, \ldots , t_m)$ is admissible. In particular we have $\sigma(\Gamma)= \Gamma'$. By Proposition \ref{prop:Main2Admissible} we have $W(\Gamma) \cong W_{\Phi}$ and repeated use of Proposition \ref{prop:HurwitzPresInv} yields $W(\Gamma) \cong W(\Gamma')$.
\end{proof}

\begin{remark}
We used the Hurwitz action to prove Theorem \ref{thm:Main2}, but we did not define a Hurwitz action on Carter diagrams. In fact, it might happen that there are elements $x,y \in W$ and reduced reflection factorizations $\underline{t} \in \Red_T(x), ~\underline{r} \in \Red_T(y)$ such that $\underline{t}$ and $\underline{r}$ give rise to the same Carter diagram, but the set of Carter diagrams associated to the Hurwitz orbit of $\underline{t}$ and the set of Carter diagrams associated to the Hurwitz orbit of $\underline{r}$ are {\bf not} equal (see the next example). The reason for this is that the Hurwitz orbit of a reflection factorizaton $\underline{t}=(t_1, \ldots, t_m)$ depends on the actual tuple, while the associated Carter diagram only depends on the set $\{ t_1, \ldots , t_m\}$. 
\end{remark}

\begin{example}
Consider a Coxeter system $(W, \{s_1, s_2, s_3, s_4\})$ of type $D_4$, where $s_2$ does not commute with any of the other simple reflections. Let $T$ be its set of reflections and $c = s_1 s_2s_3s_4$ be a Coxeter element. Inside the Hurwitz orbit of the corresponding factorization we find
$$
(s_1, s_2, s_3, s_4 ) \sim (s_2, s_2s_1s_2, s_3, s_4).
$$
The Carter diagram associated to $(s_2, s_2s_1s_2, s_3, s_4)$ is shown on the right side of the following picture, while the Carter diagram of $(s_1, s_2, s_3, s_4 )$ is the corresponding Dynkin diagram.
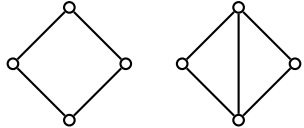
\begin{figure}[H]
\centering
\begin{tikzpicture}[scale=0.75, thick,>=latex]

  \node (21) at (6,0) [circle, draw, inner sep=0pt, minimum width=4pt]{};
  \node (31) at (6,-2) [circle, draw, inner sep=0pt, minimum width=4pt]{};
  \node (41) at (5,-1) [circle, draw, inner sep=0pt, minimum width=4pt]{};
  \node (51) at (7,-1) [circle, draw, inner sep=0pt, minimum width=4pt]{};

  \draw[-] (21) to (41);
  \draw[-] (31) to (41);
  
  \draw[-] (21) to (51);
  \draw[-] (31) to (51);
  \draw[-] (21) to (31);
  
  \node (a21) at (3,0) [circle, draw, inner sep=0pt, minimum width=4pt]{};
  \node (a31) at (3,-2) [circle, draw, inner sep=0pt, minimum width=4pt]{};
  \node (a41) at (2,-1) [circle, draw, inner sep=0pt, minimum width=4pt]{};
  \node (a51) at (4,-1) [circle, draw, inner sep=0pt, minimum width=4pt]{};

  \draw[-] (a21) to (a41);
  \draw[-] (a31) to (a41);
  
  \draw[-] (a21) to (a51);
  \draw[-] (a31) to (a51);

\end{tikzpicture}
\end{figure}
The element $w:= (s_2s_3s_2)s_1(s_1s_2s_1)(s_1s_2s_4s_2s_1) \in W$ is a quasi-Coxeter element, but not a Coxeter element. Inside the Hurwitz orbit we find
$$
(s_2s_3s_2,s_1,s_1s_2s_1,s_1s_2s_4s_2s_1) \sim (s_2s_3s_2s_1s_2s_3s_2, s_2s_3s_2,s_1s_2s_1,s_1s_2s_4s_2s_1).
$$
The Carter diagram associated to the left factorization is the one on the right side of the above picture, while the Carter diagram associated to the right factorization is the one on the left side. These two are the only Carter diagrams which appear in the Hurwitz orbit of $(s_2s_3s_2,s_1,s_1s_2s_1,s_1s_2s_4s_2s_1)$. In particular, the sets $\Red_T(c)$ and $\Red_T(w)$ have both diagrams shown above as associated Carter diagrams, while the Dynkin diagram of type $D_4$ just appears as a Carter diagram associated to a reduced reflection factorization in $\Red_T(c)$. 
\end{example}


\subsection{The non-crystallographic reflection groups}

Let $(W,S)$ be an arbitrary Coxeter system of rank $n$ with set of reflections $T$. We generalize the definiton of a Carter diagram as follows. 

\begin{Definition}
Let $w \in W$ and $(t_1, \ldots , t_n) \in \Red_T(w)$. We define the \defn{Carter diagram} $\Gamma$ corresponding to this reduced reflection factorization to be the graph on $n$ vertices corresponding to $t_1 ,\ldots , t_n$. Two vertices corresponding to $t_i$ and $t_j$ ($i \neq j$) are joined by $w_{ij}:=m_{ij}-2$ edges, where $m_{ij}$ is the order of $t_it_j$.

Note that this definion is consistent with Definition \ref{def:CarterDiagram} if $W=W_{\Phi}$ is a Weyl group (by Remark \ref{rem:MinCarter}).

We call $\Gamma$ to be of type $I_2(m)$ (resp. $H_3$ resp. $H_4$) if $\langle t_1, \ldots , t_n \rangle$ is a Coxeter group of type $I_2(m)$ (resp. $H_3$ resp. $H_4$).
\end{Definition}

The only Carter diagram of type $I_2(m)$ is the corresponding Coxeter diagram. For the types $H_3$ and $H_4$ we list all Carter diagrams in Figures \ref{fig:CarterH3} and \ref{fig:CarterH4}.

\begin{figure}[H]
\centering
\begin{tikzpicture}[scale=0.65, thick,>=latex, double distance = 2.8pt]
  \tikzset{triple/.style={double,postaction={draw,-}}}

  \node (11) at (4,-1) [circle, draw, inner sep=0pt, minimum width=4pt]{};
  \node (21) at (1,-1) [circle, draw, inner sep=0pt, minimum width=4pt]{};
  \node (41) at (2.5,-1) [circle, draw, inner sep=0pt, minimum width=4pt]{};

  \draw[triple] (21) to (41);
  \draw[-] (41) to (11);

  \node (b1) at (5,0) [circle, draw, inner sep=0pt, minimum width=4pt]{};
  \node (b2) at (7,0) [circle, draw, inner sep=0pt, minimum width=4pt]{};
  \node (b4) at (5,-2) [circle, draw, inner sep=0pt, minimum width=4pt]{};
  
  \draw[triple] (b1) to (b2);
  \draw[-]      (b4) to (b1);
  \draw[triple] (b2) to (b4);

  \node (c1) at (8,0) [circle, draw, inner sep=0pt, minimum width=4pt]{};
  \node (c2) at (10,0) [circle, draw, inner sep=0pt, minimum width=4pt]{};
  \node (c4) at (8,-2) [circle, draw, inner sep=0pt, minimum width=4pt]{};
  
  \draw[triple] (c1) to (c2);
  \draw[triple] (c4) to (c1);
  \draw[triple] (c2) to (c4);

  \node (d1) at (11,0) [circle, draw, inner sep=0pt, minimum width=4pt]{};
  \node (d2) at (13,0) [circle, draw, inner sep=0pt, minimum width=4pt]{};
  \node (d4) at (11,-2) [circle, draw, inner sep=0pt, minimum width=4pt]{};
  
  \draw[-] (d1) to (d2);
  \draw[-] (d4) to (d1);
  \draw[triple] (d2) to (d4);

  \node (a11) at (17,-1) [circle, draw, inner sep=0pt, minimum width=4pt]{};
  \node (a21) at (14,-1) [circle, draw, inner sep=0pt, minimum width=4pt]{};
  \node (a41) at (15.5,-1) [circle, draw, inner sep=0pt, minimum width=4pt]{};

  \draw[triple] (a21) to (a41);
  \draw[triple] (a41) to (a11);

\end{tikzpicture}
\caption{Carter diagrams of type $H_3$} \label{fig:CarterH3}
\end{figure}
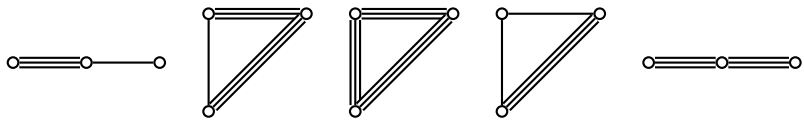

\newpage

\begin{figure}[H]
\centering
\begin{tikzpicture}[scale=0.55, thick,>=latex, double distance = 2.8pt]
  \tikzset{triple/.style={double,postaction={draw,-}}}

  \node (a1) at (0,0) [circle, draw, inner sep=0pt, minimum width=4pt]{};
  \node (a2) at (2,0) [circle, draw, inner sep=0pt, minimum width=4pt]{};
  \node (a3) at (4,0) [circle, draw, inner sep=0pt, minimum width=4pt]{};
  \node (a4) at (0,-2) [circle, draw, inner sep=0pt, minimum width=4pt]{};

  \draw[-]      (a1) to (a2);
  \draw[-]      (a2) to (a3);
  \draw[triple] (a1) to (a4);
  \draw[triple] (a2) to (a4);

  \node (b1) at (6,0) [circle, draw, inner sep=0pt, minimum width=4pt]{};
  \node (b2) at (8,0) [circle, draw, inner sep=0pt, minimum width=4pt]{};
  \node (b3) at (10,0) [circle, draw, inner sep=0pt, minimum width=4pt]{};
  \node (b4) at (6,-2) [circle, draw, inner sep=0pt, minimum width=4pt]{};

  \draw[triple] (b1) to (b2);
  \draw[-]      (b2) to (b3);
  \draw[triple] (b1) to (b4);
  \draw[triple] (b2) to (b4);

  \node (c1) at (12,0) [circle, draw, inner sep=0pt, minimum width=4pt]{};
  \node (c2) at (14,0) [circle, draw, inner sep=0pt, minimum width=4pt]{};
  \node (c3) at (16,0) [circle, draw, inner sep=0pt, minimum width=4pt]{};
  \node (c4) at (12,-2) [circle, draw, inner sep=0pt, minimum width=4pt]{};
  
  \draw[-]      (c1) to (c2);
  \draw[triple] (c2) to (c3);
  \draw[-]      (c1) to (c4);
  \draw[-]      (c4) to (c2);

  \node (d1) at (18,-1) [circle, draw, inner sep=0pt, minimum width=4pt]{};
  \node (d2) at (20,-1) [circle, draw, inner sep=0pt, minimum width=4pt]{};
  \node (d3) at (22,-1) [circle, draw, inner sep=0pt, minimum width=4pt]{};
  \node (d4) at (24,-1) [circle, draw, inner sep=0pt, minimum width=4pt]{};

  \draw[triple] (d1) to (d2);
  \draw[-]      (d2) to (d3);
  \draw[-]      (d3) to (d4);

  \node (e1) at (0,-3) [circle, draw, inner sep=0pt, minimum width=4pt]{};
  \node (e2) at (2,-3) [circle, draw, inner sep=0pt, minimum width=4pt]{};
  \node (e3) at (0,-5) [circle, draw, inner sep=0pt, minimum width=4pt]{};
  \node (e4) at (2,-5) [circle, draw, inner sep=0pt, minimum width=4pt]{};

  \draw[-]      (e1) to (e2);
  \draw[triple] (e2) to (e3);
  \draw[triple] (e1) to (e3);
  \draw[-]      (e3) to (e4);
  \draw[triple] (e2) to (e4);

  \node (f1) at (6,-3) [circle, draw, inner sep=0pt, minimum width=4pt]{};
  \node (f2) at (8,-3) [circle, draw, inner sep=0pt, minimum width=4pt]{};
  \node (f3) at (6,-5) [circle, draw, inner sep=0pt, minimum width=4pt]{};
  \node (f4) at (8,-5) [circle, draw, inner sep=0pt, minimum width=4pt]{};

  \draw[-]      (f1) to (f2);
  \draw[triple] (f1) to (f3);
  \draw[-]      (f1) to (f4);
  \draw[triple] (f2) to (f3);
  \draw[-]      (f2) to (f4);
  \draw[triple] (f3) to (f4);

  \node (g1) at (12,-3) [circle, draw, inner sep=0pt, minimum width=4pt]{};
  \node (g2) at (14,-3) [circle, draw, inner sep=0pt, minimum width=4pt]{};
  \node (g3) at (12,-5) [circle, draw, inner sep=0pt, minimum width=4pt]{};
  \node (g4) at (14,-5) [circle, draw, inner sep=0pt, minimum width=4pt]{};

  \draw[triple] (g1) to (g2);
  \draw[triple] (g1) to (g3);
  \draw[triple] (g1) to (g4);
  \draw[-]      (g2) to (g3);
  \draw[triple] (g2) to (g4);
  \draw[triple] (g3) to (g4);

  \node (h1) at (18,-3) [circle, draw, inner sep=0pt, minimum width=4pt]{};
  \node (h2) at (20,-3) [circle, draw, inner sep=0pt, minimum width=4pt]{};
  \node (h3) at (18,-5) [circle, draw, inner sep=0pt, minimum width=4pt]{};
  \node (h4) at (20,-5) [circle, draw, inner sep=0pt, minimum width=4pt]{};

  \draw[triple] (h1) to (h2);
  \draw[-]      (h2) to (h4);
  \draw[triple] (h4) to (h3);
  \draw[-]      (h3) to (h1);

    \node (i1) at (0,-6) [circle, draw, inner sep=0pt, minimum width=4pt]{};
  \node (i2) at (2,-6) [circle, draw, inner sep=0pt, minimum width=4pt]{};
  \node (i3) at (0,-8) [circle, draw, inner sep=0pt, minimum width=4pt]{};
  \node (i4) at (2,-8) [circle, draw, inner sep=0pt, minimum width=4pt]{};

  \draw[-] (i1) to (i2);
  \draw[triple] (i2) to (i3);
  \draw[-] (i1) to (i3);
  \draw[-] (i3) to (i4);
  \draw[triple] (i2) to (i4);

  \node (j1) at (6,-6) [circle, draw, inner sep=0pt, minimum width=4pt]{};
  \node (j2) at (8,-6) [circle, draw, inner sep=0pt, minimum width=4pt]{};
  \node (j3) at (6,-8) [circle, draw, inner sep=0pt, minimum width=4pt]{};
  \node (j4) at (8,-8) [circle, draw, inner sep=0pt, minimum width=4pt]{};

  \draw[-] (j1) to (j2);
  \draw[triple] (j1) to (j3);
  \draw[-] (j1) to (j4);
  \draw[-] (j2) to (j4);
  \draw[-] (j3) to (j4);

  \node (k1) at (12,-6) [circle, draw, inner sep=0pt, minimum width=4pt]{};
  \node (k2) at (14,-6) [circle, draw, inner sep=0pt, minimum width=4pt]{};
  \node (k3) at (12,-8) [circle, draw, inner sep=0pt, minimum width=4pt]{};
  \node (k4) at (14,-8) [circle, draw, inner sep=0pt, minimum width=4pt]{};

  \draw[triple] (k1) to (k2);
  \draw[triple] (k1) to (k3);
  \draw[-] (k1) to (k4);
  \draw[-] (k2) to (k3);
  \draw[triple] (k2) to (k4);
  \draw[triple] (k3) to (k4);

  \node (l1) at (18,-6) [circle, draw, inner sep=0pt, minimum width=4pt]{};
  \node (l2) at (20,-6) [circle, draw, inner sep=0pt, minimum width=4pt]{};
  \node (l3) at (18,-8) [circle, draw, inner sep=0pt, minimum width=4pt]{};
  \node (l4) at (20,-8) [circle, draw, inner sep=0pt, minimum width=4pt]{};

  \draw[triple] (l1) to (l2);
  \draw[-] (l2) to (l4);
  \draw[-] (l4) to (l3);
  \draw[-] (l3) to (l1);

  \node (m1) at (0,-9) [circle, draw, inner sep=0pt, minimum width=4pt]{};
  \node (m2) at (2,-9) [circle, draw, inner sep=0pt, minimum width=4pt]{};
  \node (m3) at (4,-9) [circle, draw, inner sep=0pt, minimum width=4pt]{};
  \node (m4) at (0,-11) [circle, draw, inner sep=0pt, minimum width=4pt]{};

  \draw[-] (m1) to (m2);
  \draw[triple] (m2) to (m3);
  \draw[triple] (m1) to (m4);
  \draw[triple] (m2) to (m4);

  \node (n1) at (6,-9) [circle, draw, inner sep=0pt, minimum width=4pt]{};
  \node (n2) at (8,-9) [circle, draw, inner sep=0pt, minimum width=4pt]{};
  \node (n3) at (10,-9) [circle, draw, inner sep=0pt, minimum width=4pt]{};
  \node (n4) at (6,-11) [circle, draw, inner sep=0pt, minimum width=4pt]{};
  
  \draw[triple] (n1) to (n2);
  \draw[triple] (n2) to (n3);
  \draw[-] (n1) to (n4);
  \draw[triple] (n2) to (n4);

  \node (o1) at (12,-9) [circle, draw, inner sep=0pt, minimum width=4pt]{};
  \node (o2) at (14,-9) [circle, draw, inner sep=0pt, minimum width=4pt]{};
  \node (o3) at (16,-9) [circle, draw, inner sep=0pt, minimum width=4pt]{};
  \node (o4) at (12,-11) [circle, draw, inner sep=0pt, minimum width=4pt]{};

  \draw[-] (o1) to (o2);
  \draw[triple] (o2) to (o3);
  \draw[triple] (o1) to (o4);
  \draw[-] (o4) to (o2);

  \node (p1) at (18,-10) [circle, draw, inner sep=0pt, minimum width=4pt]{};
  \node (p2) at (20,-10) [circle, draw, inner sep=0pt, minimum width=4pt]{};
  \node (p3) at (22,-10) [circle, draw, inner sep=0pt, minimum width=4pt]{};
  \node (p4) at (24,-10) [circle, draw, inner sep=0pt, minimum width=4pt]{};

  \draw[triple] (p1) to (p2);
  \draw[triple] (p2) to (p3);
  \draw[triple] (p3) to (p4);

  \node (aa1) at (0,-12) [circle, draw, inner sep=0pt, minimum width=4pt]{};
  \node (aa2) at (2,-12) [circle, draw, inner sep=0pt, minimum width=4pt]{};
  \node (aa3) at (4,-12) [circle, draw, inner sep=0pt, minimum width=4pt]{};
  \node (aa4) at (0,-14) [circle, draw, inner sep=0pt, minimum width=4pt]{};

  \draw[-]      (aa1) to (aa2);
  \draw[-]      (aa2) to (aa3);
  \draw[triple] (aa1) to (aa4);
  \draw[-]      (aa2) to (aa4);

  \node (bb1) at (6,-12) [circle, draw, inner sep=0pt, minimum width=4pt]{};
  \node (bb2) at (8,-12) [circle, draw, inner sep=0pt, minimum width=4pt]{};
  \node (bb3) at (10,-12) [circle, draw, inner sep=0pt, minimum width=4pt]{};
  \node (bb4) at (6,-14) [circle, draw, inner sep=0pt, minimum width=4pt]{};

  \draw[triple] (bb1) to (bb2);
  \draw[-]      (bb2) to (bb3);
  \draw[-]      (bb1) to (bb4);
  \draw[triple] (bb2) to (bb4);

  \node (cc1) at (12,-12) [circle, draw, inner sep=0pt, minimum width=4pt]{};
  \node (cc2) at (14,-12) [circle, draw, inner sep=0pt, minimum width=4pt]{};
  \node (cc3) at (16,-12) [circle, draw, inner sep=0pt, minimum width=4pt]{};
  \node (cc4) at (12,-14) [circle, draw, inner sep=0pt, minimum width=4pt]{};
  
  \draw[-]      (cc1) to (cc2);
  \draw[-]      (cc2) to (cc3);
  \draw[-]      (cc1) to (cc4);
  \draw[triple] (cc4) to (cc2);

  \node (dd1) at (18,-13) [circle, draw, inner sep=0pt, minimum width=4pt]{};
  \node (dd2) at (20,-13) [circle, draw, inner sep=0pt, minimum width=4pt]{};
  \node (dd3) at (22,-13) [circle, draw, inner sep=0pt, minimum width=4pt]{};
  \node (dd4) at (24,-13) [circle, draw, inner sep=0pt, minimum width=4pt]{};

  \draw[triple] (dd1) to (dd2);
  \draw[triple]      (dd2) to (dd3);
  \draw[-]      (dd3) to (dd4);

  \node (ee1) at (0,-15) [circle, draw, inner sep=0pt, minimum width=4pt]{};
  \node (ee2) at (2,-15) [circle, draw, inner sep=0pt, minimum width=4pt]{};
  \node (ee3) at (0,-17) [circle, draw, inner sep=0pt, minimum width=4pt]{};
  \node (ee4) at (2,-17) [circle, draw, inner sep=0pt, minimum width=4pt]{};

  \draw[-]      (ee1) to (ee2);
  \draw[triple] (ee2) to (ee3);
  \draw[triple] (ee1) to (ee3);
  \draw[triple] (ee3) to (ee4);
  \draw[-]      (ee2) to (ee4);
  \draw[triple] (ee1) to (ee4);

  \node (ff1) at (6,-15) [circle, draw, inner sep=0pt, minimum width=4pt]{};
  \node (ff2) at (8,-15) [circle, draw, inner sep=0pt, minimum width=4pt]{};
  \node (ff3) at (6,-17) [circle, draw, inner sep=0pt, minimum width=4pt]{};
  \node (ff4) at (8,-17) [circle, draw, inner sep=0pt, minimum width=4pt]{};

  \draw[-]      (ff1) to (ff2);
  \draw[-]      (ff1) to (ff3);
  \draw[-]      (ff1) to (ff4);
  \draw[triple] (ff2) to (ff3);
  \draw[-]      (ff2) to (ff4);
  \draw[-]      (ff3) to (ff4);

  \node (gg1) at (12,-15) [circle, draw, inner sep=0pt, minimum width=4pt]{};
  \node (gg2) at (14,-15) [circle, draw, inner sep=0pt, minimum width=4pt]{};
  \node (gg3) at (12,-17) [circle, draw, inner sep=0pt, minimum width=4pt]{};
  \node (gg4) at (14,-17) [circle, draw, inner sep=0pt, minimum width=4pt]{};

  \draw[triple] (gg1) to (gg2);
  \draw[triple] (gg1) to (gg3);
  \draw[triple] (gg2) to (gg3);
  \draw[triple] (gg2) to (gg4);
  \draw[-]      (gg3) to (gg4);

  \node (hh1) at (18,-15) [circle, draw, inner sep=0pt, minimum width=4pt]{};
  \node (hh2) at (20,-15) [circle, draw, inner sep=0pt, minimum width=4pt]{};
  \node (hh3) at (18,-17) [circle, draw, inner sep=0pt, minimum width=4pt]{};
  \node (hh4) at (20,-17) [circle, draw, inner sep=0pt, minimum width=4pt]{};

  \draw[triple] (hh1) to (hh2);
  \draw[-]      (hh2) to (hh4);
  \draw[triple] (hh4) to (hh3);
  \draw[triple] (hh3) to (hh1);
  \draw[-]      (hh1) to (hh4);

  \node (eee1) at (0,-18) [circle, draw, inner sep=0pt, minimum width=4pt]{};
  \node (eee2) at (2,-18) [circle, draw, inner sep=0pt, minimum width=4pt]{};
  \node (eee3) at (0,-20) [circle, draw, inner sep=0pt, minimum width=4pt]{};
  \node (eee4) at (2,-20) [circle, draw, inner sep=0pt, minimum width=4pt]{};

  \draw[triple] (eee1) to (eee2);
  \draw[triple] (eee2) to (eee3);
  \draw[triple] (eee1) to (eee3);
  \draw[triple] (eee3) to (eee4);
  \draw[triple] (eee2) to (eee4);
  \draw[triple] (eee1) to (eee4);

  \node (fff1) at (6,-18) [circle, draw, inner sep=0pt, minimum width=4pt]{};
  \node (fff2) at (8,-18) [circle, draw, inner sep=0pt, minimum width=4pt]{};
  \node (fff3) at (6,-20) [circle, draw, inner sep=0pt, minimum width=4pt]{};
  \node (fff4) at (8,-20) [circle, draw, inner sep=0pt, minimum width=4pt]{};

  \draw[-]      (fff1) to (fff2);
  \draw[triple] (fff1) to (fff3);
  \draw[-]      (fff2) to (fff4);
  \draw[triple] (fff3) to (fff4);

  \node (ggg1) at (12,-18) [circle, draw, inner sep=0pt, minimum width=4pt]{};
  \node (ggg2) at (14,-18) [circle, draw, inner sep=0pt, minimum width=4pt]{};
  \node (ggg3) at (12,-20) [circle, draw, inner sep=0pt, minimum width=4pt]{};
  \node (ggg4) at (14,-20) [circle, draw, inner sep=0pt, minimum width=4pt]{};

  \draw[triple] (ggg1) to (ggg2);
  \draw[triple] (ggg1) to (ggg3);
  \draw[-]      (ggg2) to (ggg3);
  \draw[-]      (ggg2) to (ggg4);
  \draw[-]      (ggg3) to (ggg4);

  \node (hhh1) at (18,-18) [circle, draw, inner sep=0pt, minimum width=4pt]{};
  \node (hhh2) at (20,-18) [circle, draw, inner sep=0pt, minimum width=4pt]{};
  \node (hhh3) at (18,-20) [circle, draw, inner sep=0pt, minimum width=4pt]{};
  \node (hhh4) at (20,-20) [circle, draw, inner sep=0pt, minimum width=4pt]{};

  \draw[-]      (hhh1) to (hhh2);
  \draw[triple] (hhh2) to (hhh4);
  \draw[triple] (hhh4) to (hhh3);
  \draw[-]      (hhh3) to (hhh1);
  \draw[-]      (hhh1) to (hhh4);
  \draw[-]      (hhh2) to (hhh3);

  \node (aaa1) at (0,-21) [circle, draw, inner sep=0pt, minimum width=4pt]{};
  \node (aaa2) at (2,-21) [circle, draw, inner sep=0pt, minimum width=4pt]{};
  \node (aaa3) at (2,-23) [circle, draw, inner sep=0pt, minimum width=4pt]{};
  \node (aaa4) at (0,-23) [circle, draw, inner sep=0pt, minimum width=4pt]{};

  \draw[triple]      (aaa1) to (aaa2);
  \draw[-]      (aaa2) to (aaa3);
  \draw[triple] (aaa1) to (aaa4);
  \draw[triple]      (aaa1) to (aaa3);
  \draw[-]      (aaa4) to (aaa3);

  \node (bbb1) at (6,-21) [circle, draw, inner sep=0pt, minimum width=4pt]{};
  \node (bbb2) at (8,-21) [circle, draw, inner sep=0pt, minimum width=4pt]{};
  \node (bbb3) at (10,-21) [circle, draw, inner sep=0pt, minimum width=4pt]{};
  \node (bbb4) at (6,-23) [circle, draw, inner sep=0pt, minimum width=4pt]{};

  \draw[triple] (bbb1) to (bbb2);
  \draw[triple]      (bbb2) to (bbb3);
  \draw[triple]      (bbb1) to (bbb4);
  \draw[triple] (bbb2) to (bbb4);

  \node (ccc1) at (12,-21) [circle, draw, inner sep=0pt, minimum width=4pt]{};
  \node (ccc2) at (14,-21) [circle, draw, inner sep=0pt, minimum width=4pt]{};
  \node (ccc3) at (14,-23) [circle, draw, inner sep=0pt, minimum width=4pt]{};
  \node (ccc4) at (12,-23) [circle, draw, inner sep=0pt, minimum width=4pt]{};
  
  \draw[triple]      (ccc1) to (ccc2);
  \draw[-]      (ccc2) to (ccc3);
  \draw[triple]      (ccc1) to (ccc4);
  \draw[-] (ccc1) to (ccc3);
  \draw[-]      (ccc3) to (ccc4);

  \node (ddd1) at (18,-22) [circle, draw, inner sep=0pt, minimum width=4pt]{};
  \node (ddd2) at (20,-22) [circle, draw, inner sep=0pt, minimum width=4pt]{};
  \node (ddd3) at (22,-22) [circle, draw, inner sep=0pt, minimum width=4pt]{};
  \node (ddd4) at (24,-22) [circle, draw, inner sep=0pt, minimum width=4pt]{};

  \draw[triple] (ddd1) to (ddd2);
  \draw[-] (ddd2) to (ddd3);
  \draw[triple]      (ddd3) to (ddd4);

  \node (eeee1) at (0,-24) [circle, draw, inner sep=0pt, minimum width=4pt]{};
  \node (eeee2) at (2,-24) [circle, draw, inner sep=0pt, minimum width=4pt]{};
  \node (eeee3) at (0,-26) [circle, draw, inner sep=0pt, minimum width=4pt]{};
  \node (eeee4) at (2,-26) [circle, draw, inner sep=0pt, minimum width=4pt]{};

  \draw[-] (eeee1) to (eeee2);
  \draw[triple] (eeee2) to (eeee3);
  \draw[-] (eeee1) to (eeee3);
  \draw[triple] (eeee3) to (eeee4);
  \draw[-] (eeee2) to (eeee4);
  \draw[triple] (eeee1) to (eeee4);

  \node (ffff1) at (6,-24) [circle, draw, inner sep=0pt, minimum width=4pt]{};
  \node (ffff2) at (8,-24) [circle, draw, inner sep=0pt, minimum width=4pt]{};
  \node (ffff3) at (6,-26) [circle, draw, inner sep=0pt, minimum width=4pt]{};
  \node (ffff4) at (8,-26) [circle, draw, inner sep=0pt, minimum width=4pt]{};

  \draw[-]      (ffff1) to (ffff2);
  \draw[triple] (ffff1) to (ffff3);
  \draw[triple]      (ffff2) to (ffff4);
  \draw[-] (ffff3) to (ffff4);
  \draw[-]      (ffff1) to (ffff4);
  \draw[-]      (ffff3) to (ffff2);

  \node (gggg1) at (12,-24) [circle, draw, inner sep=0pt, minimum width=4pt]{};
  \node (gggg2) at (14,-24) [circle, draw, inner sep=0pt, minimum width=4pt]{};
  \node (gggg3) at (12,-26) [circle, draw, inner sep=0pt, minimum width=4pt]{};
  \node (gggg4) at (14,-26) [circle, draw, inner sep=0pt, minimum width=4pt]{};

  \draw[triple] (gggg1) to (gggg2);
  \draw[triple] (gggg1) to (gggg3);
  \draw[-]      (gggg1) to (gggg4);
  \draw[triple] (gggg2) to (gggg4);
  \draw[triple] (gggg3) to (gggg4);

  \node (hhhh1) at (18,-24) [circle, draw, inner sep=0pt, minimum width=4pt]{};
  \node (hhhh2) at (20,-24) [circle, draw, inner sep=0pt, minimum width=4pt]{};
  \node (hhhh3) at (18,-26) [circle, draw, inner sep=0pt, minimum width=4pt]{};
  \node (hhhh4) at (20,-26) [circle, draw, inner sep=0pt, minimum width=4pt]{};

  \draw[triple]      (hhhh1) to (hhhh2);
  \draw[triple] (hhhh2) to (hhhh4);
  \draw[triple] (hhhh4) to (hhhh3);
  \draw[triple]      (hhhh3) to (hhhh1);

%
\end{tikzpicture}
\caption{Carter diagrams of type $H_4$} \label{fig:CarterH4}
\end{figure}
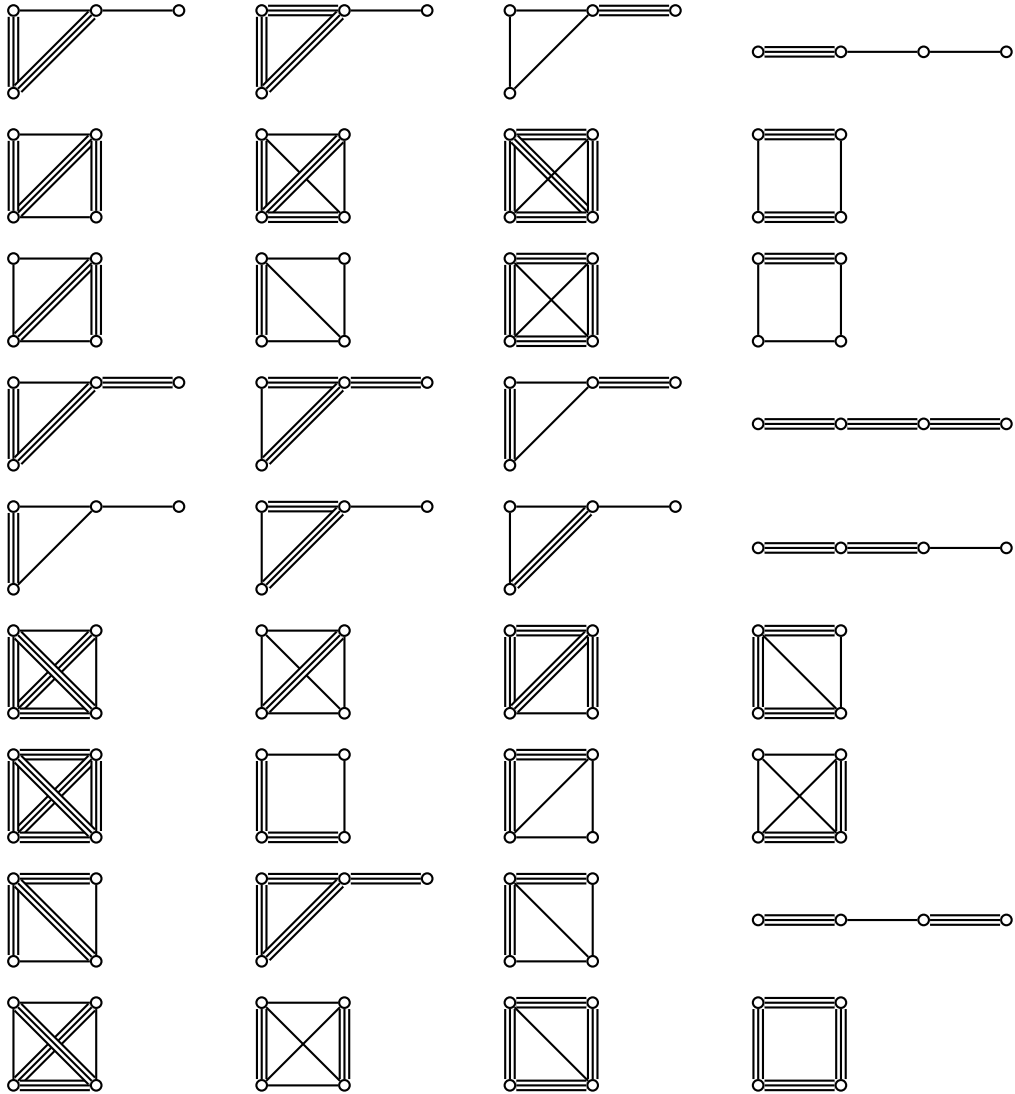

For these class of Carter diagrams, in general we do not obtain a presentation as for the crystallographic types. First we have to restrict ourselves to those Carter diagrams which are induced by reduced reflection factorizations of Coxeter elements. Those diagrams are given by the three leftmost diagrams in Figure \ref{fig:CarterH3} and by the diagrams in the first two lines in Figure \ref{fig:CarterH4}.
Then we have to replace the relation (R3) of Section \ref{sec:intro} by the following two relations
\begin{enumerate}
\item[(R3')] for any chordless cycle
$$
i_0 ~\stackrel{w_{i_0 i_1}}{\edgel} ~ i_1 ~ \stackrel{w_{i_1 i_2}}{\edgel}~ \ldots ~ \stackrel{w_{i_{d-2} i_{d-1}}}{\edgel} ~i_{d-1} ~ \stackrel{w_{i_{d-1} i_0}}{\edgel} ~ i_0,
$$
where either all weights are $1$ or $w_{i_{d-1} i_0} =3$, but not all of the other weights are $3$, we have 
$$
(t_{i_0} t_{i_1} \cdots t_{i_{d-2}} t_{i_{d-1}} t_{i_{d-2}} \cdots t_{i_1})^2 =1.
$$
\item[(R3'')] for any chordless $3$-cycle

\begin{figure}[H]
\centering
\begin{tikzpicture}[scale=0.8, thick,>=latex, double distance = 2.8pt]
  \tikzset{triple/.style={double,postaction={draw,-}}}

  \node (a1) at (0,0) []{$i_0$};
  \node (a2) at (2,0) []{$i_1$};
  \node (a3) at (4,0) []{$i_2$};
  \node (a4) at (6,0) []{$i_0$};

  \draw[triple] (a1) to (a2);
  \draw[triple] (a2) to (a3);
  \draw[triple] (a3) to (a4);
\end{tikzpicture}
\end{figure}
\noindent we have the relations 
\begin{align*}
(t_{i_0} t_{i_1} t_{i_2} t_{i_1})^3 & =1\\
(t_{i_1} t_{i_0} t_{i_1} t_{i_2} t_{i_1} t_{i_2})^2 & =1.
\end{align*}
\end{enumerate}

\medskip
Using GAP \cite{GAP2017} we arrive at the following statement. 

\begin{Proposition}
Let $(W,S)$ be a Coxeter system of type $H_3, H_4$ or $I_2(m)$, $\Gamma$ a Carter diagram of the same type with vertex set given by the reduced reflection factorization of a Coxeter element in $W$. Then the group $W(\Gamma)$ with generators $t_i$, $i$ a vertex of $\Gamma$, subject to the relations (R1), (R2), (R3') and (R3'') is isomorphic to $W$.
\end{Proposition}

\begin{remark}
The statement of the preceding proposition is wrong if we replace the Coxeter element by a quasi-Coxeter element. That is, the diagrams in Figures \ref{fig:CarterH3} and \ref{fig:CarterH4} which are not induced by reduced reflection factorizations of a Coxeter element, do not yield a presentation of $W$ given by the relations (R1), (R2), (R3') and (R3'').
\end{remark}

\nocite{*}
\bibliographystyle{amsplain}
\bibliography{mybibDiagrammatics}

\end{document}